\documentclass[letterpaper, 11pt]{amsart}
\usepackage{defs}

\def\vp{\varphi}

\def\vs{\vspace{10pt}}

\def\eps{{\varepsilon}}

\def\Prob{{\mathbb{P}}}

\def\Var{{\rm Var}}

\def\EXP{{\mathbb{E}}}

\def\complex{\mathbb{C}}

\def\bbT{\mathbb{T}}

\def\naturals{\mathbb{N}}

\def\reals{\mathbb{R}}

\def\integers{\mathbb{Z}}

\def\bP{\mathbf{P}}

\def\cA{\mathcal{A}}

\def\cB{\mathcal{B}}

\def\cC{\mathcal{C}}

\def\cD{\mathcal{D}}

\def\cF{\mathcal{F}}

\def\cG{\mathcal{G}}

\def\cE{\mathcal{E}}

\def\cL{\mathcal{L}}

\def\cM{\mathcal{M}}

\def\cO{\mathcal{O}}

\def\cV{\mathcal{V}}

\def\cX{\mathcal{X}}

\def\fN{\mathfrak{N}}

\def\fb{\mathfrak{b}}

\def\fn{\mathfrak{n}}

\def\hsigma{{\hat\sigma}}

\makeatother

\def\beq{\begin{equation}}
\def\eeq{\end{equation}}

\pgfplotsset{compat=1.16}

\begin{document}
\title[The Bootstrap for Dynamical Systems]{The Bootstrap for Dynamical Systems}
\author{Kasun Fernando}
\address{Kasun Fernando\\
  Department of Mathematics\\
  University of Toronto\\
  40 St George St. Toronto, ON, Canada M5S 2E4}
\email{{\tt kasun.akurugodage@utoronto.ca}}
\author{Nan Zou}
\address{Nan Zou\\ 
Department of Mathematics and Statistics\\
Macquarie University \\
Room 706, 12 Wally's Walk, Macquarie Park, NSW, Australia 2113
}
\email{{\tt nan.zou@mq.edu.au}}

\begin{abstract}
    Despite their deterministic nature, dynamical systems often exhibit seemingly random behaviour. Consequently, a dynamical system is usually represented by a probabilistic model
    of which the unknown parameters must be estimated
    using statistical methods. When measuring the uncertainty of such parameter estimation, the bootstrap stands out as a simple but powerful technique. In this paper, we develop the bootstrap for dynamical systems and establish not only its consistency but also its second-order efficiency via a novel \textit{continuous} Edgeworth expansion for dynamical systems. This is the first time such continuous Edgeworth expansions have been studied. Moreover, we verify the theoretical results about the bootstrap using computer simulations.
\end{abstract}
\keywords{dynamical systems, the bootstrap, continuous Edgeworth expansions, expanding maps, Markov chains}
\subjclass[2010]{62F40, 37A50, 62M05}
\maketitle
\tableofcontents

\section{Introduction}

 After its establishment in late 19th century through the efforts of \textcite{poincare1879proprietes} and \textcite{lyapunov1892general}, the theory of dynamical systems was applied to study the qualitative behaviour of dynamical processes in the real  world. For example, the theory of dynamical systems has proved useful in, among many others, astronomy \cite{contopoulos2004order}, chemical engineering \cite{barkai2003aging}, biology \cite{mcgoff2016local}, ecology \cite{turchin2013complex}, demography \cite{zincenko2021turing}, economics \cite{varian1981dynamical}, language processing \cite{port1995mind}, neural activity \cite{shenoy2013cortical}, and machine learning \cite{brunton2019data, duriez2017machine}. In particular, deep neural networks can be considered as a special class of discrete-time dynamical systems \cite{weinan2017proposal}. 
    
While a continuous-time dynamical system is often expressed as a system of differential equations, a discrete-time dynamical system can be written as a process $\{X_{i}, i=0,1,2,\dots\}$ iteratively generated by a 
deterministic transformation function $g$ which, in practice, is often unknown:
    \begin{equation}\label{eq:DataGenerating1}
     X_{i} = g(X_{i-1}), i = 1, 2, \dots.
    \end{equation}
    Since the transformation $g$ is deterministic, conditional on a fixed $X_{0}$, the process $\{X_{i}\}$ 
    will also be deterministic. However, in the long run, this deterministic process 
    can still exhibit incomprehensibly complex behaviors and possess seemingly random patterns 
    \cite{Gora, LasotaMackey}. For example, 
    has \textit{chaotic} sample paths $\{X_{i}\}$; see Figure \ref{fig:logistic}. For another example, note that given the initial velocity and acceleration of a coin toss, the orbit of the coin should be fully determined by the laws of physics, and hence, should be deterministic \cite{keller1986probability}; however, the landing of the coin may still appear to be random. In light of this, instead of trying to figure out the exact value of $X_{i}$ with a given initial state $X_{0}$, one 
    makes the compromise of analyzing the probabilistic features of the process $\{X_{i}\}$ 
    assuming that the initial state $X_{0}$ is randomly generated from an unknown initial distribution $\mu$, i.e.,
    \begin{equation}\label{eq:DataGenerating2}
    X_{0} \sim \mu.\vspace{-10pt}
    \end{equation}
    
    \begin{figure}[h!]
    \centering
    \includegraphics[scale=0.5]{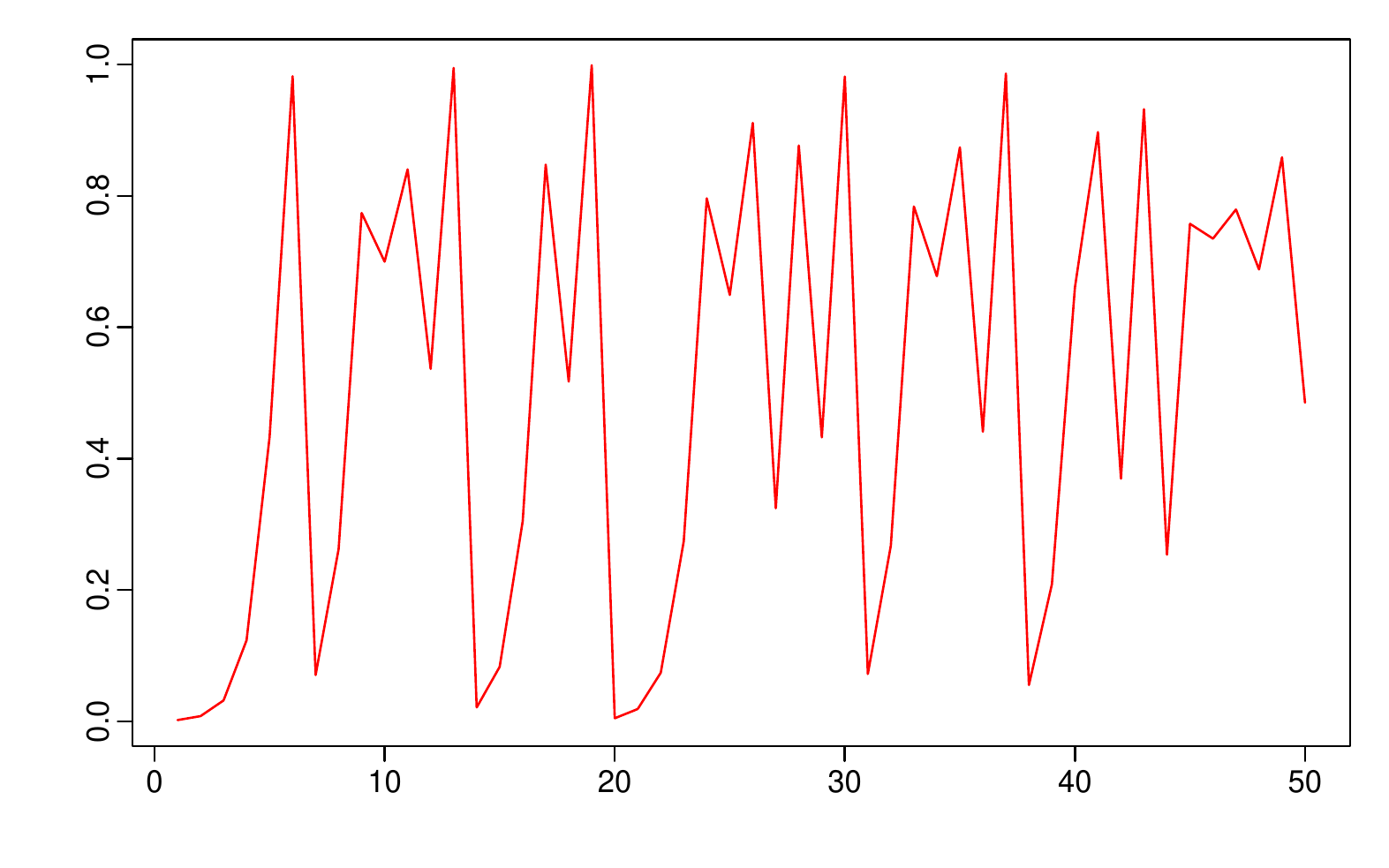}
    \vspace{-15pt}
    \caption{A sample path of the dynamical system with $g(x)=4x(1-x)$.}
    \label{fig:logistic}
    \end{figure}
    
 Usually, parameters of the model \eqref{eq:DataGenerating1} and \eqref{eq:DataGenerating2}  
    must be estimated from the available data. 
    This is carried out in a variety of fields and has wide ranging applications; for earlier surveys, see \cite{berliner1992statistics, chatterjee1992chaos, Isham1993, Jensen1993}; for a recent review with numerous references, see \cite{KM}. There is also 
    an increasing trend to study the estimation and prediction
    in dynamical systems theoretically, and several recent works in this vein include \cite{steinwart2009consistency,mcgoff2015consistency, hang2017bernstein, navarrete2018prediction, mcgoff2020empirical, mcgoff2021empirical}.     
    They present the consistency and/or the rate of convergence in various point estimation or prediction settings. However, as far as we understand, 
    the limiting distributions of these 
    estimators or predictors have not been studied yet, although some, e.g., \cite{mcgoff2015consistency, mcgoff2021empirical},
    indicated determining limiting distributions as a future direction.
    
    When approximating the limiting distribution of the estimators or predictors, the bootstrap \cite{efron1979} stands out as a 
    simple but powerful data-driven technique. In fact, the bootstrap is deceptively simple to state. Initially, by mimicking the generating process of the original dataset, the bootstrap creates a number of pseudo-datasets, each of which has the same size as the original dataset. Afterwards, the bootstrap uses the variation among these pseudo-datasets to approximate the randomness of the original dataset and the distribution of the estimators or predictors calculated from the pseudo-datasets to approximate the distribution of the original estimator or predictor; see \Cref{iidBoot} for an illustration of the classical bootstrap for the mean with iid sampling. 
    
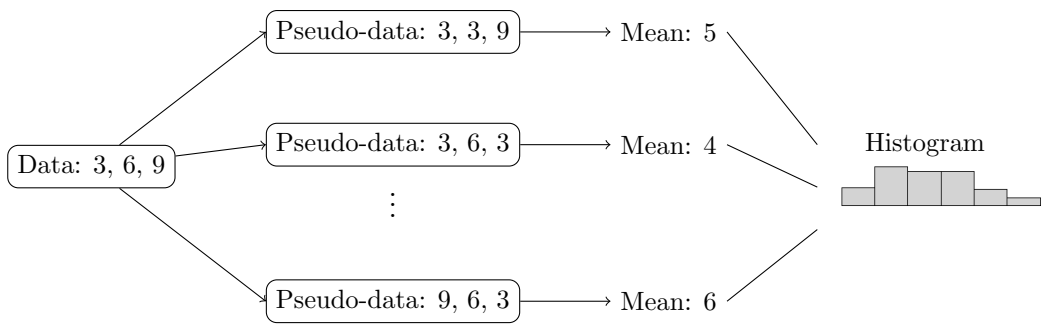
\begin{figure}[ht]
\centering
\begin{tikzpicture}[node distance = 15mm and 12mm,  block/.style = {draw, rectangle, rounded corners, align=center}]
\node [block] (init) {\small  Data: 3, 6, 9};
\node [block, above right = of init.east] (sample1) {\small  Pseudo-data: 3, 3, 9};
\node [block, below = of sample1.north] (sample2) {\small  Pseudo-data: 3, 6, 3};
\node [block, below = of sample2] (sampleB)
{\small  Pseudo-data: 9, 6, 3};
\node [below = of sample1] () {\vdots};
\node [right = of sample1] (stat1) {\small  Mean: 5};
\node [right = of sample2] (stat2) {\small  Mean: 4};
\node [right = of sampleB] (statB) {\small  Mean: 6};
\node [below right = of stat1.east] (EDF) {{\includegraphics[width=3cm]{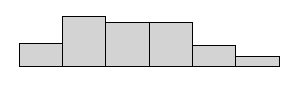}}};
\node [right = of stat2.east, xshift = 5mm] (caption){\small  Histogram};

\draw [->] (init) edge node [sloped, anchor=center, above] {} (sample1.west);
\draw [->] (init) edge (sample2.west);
\draw [->] (init) edge (sampleB.west);
\draw [->] (sample1) edge (stat1.west);
\draw [->] (sample2) edge (stat2.west);
\draw [->] (sampleB) edge (statB.west);
\draw [-] (stat1.east) edge (EDF.north west);
\draw [-] (stat2.east) edge (EDF.west);
\draw [-] (statB.east) edge (EDF.south west);
\end{tikzpicture}
\caption{Approximating the distribution of the sample mean by the classical bootstrap.}
\label{iidBoot}
\end{figure}

    Admittedly, the application of the bootstrap to dynamical systems is not brand new: in \cite{haraki2007bootstrap}, local-bootstrap predictors
    are averaged out to improve a derivative-based prediction of $X_{i}$; in \cite{lodhi2011bootstrapping}, averaged block-bootstrap estimators 
    are used with the intention of 
    refining an adaptive least square estimator; in \cite{frohlich2014uncertainty}, the empirical distribution obtained from a parametric bootstrap approach is used to
    estimate some non-identifiable parameters. However, as far as we know, there is 
    neither a theoretical study of a bootstrap method for dynamical systems nor any bootstrap methods developed for the generic application in the dynamical setting. 
    
    The absence of the related literature may be due to the fact that developing a bootstrap method for the generic dynamical system setting is not a straightforward task.
    For example, since the transformation function $g$ is deterministic, dynamical systems generated by \eqref{eq:DataGenerating1} and \eqref{eq:DataGenerating2} is in general not 
    $\alpha$-mixing as in \cite{rosenblatt1956central}; 
    see \cite[p.~709]{hang2017bernstein}. Hence, it is not completely clear if the block bootstrap in \cite{kunsch1989jackknife, liu1992moving} could be applied to dynamical systems. In this paper, instead of analyzing the block bootstrap, we design a novel bootstrap method specifically for dynamical systems (see \Cref{sec:BootstrapAlgo}) in which
    (1) generate pseudo initial state $X_{0}^{*}$ from $\mu^{*}$, an initial distribution that 
    may depend on $\{X_{i}\}$, (2) obtain $\wh g$, an estimate of the transformation function $g$, and finally, (3) generate pseudo data $\{X_{i}^{*}\}$ by 
    \begin{equation}\label{dynamicalbootstrap}
    X_{i}^{*} = \wh g(X_{i-1}^{*}),\,\, i = 1, 2, \dots.
    \end{equation}
    
    The consistency of the dynamical system bootstrap in \eqref{dynamicalbootstrap} is not something we immediately expect, let alone its second-order efficiency. 
    Indeed, the bootstrap in \eqref{dynamicalbootstrap} tries to mimic the original dynamical system by mimicking the transformation function $g$ and initial distribution $\mu$, but a small perturbation in $g$ and $\mu$ may cause 
    major changes in the dynamical system.
    First, the orbit of a nonlinear dynamical systems is generally very sensitive to small changes in its initial state $X_{0}$.
    Hence, it is not clear whether the dynamical system with an initial state generated from  the bootstrap initial distribution $\mu^{*}$ is a good approximation of the dynamical system with the true initial distribution $\mu$. Second, statistical properties like ergodicity may not be preserved under small changes in the transformation function $g$. As a result, it is not clear if the dynamical system 
    based on the estimated transformation, $\wh g$, is a good 
    approximation of the actual one. For a further discussion 
    about robustness of dynamical systems, or lack thereof, we refer the reader to \cite{Devaney, ott_2002, PughShub, takens} and references therein. 
    
    In this paper, under verifiable assumptions on $\mu^{*}$ and $\wh g$, we establish the consistency and the second-order efficiency of the dynamical system bootstrap in \eqref{dynamicalbootstrap} for the statistic
    \begin{equation}\label{eq:estimator}
    \frac{1}{n}\sum_{i=0}^{n-1}h(X_{i}),
    \end{equation}
    where $h$ is a known deterministic function which is commonly referred to as an \textit{observable}. In particular, when the variance of the asymptotic distribution of \eqref{eq:estimator}, $\sigma^2$, is known or can be easily estimated, we develop a pivoted bootstrap and prove its second-order efficiency when approximating the limiting distribution of \eqref{eq:estimator}; see \Cref{sec:AsympAccuPivot}. Moreover, when $\sigma$ is unknown and cannot be readily estimated, we develop a non-pivoted bootstrap that attains the first-order efficiency; see \Cref{sec:AsympAccuNonpivot}. As far as we know, there is no known consistent estimator for $\sigma$ in the dynamical systems setting. Therefore, as of now, the non-pivoted bootstrap is the only valid way to approximate the limiting distribution. 
    
    The second-order efficiency of the bootstrap is established using the \textit{continuous} first-order Edgeworth expansion for dynamical systems.
    It is an Edgeworth expansion that holds uniformly 
    with respect to the transformation function $g$, initial distribution $\mu$,
    and the observable $h$. This is the first time such expansions are established for dynamical systems, and our results are a significant extension of the 
    Edgeworth expansion results in the dynamical system literature, e.g., \cite{FernandoPene, jirak2021sharp}.
    
    The dynamical system in $\eqref{eq:DataGenerating1}$ may, as in Remark 2.4 of \cite{mcgoff2015consistency}, be viewed as a \textit{degenerate} Markov chain with one step transition probability  
    \begin{equation}\label{eq:Markov}
        p(x,A)={\bf 1}{\{g(x) \in A\}}
    \end{equation}
    whose Markov operator is, indeed, the Koopman operator of the dynamical system in $\eqref{eq:DataGenerating1}$. Nevertheless, each one of the Markovian settings in 
    \cite{rajarshi1990bootstrap, DM, HJ, PP, MM, paparoditis2001markovian, franke2002bootstrap, bertail2006regenerative}  
    is significantly different from the dynamical system setting, and hence, their results do not apply to our setting. 
    
    First, we recall, from \cite[p.~1]{franke2002bootstrap}, \cite[pp.~254-255]{rajarshi1990bootstrap}, \cite[Assumption 2 (iv)]{HJ}, \cite[Assumption A1 (i) ]{paparoditis2001markovian}, \cite[Assumption A2 (i)]{PP}, and \cite[Assumption A2]{MM}, the assumption of the Markov transition probability 
    $p(x,A)$ being absolutely continuous. However, since the transformation function $g$ in our setting \eqref{eq:DataGenerating1} is deterministic, by \eqref{eq:Markov}, this absolute continuity assumption in the Markov process literature is violated. In a word, it is exactly the determinism in dynamical systems that differentiate them from the setting in these Markov process literature.
    
    Second, on \cite[p.~86]{DM}, the existence of a Markov transition-probability estimator, $p_n(x,A)$, such that  
    \begin{equation}\label{eq:DattaAssump}
       \lim_{n \to \infty}\sup_{x,A} |p_n(x,A)-p(x,A)| = 0,\,\, \text{a.s.}
    \end{equation}
    is assumed. Due to \eqref{eq:Markov}, \eqref{eq:DattaAssump} is equivalent to the existence of a transformation-function estimator $g_{n}$ such that, almost surely, $g_n=g$ when $n$ is large enough. However, this assumption is much stronger than the uniform convergence or even the $C^1-$convergence and is unrealistic for applications. Moreover, \eqref{eq:DattaAssump} implies strong convergence of the Markov operators of which the dynamical equivalent, strong convergence of transfer operators, is too restrictive, see \cite{Liverani} for details. 
    
    Finally, it is assumed on \cite[p.~692]{bertail2006regenerative} that there exists an accessible atom $B$, which is a positive-measure set satisfying
    $$p(x, \cdot) = p(y, \cdot), \ \text{for all} \ x, y \in B.$$ 
    In the dynamical system setting \eqref{eq:DataGenerating1}, this assumption is equivalent to $g$ being a constant function over a set with positive measure and is too restrictive in practice. 
    
    We also remark that, unlike some of the above Markovian settings, we do not assume the initial measure $\mu$ to be a invariant measure; hence, we do not need $\{X_{i}\}$ to be stationary. 

    This article is divided into two parts. In Part~\ref{part:EdgeExp}, we establish the continuous first-order Edgeworth expansions for dynamical systems whose twisted transfer operators have a spectral gap. This is implemented using the Nagaev-Guivarc'h perturbation method \cite{Nagaev,Guivarch,HennionHerve} via the Keller-Liverani approach \cite{KellerLiverani} combined with recent developments from \cite{FL, FernandoPene}. The required spectral assumption and their implications are discussed in \Cref{sec:SpectralAssump}. In \Cref{sec:Expansions}, we use these assumptions to prove the key theorem in this paper, \Cref{thm:CtsEdgeExp}, which establishes the first order continuous Edgeworth expansion. Moreover, in \Cref{sec:Exmp}, we illustrate our general results by considering several classes of examples: smooth expanding maps of the circle, piece-wise uniformly expanding maps of an interval, and Markov models including $V-$geometrically ergodic Markov chains. In Part~\ref{part:Bootstrap}, we discuss the bootstrap algorithm in detail and establish its consistency and second-order efficiency 
    using the Edgeworth expansions. This is done in \Cref{sec:BootstrapAlgo} and \Cref{sec:Accu}. Finally, we discuss the simulation results for the doubling map, 
    the drill map and the logistic map in \Cref{sec:simulationresults}.
    
\section{Notation and Preliminaries}\label{sec:Notation}

Let $X$ be a metric space with a reference Borel probability measure $m$, 
$J \subset \reals $ be a neighbourhood of $0$, and $g_\theta:X\to X,\, \theta \in J$ 
be a family of dynamical systems, as in \eqref{eq:DataGenerating1}. We assume these systems are non-singular, i.e., for all $\theta$, for all $U \subseteq X$ Borel subsets such that $m(U)=0$, we have that $m({g}^{-1}_\theta U)=0$.   
Denote by $\cM_1(X)$ the set of Borel probability measures on $X$. Let $\nu \in \cM_1(X)$.
For $p\geq 1$, by $L^p(\nu)$, we denote the standard Lebesgue spaces with respect to $\nu$, i.e.,
$$L^p(\nu)\colonequals \{ h: X \to X\,|\, h\, \text{is Borel measurable},\, \nu(|h|^p)<\infty \}$$
where the notation $\nu(h)$ refers to the integral of a function $h$ with respect to a measure $\nu$ and the corresponding norm is denoted by $\|\cdot\|_{L^p(\nu)}$. When $\nu=m$, we often write, $L^p$ instead of $L^p(m)$. 

For us, an observable is a function $h \in L^{3}$, as in \eqref{eq:estimator}. Given a family of observables $\{h_\theta\}_{\theta \in J}$, we consider the family of Birkhoff sums (also commonly referred to as ergodic sums),
\begin{equation}\label{eq:BirkhoffSum}
    S_{\theta,n}(h_\theta) = \sum_{k=0}^{n-1} h_\theta \circ 
    {g}^{k}_\theta.
\end{equation}
We denote by $H_\theta$ the operator corresponding to multiplication by $h_\theta$, i.e., 
\begin{equation}\label{MultiOper}
H_\theta(\psi) = h_\theta \psi.
\end{equation}
Recall that $\mu$ is the initial distribution as in \eqref{eq:DataGenerating2}. Write 
\begin{align}\label{eq:Mean}
    &A_{\theta} = \lim_{n \to \infty} \EXP_\mu\left(\frac{S_{\theta,n}(h_\theta)}{n}\right),\\ \label{eq:Variance}
    &\sigma^2_{\theta} =\lim_{n\to\infty} \EXP_\mu\left(\frac{S_{\theta,n}(h_\theta)-nA_{\theta}}{\sqrt{n}}\right)^2,\,\,\text{and} \\ \label{eq:ThirdMoment}
    &M_{\mu,\theta}=\EXP_\mu\left(\frac{S_{\theta,n}(h_\theta)-nA_{\theta}}{n^{1/3}}\right)^3
\end{align}
for the asymptotic mean, the asymptotic variance and the asymptotic third moment of Birkhoff sums,  $S_{\theta,n}(h_\theta)$, respectively. We will see later in \Cref{ContVar} that the first two are independent of the choice of $\mu$. 

We say $\cL_\theta:L^1\to L^1$ is the transfer operator of $g_\theta$ with respect to $m$, if for all $\vp\in L^1$ and $\psi \in L^\infty$,
\begin{equation}\label{eq:Duality}
    m(\cL_\theta(\vp)\cdot \psi) = m(\vp\cdot \psi \circ g_\theta )
\end{equation}
Let $\mu \in \cM_1(X)$ be absolutely continuous with respect to $m$ with density $\rho_\mu$.
Then, from \eqref{eq:Duality}, it follows that
\begin{equation}\label{eq:CharFunc}
    \EXP_\mu(e^{isS_{\theta, n}(h_\theta)})=m\left( \cL^n_{\theta, is}(\rho_\mu)\right)
\end{equation}
where 
\begin{equation}\label{eq:cL_theta_is}
    \cL_{\theta, is}(\cdot)=\cL_{\theta}(e^{is h_\theta} \cdot ),\ s \in \reals \ (\text{with} \ \cL_\theta \colonequals \cL_{\theta,0})\, ;
\end{equation}
see \cite[Chapter XI]{HennionHerve}.

Let the standard Gaussian density and the corresponding distribution function be denoted by, respectively,
\begin{equation}\label{eq:standardGauss}
\fn(x)=\frac{1}{\sqrt{2\pi}}e^{-x^2/2}\,\,\, \text{and}\,\,\, \fN(x)=\int_{-\infty}^x \fn(y)\, dy\, .
\end{equation}
We recall the definition of Edgeworth expansions for dynamical systems which were studied extensively in \cite{FL, FernandoPene}. 
\begin{defin}[Edgeworth expansions for the dynamical system $g_0$]
The family of Birkhoff sums $\{S_{0, n}(h)\}_{n}$ satisfies the order $r$ continuous Edgeworth expansion if there exist $A_{0},\sigma_{0}$, and for $k=1,\dots,r$ polynomials $P_{k}$ such that 
  $$ \sup_{x\in \reals} \left|\Prob_{\mu}\left(\frac{S_{0,n}(h)-nA_{0}}{\sigma_0\sqrt{n}} \leq x\right) - \fN(x) -\sum_{k=1}^r \frac{P_k(x)}{n^{k/2}}\fn(x)\right| = o(n^{-r/2})$$
as $n \to \infty$.
\end{defin}
\noindent
Note that, in the above definition, the dynamical system is fixed and the polynomials $P_k$ depend on the choice of the dynamical system. 

In the setting we work on, we can consider the continuous Edgeworth expansions for the family of dynamical systems $\{g_\theta\}$. 
\begin{defin}[Continuous Edgeworth expansions for the family $\{(g_\theta, h_\theta)\}_{\theta\in J}$]
The family of Birkhoff sums $\{S_{\theta,n}(h_\theta)\}_{\theta,n}$ satisfies the order $r$ continuous Edgeworth expansion if there exist $A_\theta, \sigma_\theta$, and for $k=1,\dots,r$ polynomials $P_k$ (independent of $\theta$) such that 
  $$ \sup_{x\in \reals} \left|\Prob_{\mu}\left(\frac{S_{\theta,n}(h_\theta)-nA_\theta}{\sigma_\theta\sqrt{n}} \leq x\right) - \fN(x) -\sum_{k=1}^r \frac{P_k(x)}{n^{k/2}}\fn(x)\right| = o(n^{-r/2})$$
as $n \to \infty$ and $\theta \to 0$.
\end{defin}
\noindent
Note that the polynomials $P_k$ above, unlike the previous case, do not depend on $\theta$. It is also worth mentioning that these polynomials can be explicitly computed, and that their coefficients depend only on the asymptotic moments of $S_{0,n}(h_0)$.  

To illustrate the importance of continuous Edgeworth expansions, suppose the statistic of interest $T$ is asymptotically normally distributed and its distribution $G$ and the family of bootstrap estimates, say $\{\wh{G}\}$, which depend on the sample size $n$, admit a continuous first-order Edgeworth expansion. Then 
$$G(x)=\fN(x)+\frac{1}{\sqrt{n}}P_1(x)\fn(x)+o(n^{-1/2}),\,\,\, \wh G(x)=\fN(x)+\frac{1}{\sqrt{n}}P_1(x)\fn(x)+o_{\text{a.s.}}(n^{-1/2}).$$
Therefore, we have $G-\wh G = o_{\text{a.s.}}(n^{-1/2})$. This is the core of the proof of asymptotic accuracy of our bootstrap algorithms. With extra information, we may improve the error of approximation to $\cO_{\text{a.s.}}(n^{-1})$ 
which is significantly better than that of the Gaussian approximation. 

In what follows,  $\cL(\cB_1,\cB_2)$ denotes the space of bounded linear operators from a Banach space $(\cB_1,\|\cdot\|_{\cB_1})$ to a Banach space $(\cB_2, \|\cdot\|_{\cB_2})$, 
and $\cB_1^{\,\prime} = \cL(\cB_1,\complex)$, i.e., the space of continuous linear functionals on $\cB_1$. Recall that $\cL(\cB_1,\cB_2)$ has a standard topology generated by the operator norm which we denote by $\|\cdot\|_{\cB_1,\cB_2}$. Given two operators $\cL_1$ and $\cL_2$, we write $\cL_1 \cL_2$ to denote their composition $\cL_1 \circ \cL_2$ when their composition is well-defined, and when the $n-$fold composition of $\cL_1$ with itself is well-defined, we denote it by $\cL^n_1$. In addition, $\cB_1 \hookrightarrow \cB_2$ denotes continuous embedding of spaces, i.e., $\cB_1 \subset \cB_2$ and there exists $ c >0$ such that $\|\cdot\|_{\cB_2} \leq  c \|\cdot\|_{\cB_1}$. Whenever we mention regularity of an operator-valued function from $\reals^n$ to $\cL(\cB_1,\cB_2)$, the regularity is with respect to the standard topologies on the two spaces. Moreover, we say an operator $\cL$ has a \textit{spectral gap} of $(1-\kappa)$ on $\cB_1$ if $1$ is a dominating simple eigenvalue of $\cL:\cB_1\to \cB_1$ and rest of its spectrum is strictly inside a disk in $\complex$ centred at the origin and having radius $\kappa$.

Throughout the article, we will use $\delta$ and $\kappa$ to denote \textit{small} positive constants with $\kappa<1$, while $c$ and $C$ are used for possibly \textit{large} positive constants. Occasionally, we simply use $\lesssim$ to denote that the estimates hold up to a constant. 
However, the uniformity of $\delta$ and $\kappa$ with respect to the parameter $\theta$ is crucial. In fact, coming up with verifiable assumptions that lead to this uniformity is one of the main contributions of this article. When several values of $\delta$'s arise from different assumptions, we will always pick the minimum of these $\delta$'s without explicitly stating that we do. When different $\kappa$'s are present, we will always pick their maximum. Finally, note that the values of constants can change from one line to the other. Even though it is an interesting problem to figure out optimal constants, we choose to focus entirely on the asymptotic accuracy (for which the value of the constants are irrelevant), and in turn, keep the exposition simpler.  

\part{Continuous Edgeworth Expansions} \label{part:EdgeExp}

Our goal in this part of the paper is to establish the first-order continuous Edgeworth expansions for the Birkhoff sums, $S_{\theta,n}(h_\theta)$. 
To this end, we employ the spectral approach. It has been used to establish limit theorems for dynamical systems and Markov chains with much success (see \cite{HennionHerve, FL, FernandoPene, HervePene, Gora} and references therein for a discussion). The operator that takes the centre stage in this approach is the transfer operators, $\cL_\theta$, defined by the duality relation \Cref{eq:Duality}. If the reader is familiar with Markov operators in Markov chains, then they may see some correspondence. However, the transfer operator is, in general, not a Markov operator.

The coding of the characteristic function using the transfer operator in \eqref{eq:CharFunc} allows us to translate the spectral properties of $\cL_{\theta, is}$ to properties of $S_{\theta,n}(h_\theta)$.
Standard facts about transfer operators which we state without proof are discussed in detail in \cite{HennionHerve, HervePene, FernandoPene, Gora}. We will refer the reader to those references whenever we omit proofs.

Now, we state a key example to keep in mind throughout the discussion. 

\begin{exmp}[$C^1$ perturbations of $C^2$ expanding maps of $\bbT$]
Let $X=\bbT$ be the one dimensional torus (the interval $[0,1]$ with $0$ and $1$ identified) along with the standard Lebesgue measure as the reference measure $m$. Let $g \in C^2(\bbT, \bbT)$ uniformly expanding map with $\|g'\|_{L^\infty} > 2$.  Let $g_\theta,\,\, \theta \in [0,1]$ with $g_0\colonequals g$ be
such that $d_{C^1}(g_\theta,g)=\|g_\theta-g\|_{L^\infty}+\|g'_\theta-g'\|_{L^\infty}\leq \theta$. We recall that the transfer operator $\cL_\theta$ takes the following form: 
$$\cL_\theta (\vp) (x)= \sum_{y \in g^{-1}_\theta\{ x\}} \frac{\vp(y)}{|g'_\theta(y)|},\,\,\, \forall x \in \bbT,\, \forall \vp \in L^1.$$
From the Nagaev-Guivarc'h spectral approach in \cite{Nagaev,Guivarch}, it follows that for all $\theta$, $\cL_\theta$ as an operator acting on $C^1(\bbT,\reals)$ has simple maximal eigenvalue $1$, and that $g_\theta$ has a unique absolutely continuous invariant probability measure (acip) $\nu_\theta$, i.e., $\nu_\theta(g_\theta^{-1}(U))=\nu_\theta(U)$ for all Borel measurable $U \subseteq X$ and $\nu_\theta$ has a density with respect to $m$. Also, $\nu_\theta$ is exponentially mixing, i.e., there exists $\kappa_\theta \in (0,1)$ such that for all $\vp, \psi \in C^1(\bbT,\reals)$,
$$\left|\nu_\theta(\vp \cdot \psi\circ g^n_\theta) - \nu_\theta(\vp)\nu_\theta(\psi)\right| \leq C_{\vp,\psi, \theta} \kappa_\theta^n.$$
Later, we shall see that $\kappa_\theta$ and $C_{\vp,\psi,\theta}$ can be chosen to be independent of $\theta$ (at least for $\theta$ close to $0$). In addition, under some non-degeneracy assumption on the observable $h_\theta$, the Birkhoff sums $S_{n,\theta}(h_\theta)$ satisfy the central limit theorem (CLT), the large deviation principle (LDP) and the first-order Edgeworth expansion; see \cite{FL, FH}. 

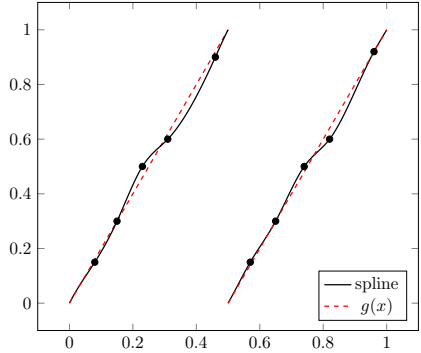
\begin{figure}[h!]
    \centering
    \begin{tikzpicture}[scale=0.6]
	\pgfplotsset{scale only axis,
		}
		\begin{axis}[
			samples=100,
			legend pos=south east,
			]
	\addplot[thick][domain=0:0.08]{+54.15459505378155*x^3+-9.391031776992637*x^2+2.279693133815209*x^1+0*x^0};
	\addplot[red, thick, dashed][domain=0:0.5] {2*x};
    \addplot [only marks] table {
                            
                            0.08 0.15
                            0.15 0.3
                            0.23 0.5
                            0.31 0.6
                            0.46 0.9

                            0.57 0.15
                            0.65 0.3
                            0.74 0.5
                            0.82 0.6
                            0.96 0.92
                            
                            };
	\addplot[thick][domain=0.08:0.15]{+15.007325676347904*x^3+0.004312873591436501*x^2+1.528065561768483*x^1+0.020043401921246026*x^0};
	\addplot[thick][domain=0.15:0.23]{+-91.0879158180058*x^3+47.7471715460506*x^2+-5.633363239100392*x^1+0.3781148419646898*x^0};
	\addplot[thick][domain=0.23:0.31]{+91.18684339132918*x^3+-78.02241230839053*x^2+23.29364104742107*x^1+-1.839622153335289*x^0};
	\addplot[thick][domain=0.31:0.46]{+-10.049563234265875*x^3+16.12744585341287*x^2+-5.892814982737985*x^1+1.1763116364478134*x^0};
	\addplot[thick][domain=0.46:0.5]{+-97.08400305932165*x^3+136.23497281198985*x^2+-61.14227738368339*x^1+9.647895871259442*x^0};		
	\addplot[thick][domain=0.5:0.57]{+-53.88693934684116*x^3+86.08930962336797*x^2+-43.63532500979787*x^1+7.031202517412088*x^0};
	\addplot[thick][domain=0.57:0.65]{+58.859430203577794*x^3+-106.70698230784845*x^2+66.25856139099548*x^1+-13.84863589873865*x^0};
	\addplot[thick][domain=0.65:0.74]{+-79.97426476872631*x^3+164.01872288814457*x^2+-109.71314698639998*x^1+24.278567583030366*x^0};
	\addplot[thick][domain=0.74:0.82]{+106.10794779434832*x^3+-249.08378900188112*x^2+195.98271181221904*x^1+-51.12641092062899*x^0};
	\addplot[thick][domain=0.82:0.96]{+-46.55013338210965*x^3+126.45509069220549*x^2+-111.95916953693198*x^1+33.044369981472286*x^0};
	\addplot[thick][domain=0.96:1]{+107.2349504281378*x^3+-316.4459506813071*x^2+313.22583018164016*x^1+-103.01482992847079*x^0};
	\addplot[red, thick, dashed][domain=0.5:1] {2*x-1};
	\addlegendentry{spline}
	\addlegendentry{$g(x)$}
	\end{axis}
	\end{tikzpicture}
	\vspace{-10pt}
    \caption{The doubling map and a periodic cubic spline approximation.}
    \label{fig:doubling}
\end{figure}

More concretely, consider the doubling map, $g:\bbT\to \bbT$, given by
$$g(x)=2x\hspace{-6pt} \mod 1$$
and its  
periodic cubic spline approximations. We recall that this is one of the simplest examples of $1-$dimensional smooth maps that gives rise to chaos. As we shall see later, period cubic spline approximation can be considered as a $C^1-$perturbation of $g$. In \Cref{fig:doubling}, we depict a periodic cubic spline approximation of the doubling map sketched using some noisy data. In simulations, since we do not have the \textit{a priori} knowledge of whether the data is generated by a map which is continuous on $\bbT$, we consider $\bbT$ to be $[0,1]$ with $0$ and $1$ identified, and use natural spline approximations. 
\end{exmp}

\section{Spectral Assumptions}\label{sec:SpectralAssump}
In this section, we state assumptions similar to those in \cite[Section 1.2]{FernandoPene}  (which are based on ideas from \citep{KellerLiverani, HervePene}), in order to establish continuous Edgeworth expansions for (possibly unbounded) observables. 
\vs

\begin{asm}\label{asm:alpha}
There exist $\delta>0$, 
$p_0 \geq 1$ 
and Banach spaces $\cB$ and $\wt \cB$ such that
\begin{equation}\label{TwoSpaceChain}
    \cB \hookrightarrow\wt \cB \hookrightarrow L^{p_0}(m)
\end{equation}
each containing ${\bf 1}_X$, and satisfying
\begin{enumerate}[leftmargin=*]
\item For all $\theta \in (-\delta,\delta) 
$ and $s\in \reals$, $\cL_{\theta, is} \in \cL(\cB,\cB) \cap \cL(\wt \cB,\wt \cB)$ and $(\theta, s) \mapsto \cL_{\theta,is}\in\mathcal L(\cB,\wt\cB)$ is continuous,
\item Either $\cB = \wt \cB$, or
there exist $c,C>0$ and 
$\kappa\in (0,1)$ such that
\begin{equation}\label{eq:UniformDF0}
    \sup_{
    |\theta|, |s| \in [0,\delta]}\Vert \cL_{\theta,is}^n \psi\Vert_{\cB}\le c\cdot  \kappa^n\Vert  \psi\Vert_{\cB}+ C^n\Vert  \psi\Vert_{\wt \cB}\,
\end{equation}
for all $\psi\in\cB$ and $n$.
\end{enumerate}
\end{asm}

\begin{asm}\label{asm:beta}
There exist  
$\delta>0$ and
a sequence of Banach spaces
\begin{equation}\label{SpaceChain}
     {\cB} = \cX_0\hookrightarrow\cX^{(+)}_0\hookrightarrow\cX_1\hookrightarrow\cX^{(+)}_1\hookrightarrow \cX_{2}\hookrightarrow \cX^{(+)}_{2}\hookrightarrow\mathcal X_{3}\hookrightarrow \cX_{3}^{(+)}\hookrightarrow {\cX_4=\wt \cB} 
\end{equation}
each containing ${\bf 1}_X$, and satisfying
\begin{enumerate}[leftmargin=*]
\item For all $a=0,1,2,3$, for all $\theta \in (-\delta,\delta)$ and $s \in \reals$, $\cL_{\theta, is} \in \cL(\cX_a, \cX_a) \cap \cL(\cX_a^{(+)},\cX_a^{(+)})$,
\item For all $\theta \in
(-\delta,\delta)$,  
for all $a= 0, 1,2$ and $j=1,\dots,3-a$, 
the map $s \mapsto \cL_{\theta, is}\in\mathcal L(\cX_a^{(+)},\cX_{a+j})$ is $C^{j}$ on 
$(-\delta,\delta)$
with the $j$-th derivative: 
$$(\cL^n_{\theta,is})^{(j)}{(\cdot)}\colonequals \cL^n_{\theta,is}((i S_{\theta,n}(h_\theta))^j\,\cdot\, )\in\mathcal L(\cX^{(+)}_a,\cX_{a+j}),$$
\item  Either 
\begin{itemize}
    \item $\cX_0 = \cX^{(+)}_3,$ or
    \item $(\theta,s) \mapsto \cL_{\theta,is}\in\mathcal L(\cX_a,\cX_{a}^{(+)})$ is continuous on  $(-\delta,\delta)\times \reals$ and there exist $c,C>0$ and $\kappa\in (0,1)$ such that
\begin{equation}\label{eq:UniformDF}
    \sup_{
    |\theta|, |s| \in [0,\delta]}\Vert \cL_{\theta,is}^n \psi\Vert_{\cC}\le c\cdot  \kappa^n\Vert  \psi\Vert_{\cC}+ C^n\Vert  \psi\Vert_{\wt \cB}\,
\end{equation}
for all $\psi\in\cC$ and $n$, and whenever $\cC = \cX_a$ or $\cX^{(+)}_a$ for $a=0,1,2,3$.
\end{itemize}
\end{enumerate}
\end{asm}

\begin{asm}\label{asm:gamma}
There exist $\kappa \in (0,1)$, $\delta >0$ such that
\begin{enumerate}[leftmargin=*]
    \item 
    $\cL_0$ has a spectral gap of 
    $(1-\kappa)$ on $\cB$,
    \item For all $\theta \in 
    (-\delta, \delta)$, $1$ is an eigenvalue of $\cL_\theta:\cB \to \cB$, 
    \item {$\cL_0$ has a spectral gap of 
    $(1-\kappa)$ in $\cX_a$ and $\cX_a^{(+)}$ for all $a=0,1,2,3$,}
    \item For all $\theta \in (-\delta, \delta)$ and for all $s \neq 0$, 
    The spectrum of the operators $\cL_{\theta,is}$ acting on either $\cX_a$ or $\cX_a^{(+)}$ for some $a=0,1,2,3$ is contained in $\{z\in \complex\ |\ |z|<1\}$.  
\end{enumerate}
\end{asm}

These assumptions are natural in the context of dynamical systems. Consider \Cref{asm:alpha}, and for a fixed $\theta$, \Cref{asm:beta} and assume that {$\cB=\wt \cB$}. Then the former allows us to apply classical perturbation theory in \cite{Kato} and the latter is the standard assumption to implement the Nagaev-Guivarc'h perturbation method and prove limit theorems for dynamical systems as in \cite{HennionHerve,Gora}. Note that, in this case, the regularity \Cref{asm:beta}(2) combined with \eqref{eq:CharFunc} imply that $s\mapsto \EXP_\mu(e^{isS_{n,\theta}(h_\theta)})$ is three times continuously differentiable. This means that $S_{n,\theta}(h_\theta)$ has three finite moments. Later, we shall see that this also implies the existence of the first three asymptotic moments, $A_\theta, \sigma^2_\theta$ and $M_{\mu,\theta}$. 

In the general case, \Cref{asm:beta}, 
in particular, the uniform Doeblin-Fortet inequality given by \eqref{eq:UniformDF} allows us to apply the Keller-Liverani perturbation result \cite{KellerLiverani} in our context. This is the approach used in \cite{HervePene} in a Markovian context, and more recently, in \cite{FernandoPene} for some dynamical systems. The advantage of the method is that it allows to establish various limit theorems for unbounded observables, and hence, 
under moment conditions very close to the optimal assumptions in the iid setting, and in particular, more general observables than the observables considered in \cite{DM, HJ}.  

\Cref{asm:gamma}(1) is equivalent (see \cite[Section 1]{HervePene}) to $$\lim_{n \to \infty}\|\cL^n_0-\Pi_0\|_{\cB,\cB} \leq C \kappa^n$$ where $\Pi_0$ is the rank one eigenprojection to the eigenspace of the eigenvalue $1$ having the form $\Pi_0 (\vp) = m(\vp)\rho_{\nu_0}$ where $\rho_{\nu_0}$ is the density of the unique exponentially mixing acip $\nu_0$. In Markovian settings, this is referred to as geometric ergodicity. In \Cref{asm:gamma}(2), we require that $1$ is an eigenvalue of $\cL_\theta$, which means that $g_\theta$ has an absolutely continuous invariant measure. As we shall see later, along with \Cref{asm:alpha}, this implies that $g_\theta$ (at least, for $\theta$ close to $0$) has a unique and exponentially mixing acip. \Cref{asm:gamma}(4) implies that $g$ is also an exact dynamical system \cite[Chapter 4.3]{LasotaMackey}. So, our exposition is limited to dynamical systems that exhibit strong pseudo-stochastic behaviour. 

\begin{rem}\label{WeakerAlpha}
Since, in proofs, we only need $\theta-$continuity of spectral data at $0$ and $s-$continuity everywhere (for $\theta$ close to $0$), we can replace \Cref{asm:alpha}(1) with
 \begin{equation}\label{eq:AltAssump}
     \lim_{(\theta,s)\to (0,0)}\|\cL_{\theta,is}-\cL_{0}\|_{\cB,\wt \cB} = 0\,\,\,\text{and}\,\,\lim_{s\to \bar{s}}\|\cL_{\theta,is}-\cL_{\theta,i\bar s}\|_{\cB,\wt \cB} = 0,\,\,\text{for }\theta\,\,\text{small, for all}\,\,\bar s \in \reals .\,
 \end{equation}
However, the $(\theta,s)-$continuity on $[-\delta,\delta]\times \reals$ is true in all the examples of dynamical systems we consider. 
\end{rem}

\begin{rem}\label{rem:Lebesgue}
In \cite{FernandoPene, HervePene}, assumptions equivalent to \eqref{TwoSpaceChain} and \eqref{SpaceChain} are stated for $L^{p}$ spaces with respect to invariant measures. Here, we use the reference measure because in applications the invariant measure is not known \textit{a priori}, and hence, the assumptions with respect to the invariant measure cannot be easily verified. We note that the results in \cite[Appendix A]{HervePene} are abstract (in the sense that they are results about operators in general) and do not depend on the choice of the measure. In fact, see comments on Condition ($\wt K$) in \cite[Section 4]{HervePene}. It is sufficient that $m \in \wt \cB\, {}^\prime$. This, we will assume throughout the discussion, and in all our examples, it is true because we let $\wt \cB=L^1(m)$. 
\end{rem}

In order to establish the $\theta-$continuity of $s-$derivatives of $\cL_{\theta,is}$ (whose existence we will show using previous assumptions), we need an extra assumption on $\{H_\theta\}$, and hence, on $\{h_\theta\}$.

\begin{asm}\label{asm:delta}
Let $\cC_1 \hookrightarrow \cC_2 \neq \wt \cB$ be two spaces appearing next to each other in \eqref{SpaceChain}. For all $\theta$,  for all $\cC_1,\cC_2$, $H_\theta(\cdot) \in \cL(\cC_1,\cC_2)$ and
\begin{equation}\label{eq:CtsMult}
    \lim_{\theta \to 0}\|H_\theta - H_0\|_{\cC_1,\cC_2}=0.
\end{equation}
\end{asm}
\vs

\begin{rem}\label{rem:Mult}
In the key application we discuss in \Cref{part:Bootstrap}, $h_\theta = h$ is independent of $\theta$. Therefore, \eqref{eq:CtsMult} is satisfied automatically. Also, in the special case of all $\cB=\cX_0=\cX^{(+)}_3$, 
we have that $H_\theta(\cdot) \in \cL(\cB,\cB)$ if $\cB$ is a Banach Algebra and $h_\theta \in \cB$. Indeed, this will be the case in most of our examples.
\end{rem}
Finally, we state an assumption that is required for the CLT to be non-degenerate. In particular, the following assumption guarantees that the asymptotic variance is non-zero.
\begin{asm}\label{asm:ve} \,
\begin{enumerate}[leftmargin=*]
    \item There does not exist $\ell\in L^2(m)
$ and a constant $c$ such that $h_0 = \ell \circ g_0 - \ell + c$,
    \item the sequence $$\left\{\sum_{k=0}^{n-1} h_0 \circ g^k_0\right\}_{n \in \naturals}$$ has an $L^2(m)-$weakly convergent subsequence.
\end{enumerate}
\end{asm}\vs

\begin{rem}\label{rem:cohomology} 
The \Cref{asm:ve}(1) is commonly written as
\begin{equation}\label{eq:CoBd} 
    \text{$h_0$ is not $g_0-$cohomologous to a constant in $L^2(m)$},
\end{equation}
Suppose this is true. Then,
$$\frac{1}{\sqrt{n}}\left(\sum_{k=0}^{n} h_0(g_0^k(x_0)) - nc\right) = \frac{h_0(f^{n}(x_0)) - h_0(x_0)}{\sqrt{n}}$$
So, the left hand side goes to $0$, and hence, $\sigma_0=0$. For the converse, see \citep[Lemma A.16]{DeSimoiLiverani}. In this case, the CLT is degenerate.

Along a periodic orbit $\{g^k_0(x_0)|k=0,\dots,n-1\}$ of length $n$, 
$$\sum_{k=0}^{n} h_0(g_0^k(x_0))= \sum_{k=0}^{n}  h(g^k(x_0))= \sum_{k=0}^{n}\left( \ell (g^{k+1}_0(x_0) - \ell(g^k_0(x_0)) + c \right) = nc$$
because $g^{n}(x_0)=x_0$. So, if there are two periodic orbits along which the ergodic averages are different, then we have \eqref{eq:CoBd}. Moreover, if $\cB \hookrightarrow L^2$ then the second condition is satisfied; see \Cref{rem:0avg}. 
\end{rem}

Under the spectral assumptions, we prove three Lemmas in the next section. We use ideas in  \cite{KellerLiverani, HervePene, FernandoPene} as well as standard functional analytic tools to prove these results. The significance here is the uniformity of certain conclusions in $\theta$. This is the crucial first step to obtain asymptotic expansions uniform in $\theta$. 

\subsection{Spectral decomposition of transfer operators}
The gist of the perturbation theorems is the following:  small regular perturbations of a bounded linear operator 
\begin{itemize}[leftmargin=25pt, topsep=0pt, itemsep=0pt]
    \item[--] result in perturbations of its spectrum that are as regular as the original perturbations, 
    \item[--] do not change the structure of the spectrum.
\end{itemize}
In our case, $\cL_0$ has a spectral gap, and hence, a leading simple eigenvalue $1$. $\cL_{\theta,is}$ is a continuous perturbation of $\cL_0$. Therefore, for small enough $\theta$ and $s$, $\cL_{\theta,is}$ has a leading simple eigenvalue which remains closer to $1$ (in fact, when $s=0$ it is equal to $1$), and depends continuously on the perturbation $(\theta,s)$. This is part of the conclusion of the lemma below.

\begin{lem} \label{WeakPert}
Suppose the Assumptions \ref{asm:alpha} and \ref{asm:gamma}(1,2)
hold and let $\bar \kappa \in (\kappa,1)$.  
Then there exist 
$\delta>0$  and a family $(\lambda_\theta(is), \Pi_{\theta, is}, \Lambda_{\theta,is})_{(\theta,s) \in [-\delta,\delta]^2 }$ which is continuous 
as a function from $ [-\delta,\delta]^2$ to $\complex \times (\cL(\cB,\wt \cB))^2 $, $$\Pi_{\theta,is} \cL_{\theta,is} = \cL_{\theta,is} \Pi_{\theta,is} = \lambda_\theta(is)\Pi_{\theta,is},$$
and for any $n$,
\begin{equation}\label{OpDecom1}
    \cL_{\theta,is}^n = \lambda_{\theta}(is)^n\Pi_{\theta,is}+\Lambda_{\theta,is}^n \in \cL(\cB,\cB)\cap \cL(\wt \cB,\wt \cB)
\end{equation}
where 
\begin{equation}\label{ConvRate}
    \sup_{|\theta|, |s| \in [0,\delta]} \|\Lambda^n_{\theta,is}\|_{\cB,\cB} =\cO(\bar\kappa^n).
\end{equation}
In addition, for all $\theta \in 
[-\delta,\delta]$,  we have that 
\begin{equation}\label{LambdaContinuous}
\lambda_\theta\colonequals \lambda_\theta(0)=1 = \lim_{s \to 0} \lambda_{\theta} (is),
\end{equation}
$g_\theta$ admits a unique invariant measure $\nu_\theta$ with density $\rho_{\nu_\theta}$, $\Pi_{\theta}{(\cdot)}\colonequals \Pi_{\theta, 0}{(\cdot)} = m(\,\cdot\,)\rho_{\nu_\theta}$, and
\begin{equation}\label{eq:ReslvBnd}
    K \colonequals  \sup_{|\theta|, |s| \in [0,\delta]}  \sup \Big\{\|(z-\cL_{\theta,is})^{-1}\|_{\cB,\cB}\,\Big|\, |z|\geq \bar\kappa, |z-1|\geq (1-\bar\kappa)/2\Big\} < \infty .
\end{equation}
\end{lem}
\begin{rem}   From \eqref{ConvRate}, it follows that, for small $\theta$ and $s$, the spectral radius of $\Lambda_{\theta,is}:\cB\to\cB$ is at most $\bar\kappa<1$. 
In fact, $\lambda_{\theta}(is)$ is the dominating eigenvalue of $\cL_{\theta,is}$ with the corresponding rank-one eigenprojection  $\Pi_{\theta,is}$. Also, $$\lim_{(\theta,s)\to (0,0)} \|\Pi_{\theta,is}-\Pi_{0}\|_{\cB, \wt\cB}=0$$ by continuity at $(0,0)$. The uniform boundedness of the resolvent \eqref{eq:ReslvBnd} is used in \Cref{lem:ProjBds} to show that the first three $s-$derivatives of $\Pi_{\theta,is}$ are continuous at $\theta=0$. 
\end{rem}

\begin{proof}
There are two cases. 

First, assume that $\cB=\wt \cB$. Then, by the classical perturbation theory of linear operators, (\cite[Chapter XI]{HennionHerve} and \cite[Chapters 7 and 8]{Kato}),
there exists $\delta>0$ such that for all $|s|, |\theta|\leq \delta$, $\cL_{\theta, is}$ as an operator in $\cL(\cB,\cB)$ has the decomposition 
\begin{equation}\label{OpDecom}
    \cL_{\theta,is} = \lambda_\theta(is) \Pi_{\theta, is} + \Lambda_{\theta, is}
\end{equation}
where $\Pi_{{\theta,}is}$ is the eigenprojection to the top eigenspace of $\cL_{{\theta,}is}$, the essential spectral radius of $\Lambda_{\theta, is}$ is strictly less than $|\lambda_\theta(is)|$, and $\Lambda_{\theta, is}\Pi_{\theta, is} = \Pi_{\theta, is}\Lambda_{\theta, is} = 0$. Also, $(\theta, s) \mapsto (\lambda_\theta(is),\Pi_{\theta,is}, \Lambda_{\theta,is})$ is continuous from $[-\delta,\delta]^2$ to $\mathbb C\times(\cL(\cB,\cB))^2$. 
Iterating \eqref{OpDecom}, it follows that for all $|s|, |\theta|\leq \delta$,
\begin{equation}\label{IterOpDecom}
 \cL^n_{\theta, is} = \lambda_\theta(is)^n\Pi_{\theta, is}+\Lambda^n_{\theta, is} \, ,
\end{equation}
with $\sup_{|s|,|\theta| \in [0, \delta]}\Vert \Lambda_{is}^n \Vert_{\cB,\cB}=\mathcal O(\bar\kappa^n)$. 

Second, if the spaces are different, we apply the Keller-Liverani perturbation theorem, \cite[Theorem 1]{KellerLiverani} as formulated in \cite[Theorem (K-L)]{HervePene}. 
This gives the decomposition $\eqref{OpDecom}$ as operators in $\cL(\cB,\wt \cB)$, the continuity of $(\theta, s) \mapsto (\lambda_\theta(is),\Pi_{\theta,is}, \Lambda_{\theta,is})$ as a map from $[-\delta,\delta]^2$ to $\mathbb C\times(\cL(\cB, \wt \cB))^2$, and the estimate $$\sup_{|\theta|,|s|\in [0, \delta]}\Vert \Lambda_{\theta, is}^n \Vert_{\cB, \cB}=\mathcal O(\bar\kappa^n).$$

In both cases, $\lambda_\theta\colonequals \lambda_\theta(0)$ is the leading eigenvalue of $\cL_\theta $ and that the rest of the spectrum of $\cL_\theta $ is inside a disk of radius $\bar\kappa$ centered at the origin. So, $\lambda_\theta=1$. Also, $1$ is an isolated simple eigenvalue of $\cL_\theta$ because the rank of eigenprojections (multiplicity of eigenvalues) are preserved  
(see \cite[Chapter 8.3]{Kato} and \cite[Lemma 1]{KellerLiverani}). 
So, there exists $\rho_{\nu_\theta} \geq 0$ such that $m(\rho_{\nu_\theta})=1$ 
and $\cL_\theta(\rho_{\nu_\theta})=\rho_{\nu_\theta}$ . Since $\cL_\theta$ is a transfer operator of $g_\theta$, this is equivalent to the existence of a unique ergodic absolutely continuous invariant measure $\nu_\theta$ with density $\rho_{\nu_\theta}$ (See \cite[Proposition 4.2.7]{Gora}). 
Both $\Pi_{\theta}(\cdot)=m(\,\cdot\,)\rho_{\nu_\theta}$ 
and \eqref{eq:ReslvBnd} follow from 
\cite[Chapter XI]{HennionHerve} and \cite{KellerLiverani}.  
\end{proof}
\begin{rem}
Note that, for each $\theta$, we could have applied the 
corresponding perturbation theorems to $s \mapsto \cL_{\theta, is}$ with the additional assumption that $\cL_\theta$ has a spectral gap.
From this, we do not obtain results uniform in $\theta$. 
\end{rem}
\begin{rem}\label{rem:AltAssump1}
Under the weaker assumption \eqref{eq:AltAssump}, we would still have the spectral decomposition \eqref{OpDecom} and the rest of the conclusion except the $(\theta,s)-$continuity of  $(\lambda_\theta(is), \Pi_{\theta, is}, \Lambda_{\theta,is})$. Instead, it would be continuous at $(0,0)$, and also, there is an $s-$neighbourhood 
$(-\delta,\delta)$ on which for each $\theta$ near $0$, $s \mapsto  (\lambda_\theta(is), \Pi_{\theta, is}, \Lambda_{\theta,is})$ is continuous. 
\end{rem}

To establish a uniform first-order Edgeworth expansion, the continuity of $s \mapsto \cL_{\theta,is}$, established in \Cref{WeakPert} is not sufficient. 
Moreover, $S_{\theta,n}(h_\theta)$ should have at least three moments. 
The regularity assumption in \Cref{asm:beta} of $s \mapsto \cL_{\theta,is}$ being $C^3$ is sufficient for the existence of three moments (see \Cref{sec:Moments}) and yields the asymptotic expansions for the characteristic functions (see \Cref{sec:AsympCharFn}).  

\begin{lem} \label{StrongPert}
Suppose the Assumptions \ref{asm:alpha}, \ref{asm:beta}, and \ref{asm:gamma}(1--3) hold. 
Let $\bar \kappa \in (\kappa,1)$. Then, there exists $\delta>0$ and a family $(\lambda_\theta(is),\Pi_{\theta,is},\Lambda_{\theta,is})_{(\theta, s)\in [-\delta,\delta]^2}$ which is continuous as a function in $(\theta,s)$ from $[-\delta,\delta]^2$ to $\mathbb C\times (\cL(\cB,\wt\cB))^2$ and for each $\theta \in [-\delta,\delta]$, $C^3-$smooth as a function in $s$ from $[-\delta,\delta]$ to $\mathbb C\times (\cL(\cX_0,\cX_{3}^{(+)}))^2$, 
$$\Pi_{\theta,is} \cL_{\theta,is} = \cL_{\theta,is} \Pi_{\theta,is} = \lambda_\theta(is)\Pi_{\theta,is},$$ and for all $n$,
\begin{equation}\label{DecompOp3'}
\mathcal L_{\theta, is}^n=\lambda_\theta(is)^n\Pi_{\theta,is}+\Lambda_{\theta,is}^n\,\,\,\text{in}\,\,\,\bigcap_{a=0}^{3}\left(\mathcal L(\cX_a,\cX_a)\cap\cL(\cX_a^{(+)},\cX_a^{(+)})\right),
\end{equation}
where for each $\theta$,
\begin{equation}\label{DecompOp3'Compl}
\max_{j=0, \dots ,3}\sup_{|s| \in [0,\delta]} \Vert (\Lambda_{\theta,is}^n)^{(j)}\Vert_{\cX_0,\cX^{(+)}_{3}}=\cO(\bar\kappa^n).
\end{equation}
\end{lem}
\begin{rem}
Note that this Lemma is similar to \citep[Proposition 1.9]{FernandoPene}. However, we consider a two-dimensional perturbation: a $C^3$ $s-$perturbation and a continuous $\theta-$perturbation. If we were to apply the two dimensional generalization of \citep[Proposition 1.9]{FernandoPene}, we would have required the $\theta-$perturbation to be $C^3$ as well. Instead, we opt for weaker assumptions which are easier to verify. As a result, our conclusions are weaker but sufficient for applications we have in mind.
\end{rem}

\begin{proof} If $\cB = \wt \cB$, then the theorem follows directly from the classical perturbation theory of linear operators. 

In the case of $\cX_0=\cX^{(+)}_3$ and $\cB \neq \wt \cB$, we consider the chain of spaces:
$$\cB \hookrightarrow \wt \cB \hookrightarrow L^{p_0}(m)$$
Write $\cV_{1}=(-\delta,\delta)$, with $\delta$ as in \Cref{WeakPert}.
The Assumptions \ref{asm:alpha} and \ref{asm:gamma}(1--2) yield the spectral decomposition \eqref{DecompOp3'} of $\cL_{\theta,is}$ as operators in $\cL(\cB,\cB) \cap \cL(\wt \cB,\wt \cB)$, and the continuity of the spectral data as functions of $(\theta,s) \in \cV_1 \times \cV_1$. Also, for $\theta \in \cV_1$, $\cL_\theta$ has a spectral gap of $(1-\bar\kappa)$ on $\cB$ and we have the uniform boundedness of the resolvent given in \eqref{eq:ReslvBnd}. \Cref{asm:beta} says that for all $\theta \in \cV_1$, $s\mapsto \cL_{\theta,is}$ is $C^3$ as a function from $\cV_1$ to $\cL(\cB,\cB)$.  

Recall from \cite[Section 7.2]{HervePene} that $$\Pi_{\theta,is} = \frac{1}{2\pi i } \oint_{\Gamma}(z-\cL_{\theta,is})^{-1} \, dz$$
where $\Gamma$ is a circle centred at $1 \in \complex$ with radius $\eps+(1-\bar\kappa)/2$ for any $\eps$ sufficiently small. For a function $f$ of two variables $x$ and $y$, let $\Delta_h[f(x,y)]\colonequals  f(x,y+h) - f(x,y)$, and let the superscript $(j)$ denote the $j$th derivative with respect to $s$. Then,
\begin{align}\label{eq:Proj1}
    \Pi^{(1)}_{\theta,is} &=\lim_{h \to 0} \frac{1}{h}(\Pi_{\theta,i(s+h)} - \Pi_{\theta,i(s+h)}) \nonumber\\ &= \frac{1}{2\pi i } \oint_{\Gamma} \lim_{h\to 0}\frac{1}{h}[(z-\cL_{\theta,i(s+h)})^{-1}- (z-\cL_{\theta,is})^{-1}] \, dz \nonumber \\ &=\frac{1}{2\pi i } \oint_{\Gamma} \lim_{h\to 0}\frac{1}{h}[(z-\cL_{\theta,is})^{-1}(\cL_{\theta,i(s+h)} - \cL_{\theta,is})(z-\cL_{\theta,i(s+h)})^{-1}] \, dz \nonumber
    \\ &=\frac{1}{2\pi i } \oint_{\Gamma} (z-\cL_{\theta,is})^{-1}\cL^{(1)}_{\theta,is}(z-\cL_{\theta,is})^{-1} \, dz\, 
\end{align}
\begin{align}\label{eq:Proj2}
    \Pi^{(2)}_{\theta,is} &= \frac{1}{2\pi i } \oint_{\Gamma} \lim_{h\to 0}\frac{1}{h}\Delta_{h}[(z-\cL_{\theta,is})^{-1}\cL^{(1)}_{\theta,is}(z-\cL_{\theta,is})^{-1}] \, dz \nonumber \\ &=\frac{1}{2\pi i } \oint_{\Gamma} \lim_{h\to 0}\frac{1}{h}\Delta_{h}[(z-\cL_{\theta,is})^{-1}]\cL^{(1)}_{\theta,is}(z-\cL_{\theta,is})^{-1} \, dz  \nonumber \\ &\phantom{aaaaaaaaa}+ \frac{1}{2\pi i } \oint_{\Gamma} \lim_{h\to 0}\frac{1}{h}(z-\cL_{\theta,i(s+h)})^{-1}\Delta_{h}[\cL^{(1)}_{\theta,is}](z-\cL_{\theta,is})^{-1}\, dz \nonumber \\ &\phantom{aaaaaaaaaaaaaaa}-\frac{1}{2\pi i } \oint_{\Gamma} \lim_{h\to 0}\frac{1}{h} \nonumber (z-\cL_{\theta,i(s+h)})^{-1}\cL^{(1)}_{\theta,i(s+h)}\Delta_{h}[(z-\cL_{\theta,is})^{-1}] \, dz \nonumber
    \\ &= \frac{1}{2\pi i } \oint_{\Gamma} (z-\cL_{\theta,is})^{-1}\cL^{(2)}_{\theta,is}(z-\cL_{\theta,is})^{-1}\, dz\, ,
\end{align}
and similarly, 
\begin{align}\label{eq:Proj3}
    \Pi^{(3)}_{\theta,is} = \frac{1}{2\pi i } \oint_{\Gamma} (z-\cL_{\theta,is})^{-1}\cL^{(3)}_{\theta,is}(z-\cL_{\theta,is})^{-1}\, dz 
\end{align}
whenever the integrals make sense. Note that $(z - \cL_{\theta,is})^{-1}, \cL^{(3)}_{\theta,is} \in \cL(\cB,\cB)$ for $(\theta,s)\in \cV_1 \times \cV_1$. Therefore, for $\theta \in \cV_1$, $s\mapsto \Pi_{\theta,is}$ is $C^3$ as a function from $\cV_1$ to $\cL(\cB,\cB)$.

A similar argument gives, for $j=1,2,3$,
$$\Lambda^{(j)}_{\theta,is} = \frac{1}{2\pi i } \oint_{\wt\Gamma}(z-\cL_{\theta,is})^{-1}\cL^{(j)}_{\theta,is}(z-\cL_{\theta,is})^{-1}\, dz$$
where $\wt\Gamma$ is a circle centred at $0 \in \complex$ with radius $\eps+\kappa$ for any $\eps$ sufficiently small because
$$\Lambda_{\theta,is} = \frac{1}{2\pi i } \oint_{\wt\Gamma}(z-\cL_{\theta,is})^{-1} \, dz.$$
Therefore, for $\theta \in \cV_1$, $s\mapsto \Lambda_{\theta,is}$ is $C^3$ on $\cV_1$. Moreover, since 
$$\Lambda^{n}_{\theta,is} = \frac{1}{2\pi i } \oint_{\wt\Gamma}z^n(z-\cL_{\theta,is})^{-1} \, dz\, ,$$
we have that
$$[\Lambda^n_{\theta,is}]^{(j)} = \frac{1}{2\pi i } \oint_{\wt\Gamma}z^n(z-\cL_{\theta,is})^{-1}\cL^{(j)}_{\theta,is}(z-\cL_{\theta,is})^{-1}\, dz\, ,$$
and as a result, $$\|[\Lambda^n_{\theta,is}]^{(j)} \|_{\cB,\cB} \leq (2\pi)^{-1}K^2\|\cL^{(j)}_{\theta,is}\|_{\cB,\cB} \kappa^n = \cO_\theta(\bar \kappa^n).$$

Note that for $(\theta,s)\in \cV_1 \times \cV_1$, due to \eqref{OpDecom},
$$\cL_{\theta,is}\Pi_{\theta,is} = \lambda_\theta(is)\Pi_{\theta,is}$$
and hence, $$\mu(\cL_{\theta,is}\Pi_{\theta,is}{\bf 1}_{X}) = \lambda_\theta(is)\mu(\Pi_{\theta,is}{\bf 1}_{X}).$$
In particular, $$\lim_{(\theta,s)\to (0,0)}m(\Pi_{\theta,is}{\bf 1}_{X})=m(\Pi_0{\bf 1}_{X})=m(m({\bf 1}_{X})\rho_{\nu_0})=1 \neq 0$$
So, reducing $\delta$ if necessary, for $(\theta,s) \in \cV_1 \times \cV_1$, 
\begin{equation}\label{eq:EigenVal}
    \lambda_\theta(is) =\frac{m(\cL_{\theta,is}\Pi_{\theta,is}{\bf 1}_{X})}{ m(\Pi_{\theta,is}{\bf 1}_{X})}.
\end{equation}
Note that 
\begin{align*}
    [\mu(\cL_{\theta,is} \Pi_{\theta,is}{\bf 1}_X)]^{(j)}  &=\sum_{k=0}^j \binom{j}{k} \mu(\cL_{\theta,is}^{(k)}\Pi^{(j-k)}_{\theta,is}{\bf 1}_X),\,\, j = 1,2,3,\,\,\, \text{and}\\
    [\mu(\Pi_{\theta,is}{\bf 1}_X)]^{(j)} &=\mu(\Pi^{(j)}_{\theta,is}{\bf 1}_X),\,\, j = 1,2,3.
\end{align*}
for any $\mu$. So, for each $\theta \in \cV_1$, for each $j=1,2,3$, $\lambda^{(j)}_\theta(is)$ can be written as a linear combination of $s-$continuous functions on $\cV_1$ using the quotient rule. So, $\lambda_\theta(\cdot) \in C^3(\cV_1)$ as required.

{For the general case, consider chains of length 3:
$$\cC \hookrightarrow \wt \cB \hookrightarrow L^{p_0}(m)$$
where $\cC$ is either $\cX_a$ or $\cX^{(+)}_a$ for $a=0,1,2,3$. Then conditions in the \Cref{asm:alpha} are satisfied with $\cB$ replaced by $\cC$\,: $\cL_{\theta,is}$ are bounded linear operators on $\cC$ and $\wt \cB$ and $(\theta,s) \mapsto \cL_{\theta,is} \in \cL(\cC,\wt \cB)$ is continuous. The uniform Doeblin-Fortet inequality follows from the \Cref{asm:beta}(3). Also, the \Cref{asm:gamma}(3) tells that $\cL_{0}$ has a spectral gap of $(1-\kappa)$ on $\cC$. So, we have that there exists $\cV_{\cC}$, a neighbourhood of $0$, such that for all $\theta \in \cV_\cC$, $\cL_\theta$ has a spectral gap of $(1-\bar \kappa)$ on $\cC$. Also, we have the uniform boundedness of the resolvent 
$$\sup_{|\theta|, |s| \in \cV_{\cC}}  \sup \Big\{\|(z-\cL_{\theta,is})^{-1}\|_{\cC,\cC}\,\Big|\, |z|\geq \bar\kappa, |z-1|\geq (1-\bar\kappa)/2\Big\} < \infty .$$}

{For a fixed $\theta$, we need to show the $C^3$ $s-$regularity of spectral data. To this end, we adapt the proof of \citep[Proposition 1.9]{FernandoPene}. So, we check conditions of \cite[Proposition A.1]{HervePene}, and in particular, the Condition $\cD(3)$ there. Consider the neighbourhood $\cV_2=(-\delta,\delta)$ of $0$, and write $$I=\{\cX_a, a=0,1,2,3\} \cup \{\cX_a^{(+)}, a =0,1,2,3\},$$ and let $j=1,2,3$. Define 
$T_0, T_1$ on $I$ by 
$$T_0(\cX_a)=\cX_a^{(+)},\,\, T_1(\cX^{(+)}_a)=\cX_{a+1},$$
and map to $\{0\}$ otherwise. Note that we have omitted $\cX_4$ from $I$. Then, we have condition $\cD(3)(0)$. Also, for $a=0,1,2,3$, $s \mapsto \cL_{\theta,is}$ is continuous as a function from $\cV_2$ to $\cL(\cX_a,\cX^{(+)}_a)$ due to the \Cref{asm:beta}(1). This is $\cD(3)(1)$. Note that $$\cX_{a+j}=T_1(T_0T_1)^{j-1}(\cX^{(+)}_a)$$ for $a=0,1,2$ and $j=1,\dots,3-a$. Due to the \Cref{asm:beta}(2), we have $\cD(3)(2)$ because $s \mapsto \cL_{\theta,is}$ is $C^j$ as a function from $\cV_2$ to $\cL(\cX^{(+)}_a,\cX_{a+j})$. We already have $\cD(3)(3')$ because of the uniform boundedness of $R_z(\theta,s)$ at the beginning. }

Define $\cV = \cap_{\cC} \cV_\cC$. Then, reducing $\delta$ if necessary, $[-\delta, \delta] \subseteq \cV$, and for all $\theta \in [-\delta, \delta]$, we have the $C^3$ smoothness of the spectral data as a function of $s \in [-\delta, \delta]$. This follows from the conclusion \cite[Proposition A.1]{HervePene} and the discussion preceding it. \eqref{DecompOp3'Compl} follows from \cite[Corollary 7.2]{HervePene}. 
\end{proof}

\begin{rem}\label{rem:CompProjBnd}
{It follows that if $\theta \mapsto \cL^{(j)}_{\theta,is}$ is continuous at $0$ (for which we give a sufficient condition in \Cref{lem:ProjBds}), then the implied  constant in \eqref{DecompOp3'Compl}, $$(2\pi)^{-1}K^2\|\cL^{(j)}_{\theta,is}\|_{\cB,\cB}\, ,$$ is continuous at $\theta=0$, and hence, can be bounded by a constant independent of $\theta$.}
\end{rem}

\begin{rem}
We can say more. The same argument applied to sub-chains of spaces in \eqref{SpaceChain}, we can conclude the following. 
\begin{itemize}[leftmargin=25pt]
    \item The family $(\lambda_\theta(is),\Pi_{\theta,is},\Lambda_{\theta,is})_{(\theta, s)\in [-\delta,\delta]^2}$ is continuous as a function in $(\theta,s)$ from $[-\delta,\delta]^2$ to $\mathbb C\times (\cL(\cX_j,\cX_{j+m}^{(+)}))^2$ and for each $\theta \in [-\delta,\delta]$, $C^{m}$-smooth as a function in $s$ from $[-\delta,\delta]$ to $\mathbb C\times (\cL(\cX_j,\cX_{j+m}^{(+)}))^2$ for any $0\le j\le j+m\le 3$, 
    \item 
    \begin{equation}
        \max_{a=0,\dots,3}\max_{j=0, \dots ,3-a}\sup_{|\theta|,|s| \in [0,\delta]} \Vert (\Lambda_{\theta,is}^n)^{(j)}\Vert_{\cX_a,\cX^{(+)}_{a+j}}=\cO(\bar\kappa^n).
    \end{equation}
\end{itemize} 
\end{rem}

\begin{rem}\label{rem:AltAssump2}
Under the weaker assumption \eqref{eq:AltAssump}, as before, we have the spectral decomposition \eqref{OpDecom} and the rest of the conclusion except the $(\theta,s)-$continuity of  $(\lambda_\theta(is), \Pi_{\theta, is}, \Lambda_{\theta,is})$. Instead, we have the same conclusion of \Cref{rem:AltAssump1}. See also the remark appearing after \cite[Theorem  K-L]{HervePene}. 
\end{rem}

Next, under the extra assumption, the \Cref{asm:delta}, we establish that $$\Pi^{(j)}_{\theta,is},\, j=0,1,2,3, $$ as functions from $[-\delta,\delta]^2$ to $\cL(\cX_0,\cX_{3}^{(+)})$, are continuous at $(0,0)$, and in particular, they are bounded. This fact will be used to control the error in the asymptotic expansions of characteristic functions in \Cref{sec:AsympCharFn}.

\begin{lem}\label{lem:ProjBds}

Suppose the Assumption \ref{asm:alpha}, \ref{asm:beta}, \ref{asm:gamma}(1--3), and \ref{asm:delta} hold. 
Then, for all $j=0,1,2,3$,
\begin{align}
   \lim_{(\theta,s) \to (0,0)}\|\Pi^{(j)}_{\theta, is}-\Pi^{(j)}_0\|_{\cC, \cX^{(+)}_a}=0.
\end{align}
where $\cC$ is a Banach space in \eqref{SpaceChain} such that $\cC \hookrightarrow \dots \hookrightarrow \cX^{(+)}_a$ is any sub-chain of $j+2$ spaces in \eqref{SpaceChain} and $a\in \{0,1,2,3\}$ is such that such a space $\cC$ exists.
\end{lem}
\begin{proof}$j=0$ is already a conclusion of \Cref{StrongPert}. So, we focus on $j>0$.
From, \eqref{eq:ReslvBnd}, \eqref{eq:Proj1}, \eqref{eq:Proj2} and \eqref{eq:Proj3}, we have that 
the $\theta-$continuity of $\Pi^{(j)}_{\theta,is}$ is equivalent to $\theta-$continuity of $\cL^{(j)}_{\theta,is}$. To see this, note that
\begin{align*}
    \Pi^{(j)}_{\theta,is} - \Pi^{(j)}_{\bar \theta,i\bar s} &=  \frac{1}{2\pi i }\oint_{\Gamma} (z-\cL_{\theta,is})^{-1}\cL^{(j)}_{\theta,is}(z-\cL_{\theta,is})^{-1}\, dz - \frac{1}{2\pi i } \oint_{\Gamma} (z-\cL_{0})^{-1}\cL^{(j)}_{\bar \theta,i\bar s}(z-\cL_{\bar \theta,i\bar s})^{-1}\, dz \\ &=\frac{1}{2\pi i }\oint_{\Gamma} (z-\cL_{\theta,is})^{-1}(\cL^{(j)}_{\theta,is}-\cL^{(j)}_{\bar \theta,i\bar s})(z-\cL_{\theta,is})^{-1}\, dz\\ &\phantom{aaaaaaa}+\frac{1}{2\pi i }\oint_{\Gamma} ((z-\cL_{\theta,is})^{-1}-(z-\cL_{\bar \theta,i\bar s})^{-1})\cL^{(j)}_{\bar \theta,i\bar s}(z-\cL_{\theta,is})^{-1}\, dz\\ &\phantom{aaaaaaaaaaaa}+\frac{1}{2\pi i }\oint_{\Gamma} (z-\cL_{\bar \theta,i\bar s})^{-1}\cL^{(j)}_{\bar \theta,i\bar s}((z-\cL_{\theta,is})^{-1}-(z-\cL_{\bar \theta,i\bar s})^{-1})\, dz\, ,
\end{align*}
and hence, for $\cC_1 \hookrightarrow \cC_2 \hookrightarrow \wt\cB$,
\begin{align*}
    &\|\Pi^{(j)}_{\theta,is}- \Pi^{(j)}_{\bar \theta,i\bar s}\|_{\cC_1,\cC_2} \\ &\lesssim \|(z-\cL_{\theta,is})^{-1}\|_{\cC_2,\cC_2}\|\cL^{(j)}_{\theta,is}-\cL^{(j)}_{\bar \theta,i\bar s}\|_{\cC_1,\cC_2}\|(z-\cL_{\theta,is})^{-1}\|_{ \cC_1,\cC_1} \\ &\phantom{aaa}+\|(z-\cL_{\bar \theta,i\bar s})^{-1})\|_{\cC_2,\cC_2}\|\cL_{\theta,is}-\cL_{\bar \theta,i\bar s}\|_{\cC_1,\cC_2}\|(z-\cL_{\theta,is})^{-1}\|_{ \cC_1,\cC_1} \|\cL^{(j)}_{\bar \theta,i\bar s}\|_{ \cC_1,\cC_1}\|(z-\cL_{\theta,is})^{-1}\|_{ \cC_1,\cC_1}\\ &\phantom{aaaaaa}+\|(z-\cL_{\bar \theta,i\bar s})^{-1}\|_{\cC_2,\cC_2}\|\cL^{(j)}_{\bar \theta,i\bar s}\|_{\cC_2,\cC_2}\|(z-\cL_{\bar \theta,i\bar s})^{-1})\|_{\cC_2,\cC_2}\|\cL_{\theta,is}-\cL_{\bar \theta,i\bar s}\|_{\cC_1,\cC_2}\|(z-\cL_{\theta,is})^{-1}\|_{ \cC_1,\cC_1}.
\end{align*}
whenever each of the norms are finite. Here, the implied constants are independent of $s$ and $\theta$. 

Now, to prove the $\theta-$continuity of $\cL^{(j)}_{\theta,is}$, recall that $\cL^{(j)}_{\theta,is} = \cL_{\theta, is}((ih_\theta)^j\cdot)$ and $\cL_{\theta,is} \in \cL(\cX_a,\cX_a^{(+)})$. 
So, for $j=0,1,2,3$,
$$\cL^{(j)}_{\theta,is}=i^j\cL_{\theta, is} \circ H^j_\theta \in \cL(\cC,\cX^{(+)}_a) $$
for any $\cC$ in \eqref{SpaceChain} such that $\cC \hookrightarrow \dots \hookrightarrow \cX^{(+)}_a$ is any sub-chain of $j+2$ spaces in \eqref{SpaceChain}.
Our assumptions on the continuity of $\cL_{\theta,is}$ and $H_\theta$ in $(\theta,s)$ at $(0,0)$ gives the required result.
\end{proof}

From this, it follows that $\rho_{\nu_\theta}$ and $\lambda^{(j)}_\theta(is), j=0,1,2,3$ are continuous in $\theta$ at $\theta=0$.  
\begin{cor}\label{cor:CtsEigenV}
Under the Assumptions \ref{asm:alpha}, \ref{asm:beta}, \ref{asm:gamma}(1--3), and \ref{asm:delta}, for all $j=0,1,2,3$,
$$\lim_{(\theta,s) \to (0,0)}\lambda^{(j)}_\theta(is) = \lambda^{(j)}(0).$$
\end{cor}
\begin{proof}
From \Cref{lem:ProjBds}, $\Pi_{\theta,is}{\bf 1}_X$, $\cL_{\theta,is}\Pi_{\theta,is}{\bf 1}_X$ and their first three derivatives with respect to $s$ are continuous at $(0,0)$ in $\theta$. 
Finally, due to \eqref{eq:EigenVal}, we can represent the $s-$derivatives of $\lambda_\theta(is)$ in terms of $s-$derivatives of $\cL_{\theta,is}\Pi_{\theta,is}{\bf 1}_X$ and $\Pi_{\theta,is}{\bf 1}_X$. So, we have the conclusion. 
\end{proof}
\begin{cor}\label{cor:CtsDensity}
Under the Assumptions \ref{asm:alpha}, \ref{asm:beta}, \ref{asm:gamma}(1--3) and \ref{asm:delta}, for all $\cC$ in \eqref{SpaceChain} with $\cC \neq \cB$ and $\cB \hookrightarrow \cC$
$$\lim_{\theta \to 0}\|\rho_{\nu_\theta} - \rho_{\nu_0}\|_{\cC} =0.$$
\end{cor}
\begin{proof}Assume $\cB \neq \cX^{(+)}_0$, if not, replace $\cX^{(+)}_0$ with the leftmost space $\cC$ in \eqref{SpaceChain} such that $\cC \neq \cB$ below. Then, the result follows from $$\|\rho_{\nu_\theta} - \rho_{\nu_0}\|_{\cC} \lesssim \|\rho_{\nu_\theta} - \rho_{\nu_0}\|_{\cX^{(+)}_0}= \|\Pi_\theta({\bf 1}_X)-\Pi_0({\bf 1}_X)\|_{\cX^{(+)}_0} \leq \|\Pi_\theta - \Pi_0\|_{\cB,\cX^{(+)}_0}\|{\bf 1}_X\|_\cB $$
and the continuity of $\theta \mapsto  \Pi_\theta$ as a function from $[-\delta,\delta]$ to $\cL(\cB,\cX^{(+)}_0)$ (the \Cref{asm:beta}(3)).
\end{proof}

\subsection{Asymptotic moments}\label{sec:Moments} In this section, we study the dependency of asymptotic moments on the parameter $\theta$ and the choice of initial measure $\mu$. From the analysis in \cite[Section 4]{FL}, we observe that there are constants $\{a_{k,j}\}$ (in our case, these depend on $\theta$) for $k=1,2,3$ and $j=0,\dots,[k/2]$  
such that
\begin{align*}
    i\EXP_\mu(S_{\theta,n}(h_\theta)- n A_\theta) &= a_{1,0} + \cO(\kappa^n), \\
    i^2\EXP_\mu([S_{\theta,n}(h_\theta)- n A_\theta]^2) & = a_{2,0} + a_{2,1}n + \cO(\kappa^n), \\
    i^3\EXP_\mu([S_{\theta,n}(h_\theta)- n A_\theta]^3) & = a_{3,0} + a_{3,1}n + \cO(\kappa^n),
\end{align*}
where the implied constant in $\cO(\cdot)$ is bounded by $\|\rho_\mu\|_{\cB}$ and continuous in $\theta$ at $\theta=0$ as we sall see from the proof of \Cref{ContVar} below.  

Next, we show that $A_\theta$ and $\sigma_\theta$ are independent of the initial measure $\mu$. This is important for applications we have in mind; see \Cref{sec:Accu}. Also, under the \Cref{asm:ve}, $\sigma_\theta>0$. Then from the continuity of $\sigma_\theta$, we obtain $\textstyle\inf_\theta \sigma_\theta >0.$ This fact is used in \Cref{lem:CharExp}. 

\begin{lem}\label{ContVar} 
Let $\cC$ be a space in \eqref{SpaceChain} such that $\cC \hookrightarrow \cX_3 \hookrightarrow \cX^{(+)}_3$. Suppose that the Assumptions \ref{asm:beta} and \ref{asm:gamma}(1) hold. 
Then, $A_{\theta}$, $\sigma_{\theta}$, $M_{\mu,\theta}$ are finite, the first two do not depend on the choice of $\mu$,
$$\lim_{\theta \to 0} A_\theta = A_0,\,\,\,\lim_{\theta \to 0} \sigma^2_\theta = \sigma^2_0,\,\,\,\lim_{\theta \to 0} M_{\mu,\theta} = M_{\mu,0},$$  and $\textstyle\sup_{|\theta|\leq \delta}\|M_{\mu,\theta}\|\lesssim \|\rho_\mu\|_\cC$.
Further, assume that \Cref{asm:ve} holds. Then, restricting $\delta$, if necessary, $\sigma_\theta \in (0,\infty)$ for $\theta \in [-\delta,\delta]$.
\end{lem}
\begin{proof}
Combining \eqref{eq:CharFunc} and \eqref{IterOpDecom}, we have,
\begin{equation}\label{eq:CharFnOp}
    \EXP_\mu(e^{isS_{\theta, n}(h_\theta)})=m(\cL^n_{\theta, is}(\rho_\mu)) = \lambda_\theta(is)^n m(\Pi_{\theta, is}\rho_\mu)+m(\Lambda^n_{\theta, is}\rho_\mu)
\end{equation}
Note that $m(\Pi_\theta\rho_\mu)=1$. Taking the first derivative of \eqref{eq:CharFnOp} at $s=0$, 
$$i\EXP_\mu(S_{n,\theta}(h_\theta))=\, n\, \lambda_\theta^{(1)}(0) +\, m(\Pi^{(1)}_\theta\rho_\mu)+m([\Lambda^n_{\theta}]^{(1)}\rho_\mu).$$
This yields,
$$-i\lambda^{(1)}_\theta(0)= \lim_{n \to \infty} 
\frac{1}{n}\EXP_\mu(S_{n,\theta}(h_\theta))=A_\theta.$$
Note that the limit is independent of the initial measure $\mu$ because it is $\lambda^{(1)}_\theta(0)$, and hence, depends only on the dynamical systems $g_\theta$ and the reference measure $m$. So, $A_{\theta}$ is a function only of $\theta$. 
Also, due to \Cref{cor:CtsEigenV}, it follows that $\lim_{\theta \to 0}A_\theta = \lim_{\theta \to 0} -i\lambda^{(1)}_\theta(0) =-i\lambda^{(1)}_0(0) = A_0$. 

From now on, without loss of generality, we may assume that $A_{\theta} =0$. If not, we can consider $\wt h_\theta = h_\theta - A_{\theta}$. Next, taking the second derivative of \eqref{eq:CharFnOp} at $s=0$,
$$i^2\EXP_\mu(S_{n,\theta}(h_\theta))^2=n\lambda_\theta^{(2)}(0)+m(\Pi^{(2)}_\theta\rho_\mu)+m([\Lambda^n_{\theta}]^{(2)}\rho_\mu)$$
So, we have $$\sigma^2_{\theta}=\lim_{n\to \infty}\EXP_\mu\Big(\frac{S_{n,\theta}(h_\theta)}{n}\Big)^2=-\lambda^{(2)}_\theta(0) < \infty$$
This gives that $\sigma_{\theta}$ is independent of the choice of $\mu$. 
As before, due to \Cref{cor:CtsEigenV}, $$\lim_{\theta \to 0}\sigma^2_\theta = \lim_{\theta \to 0} -\lambda^{(2)}_\theta(0) =-\lambda^{(2)}_0(0) = \sigma^2_0.$$

It is standard that the \Cref{asm:ve} 
yields $\sigma_0>0$; see \cite[Lemma A.14]{DeSimoiLiverani}. 
Since $\theta \mapsto \sigma_\theta$ is continuous at $\theta=0$, the positivity of $\sigma_\theta$ in a neighbourhood of $0$ follows. 

Finally, taking the third derivative of \eqref{eq:CharFnOp} at $s=0$, we obtain,
$$i^3\EXP_\mu(S_{n,\theta}(h_\theta))^3=n[\lambda_\theta^{(3)}(0)+3\lambda^{(2)}_\theta(0)m(\Pi^{(1)}_{\theta}\rho_\mu)]+m(\Pi^{(3)}_\theta\rho_\mu)+m([\Lambda^n_{\theta}]^{(3)}\rho_\mu)$$
and from the first derivative equation, $$m(\Pi^{(1)}_{\theta}\rho_\mu) = \lim_{n \to \infty} i\EXP_\mu(S_{n,\theta}(h_\theta)) = \lim_{n \to \infty} i\EXP_\mu(S_{n,\theta}(h_\theta) - nA_\theta)  = a_{1,0}.$$
Therefore, 
$$M_{\mu,\theta} =  \lambda_\theta^{(3)}(0)+3\lambda^{(2)}_\theta(0)m(\Pi^{(1)}_{\theta}\rho_\mu) = \lambda_\theta^{(3)}(0)-3\sigma^2_\theta a_{1,0},$$
and due to \Cref{lem:ProjBds} and \Cref{cor:CtsEigenV} we have that $M_{\theta,\mu}$ is continuous at $\theta=0$. 
Also, $$|M_{\mu,\theta}| \leq \sup_{\theta} |\lambda^{(3)}_\theta(0)| +\|\rho_\mu\|_{\cC} \sup_{\theta}|\lambda^{(2)}_\theta(0)|\|\Pi^{(1)}_\theta\|_{\cC,\cX^{(+)}_3} \lesssim \|\rho_\mu\|_{\cC}$$
which proves the last claim.
\end{proof}

\begin{rem}\label{rem:0avg} Suppose $\nu_0(h_0)=m(h_0\rho_{\nu_0})=0$; if not, write $\wt h_0 = h_0 - \nu_0(h_0)$ instead of $h_0$. Then, due to \Cref{asm:gamma}, $\|\cL^n_0(h_0 \rho_{\nu_0})\|_{\cB} \lesssim \kappa^n$, and provided that $\cB \hookrightarrow L^2$ (which is the case in our examples), we have
$$\sum_{n=0}^\infty \|\cL^n_0(h_0 \rho_{\nu_0})\|_{L^2} \lesssim \sum_{n=0}^\infty \|\cL^n_0(h_0 \rho_{\nu_0})\|_{\cB} <\infty .$$ 
This implies that the sequence 
$$\left\{\sum_{k=0}^n h_0 \circ g^k_0\right\}$$
is bounded in $L^2$, and hence, has a $L^2-$weak convergent subsequence; see \cite[Footnote 89]{DeSimoiLiverani}. So, we have the \Cref{asm:ve}(2). 
\end{rem}
\begin{rem}
Note that we may have to restrict the original $[-\delta,\delta]$ neighbourhood to make sure that $\sigma_\theta>0$. But for our analysis existence of one such neighbourhood is sufficient. So, reducing $\delta$ if necessary, we write 
\begin{equation}\label{BarTildesigma}
\bar\sigma = \inf_{|\theta| \leq \delta} \sigma_{\theta} >0\,\,\,\text{and}\,\, \tilde\sigma = \sup_{|\theta| \leq \delta} \sigma_{\theta} <\infty.
\end{equation}
These bounds will appear in later proofs.
\end{rem}

\begin{rem}\label{rem:Moments}
Note that $a_{0,0}, \dots, a_{3,1}$ are continuous at $\theta=0$. Since $a_{2,1}=-\sigma^2_\theta$ and $a_{3,1}=-iM_{\mu,\theta}$, both are continuous at $\theta=0$. It is easy to see from the proof that $$a_{1,0}=m(\Pi_\theta^{(1)}\rho_\mu), a_{2,0}=m(\Pi_\theta^{(2)}\rho_\mu),\,\,\text{and}\,\, a_{3,0}=m(\Pi_\theta^{(3)}\rho_\mu).$$ So, each one of them is continuous at $\theta=0$ as claimed. Moreover, the error terms $m([\Lambda^n_\theta]^{(j)} \rho_\mu)$ in the expressions for $\EXP(S_{\theta,n}(h_\theta)-A_\theta)^j,\, j=1,2,3$ are continuous at $\theta=0$. Therefore, for each $n$, $\EXP_\mu(S_{\theta,n}(h_\theta)-A_\theta)^j \to \EXP_\mu(S_{0,n}(h_0)-A_0)^j$ as $\theta \to 0$, and hence, for all $j$, $$\lim_{\theta \to 0}\EXP_{\mu}(h_\theta \circ g^j_\theta)= \EXP_\mu(h_0 \circ g^j_0).$$ 
\end{rem}

As we have seen, the existence of asymptotic moments was a direct result of the $s-$regularity of the perturbations of the transfer operators and the $\theta-$regularity of asymptotic moments was a result of the $\theta-$regularity of perturbations of the transfer operators. The latter, as we shall see in \Cref{sec:Exmp}, is a result of the $\theta-$regularity of the family of dynamical systems $\{g_\theta\}$. Also, as discussed in \cite[Section 4]{FL}, the coefficients of the first-order Edgeworth expansions can be expressed in terms of asymptotic moments (more specifically, in terms of $a_{k,j}$'s). Assuming this fact, \Cref{rem:Moments} implies that the coefficients depend on the spectral data of the transfer operators. We will see this in the next section where we derive the expansions.  

\section{First-order Continuous Edgeworth Expansions}\label{sec:Expansions}

The derivation of the continuous Edgeworth expansion is an extension of the classical proof of the CLT due to L\'evy. First, we obtain the asymptotic expansions for the characteristic functions, $\EXP(e^{isS_{n,\theta}(h_\theta)})$, by writing their Taylor expansions. Due to \eqref{eq:Duality} and \eqref{DecompOp3'}, this expansion follows from the expansions of $(\lambda_\theta(is))^n$ and $\Pi_{\theta,is}$, and the contribution from $\Lambda^n_{\theta, is}$ is negligible. This is discussed in \Cref{sec:AsympCharFn}. Then, in \Cref{sec:DistExp}, we use the expansion of the characteristic function, $\EXP(e^{isS_{n,0}(h_0)})$, to obtain an expansion for the distribution functions. Even though we obtain just the first-order expansion here, it is uniform in the initial distribution $\mu$ as well as in $\theta$, i.e., uniform for the family $g_\theta$, and hence, they are more general than the first order expansions in the current literature. This is precisely our theoretical contribution to the theory of Edgeworth expansions in dynamical systems.

Throughout the discussion, we let $\Omega \subset \cM_1(X)$ be such that all $\mu \in \Omega$ are absolutely continuous with density functions $\rho_\mu \in \cC$ and $$\sup_{\mu}\|\rho_\mu\|_{\cC} < \infty$$ where $\cC$ is a space in \eqref{SpaceChain} such that $\cC \hookrightarrow \cX_3 \hookrightarrow \cX^{(+)}_3$.

\subsection{Asymptotic expansions of characteristic functions}\label{sec:AsympCharFn}

Here, we establish an asymptotic expansion characteristic functions of $S_{n,\theta}(h_\theta)$, 
$$\EXP_\mu(e^{is S_{n,\theta}(h_\theta)/\sqrt{n} }),\,\, \theta \to 0,\,\, n \to \infty$$
using our spectral assumptions.   
 
\begin{lem} \label{lem:CharExp} Suppose that the Assumptions \ref{asm:alpha}, \ref{asm:beta}, \ref{asm:gamma}(1-3), \ref{asm:delta}, and \ref{asm:ve} hold.
Then there exist $\delta >0$ and $\kappa \in (0,1)$ and constants $a$ and $b$ independent of $\theta$ such that for $|s|<\bar\sigma \delta \sqrt{n}$ (where $\bar\sigma$ is as in \eqref{BarTildesigma}), 
\begin{multline}\label{eq:CharExpansion}
    \sup_{\mu \in \Omega} \left|\EXP_\mu\Big(e^{ is\frac{S_{n,\theta}(h_\theta)-nA_{\theta}}{\sigma_{\theta}\sqrt{n}}}\Big) - e^{-s^2/2}\left\{1+\frac{1}{\sqrt{n}}\left(\frac{a}{6}\cdot(is)^3+b\cdot(is)\right)\right\}\right| \\ = R(|s|)\left(e^{-s^2/4}o\Big(\frac {1}{\sqrt{n}}\Big)+\cO(\kappa^n)\right)  
\end{multline}
as $n \to\infty$ and $\theta \to 0$ where $R$ is a polynomial with $R(0)=0$. 
\end{lem}

\begin{proof}
For the sake of notational simplicity, we drop the subscript indicating the dependency on $\theta$. Whenever the dependency is relevant, we  reintroduce the subscript. Recall that, under the assumptions,
the Lemmas \ref{WeakPert}, \ref{StrongPert}, \ref{lem:ProjBds} and \ref{ContVar}, and the Corollary \ref{cor:CtsEigenV} hold. 

Due to \eqref{LambdaContinuous}, $\lambda(0)=1$. Using \eqref{eq:CharFunc} and \eqref{IterOpDecom}, for $\theta$ and $\delta>0$ sufficiently close to $0$, for
$|s|<\sigma\delta \sqrt{n}$,
\begin{equation}\label{eq:fundamental}
    \EXP_\mu\left(e^{is\frac{S_n(h)-nA}{\sigma \sqrt{n}}}\right) = \lambda\left(\frac{is}{\sigma \sqrt{n}}\right)^n e^{-is\frac{nA}{\sigma \sqrt{n}}} m\Big(\Pi_{\frac{is}{\sigma \sqrt{n}}}\rho_\mu\Big) + m\Big(\Lambda^n_{\frac{is}{\sigma \sqrt{n}}}\rho_\mu\Big) 
\end{equation}
Using the Taylor expansion,
\begin{align*}
    \lambda &\left(\frac{is}{\sigma \sqrt{n}}\right)^n = \exp n\log\lambda\left(\frac{is}{\sigma \sqrt{n}}\right) \nonumber \\
    &= \exp n\left( \sum_{k=0}^3(\log \lambda)^{(k)}(0)\frac{s^k}{k!\sigma^k n^{k/2}} + \frac{s^3}{\sigma^3 n^{3/2}}\psi\left(\frac{is}{\sigma\sqrt{n}}\right) \right) \nonumber \\
    &= \exp \left(is\frac{nA}{\sigma \sqrt{n}}+(\log\lambda)^{(2)}(0)\frac{s^2}{2\sigma^2}+ (\log\lambda)^{(3)}(0)\frac{s^3}{6\sigma^3\sqrt{n} }+  \frac{s^3}{\sigma^3 \sqrt{n}}\psi\left(\frac{is}{\sigma\sqrt{n}}\right)\right)\\
    &= \exp \left(is\frac{nA}{\sigma \sqrt{n}}+(\lambda^{(2)}(0)-\lambda^{(1)}(0)^2)\frac{s^2}{2\sigma^2}+ (\log\lambda)^{(3)}(0)\frac{s^3}{6\sigma^3\sqrt{n} }+  \frac{s^3}{\sigma^3 \sqrt{n}}\psi\left(\frac{is}{\sigma\sqrt{n}}\right)\right)\\
    &= \exp \left(is\frac{nA}{\sigma \sqrt{n}}-\frac{s^2}{2}+ (\log\lambda)^{(3)}(0)\frac{s^3}{6\sigma^3\sqrt{n} }+  \frac{s^3}{\sigma^3 \sqrt{n}}\psi\left(\frac{is}{\sigma\sqrt{n}}\right)\right)
\end{align*}
where $\psi$ is continuous at $0$ with $\psi(0)=0$. Here, we also used the fact that $iA=\lambda^{(1)}(0)$ and $i^2\sigma^2 =\lambda^{(2)}(0)-\lambda^{(1)}(0)^2$ when $A\neq 0$ which follows from \eqref{eq:CharFnOp}. 
So, 
\begin{align}
    \lambda\left(\frac{is}{\sigma \sqrt{n}}\right)^n e^{-is\frac{nA}{\sigma \sqrt{n}}}
    &= e^{-s^2/2}\exp \left((\log\lambda)^{(3)}(0)\frac{s^3}{6\sigma^3\sqrt{n} }+  \frac{s^3}{\sigma^3 \sqrt{n}}\psi\left(\frac{is}{\sigma\sqrt{n}}\right)\right),
\end{align}
and 
\begin{align}\label{NeedName1}
    \left|m\left(\Pi_{\frac{is}{\sigma \sqrt{n}}}\rho_\mu\right)
    -1-\frac{s}{\sigma\sqrt{n}}m(\Pi^{(1)}_0 \rho_\mu)\right| \leq \frac{s^2}{\sigma^2n} \sup_{\gamma \in [0,1]}\left|m(\Pi^{(2)}_{\frac{i\gamma s}{\sigma\sqrt{n}}}\rho_\mu)\right|.
\end{align}
Write $\alpha = (\log\lambda)^{(3)}(0)\frac{s^3}{6\sigma^3\sqrt{n} }+  \frac{s^3}{\sigma^3 \sqrt{n}}\psi\left(\frac{is}{\sigma\sqrt{n}}\right)$ and $\beta=(\log\lambda)^{(3)}(0)\frac{s^3}{6\sigma^3\sqrt{n} }$. Then using 
$$|e^{\alpha}-(1+\beta)| \leq e^{\max{(|\alpha|,|\beta|)}}\left(|\alpha-\beta|+\frac{1}{2}|\beta|^2\right),$$
we estimate,
\begin{multline*}
    \left| \lambda\left(\frac{is}{\sigma \sqrt{n}}\right)^n e^{-is\frac{nA}{\sigma \sqrt{n}}+s^2/2} -\left(1+(\log\lambda)^{(3)}(0)\frac{s^3}{6\sigma^3\sqrt{n}}\right)\right| \\ \leq e^{s^2/4} \left( \frac{|s|^3}{\sigma^3 \sqrt{n}}\left|\psi\left(\frac{is}{\sigma\sqrt{n}}\right)\right|+|(\log\lambda)^{(3)}(0)|^2\frac{|s|^6}{36\sigma^6n }\right).
\end{multline*}
Therefore,
\begin{multline}\label{NeedName2}
    \left| \lambda\left(\frac{s}{\sigma \sqrt{n}}\right)^n e^{-is\frac{nA}{\sigma \sqrt{n}}} -e^{-s^2/2}\left(1+(\log\lambda)^{(3)}(0)\frac{s^3}{6\sigma^3\sqrt{n}}\right)\right| \\ \leq e^{-s^2/4} \left( \frac{|s|^3}{\sigma^3 \sqrt{n}}\left|\psi\left(\frac{is}{\sigma\sqrt{n}}\right)\right|+|(\log\lambda)^{(3)}(0)|^2\frac{|s|^6}{36\sigma^6n }\right).
\end{multline}
Reintroducing the dependency on $\theta$ and $\mu$, and writing 
\begin{equation}\label{EdgePolyFourier}
     Q_{\theta}(s)=(\log\lambda_\theta)^{(3)}(0)\frac{s^3}{6\sigma_\theta^3}+m(\Pi_\theta^{(1)}\rho_\mu)\frac{s}{\sigma_\theta}.
\end{equation}
Combining \eqref{NeedName1} and \eqref{NeedName2}, and substituting it in \eqref{eq:fundamental}, we have 
\begin{align}\label{eq:thetaCharExp}
   &\left| \EXP_\mu\left(e^{is\frac{S_{\theta,n}(g)-nA_{\theta,\mu}}{\sigma_\theta \sqrt{n}}}\right) -e^{-s^2/2}\left(1+\frac{Q_\theta(s)}{\sqrt{n}}\right)\right| \\
   &\leq \left|\lambda\left(\frac{is}{\sigma \sqrt{n}}\right)^n e^{-is\frac{nA}{\sigma \sqrt{n}}} m\Big(\Pi_{\theta,\frac{i\gamma s}{\sigma_\theta\sqrt{n}}}\rho_\mu\Big) + m\Big(\Lambda^n_{\theta,\frac{i\gamma s}{\sigma_\theta\sqrt{n}}}\rho_\mu\Big) -e^{-s^2/2}\left(1+
  \frac{Q_\theta(s)}{\sqrt{n}}\right)\right| \nonumber \\
   &\leq e^{-s^2/4} \left( \frac{|s|^3}{\sigma_\theta^3 \sqrt{n}}\left|\psi_\theta\left(\frac{is}{\sigma_\theta\sqrt{n}}\right)\right|+|(\log\lambda_\theta)^{(3)}(0)|^2\frac{|s|^6}{36\sigma^6_\theta n }\right) +\sup_{|s|\leq  \bar\sigma \delta \sqrt{n}}  \Big|m\Big(\Lambda^n_{\theta,\frac{i\gamma s}{\sigma_\theta\sqrt{n}}}\rho_\mu\Big)\Big|  \nonumber  \\&\phantom{aaaaaaaaaaaaaa}+| m(\Pi^{(1)}_\theta\rho_\mu)| \sup_{\gamma\in [0,1]} |(\log \lambda_\theta)^{(3)}(i\gamma s)|\frac{|s|^4}{6\sigma_\theta^4n} + \frac{s^2}{\sigma^2_\theta n} \sup_{\gamma \in [0,1]}\left|m\Big(\Pi^{(2)}_{\theta,\frac{i\gamma s}{\sigma_\theta\sqrt{n}}}\rho_\mu\Big)\right|\nonumber  
\end{align}
and plugging in $s=0$ in \eqref{eq:CharFnOp} we have  $m\big(\Lambda^n_{\theta,0}\rho_\mu\big)=0$. Therefore, 
$$\Big|m\Big(\Lambda^n_{\theta,\frac{is}{\sigma \sqrt{n}}}\rho_\mu\Big)\Big| =\Big|m\Big(\Lambda^n_{\theta,\frac{is}{\sigma \sqrt{n}}}\rho_\mu\Big)-  m\big(\Lambda^n_{\theta, 0}(\rho_\mu)\big)\Big|\leq \left|\frac{is}{\sigma \sqrt{n}}\right| \sup_{ \gamma\in [0,1], \theta} \Big\|\big[\Lambda^n_{\theta,\frac{is\gamma}{\sigma \sqrt{n}}}\big]^{(1)}\rho_\mu\Big\|_{\wt B} \leq  \frac{C|s|}{\bar\sigma \sqrt{n}} \|\rho_\mu\|_{\cC} \bar\kappa^n.$$
where $\bar\kappa$ is as in Lemma \ref{WeakPert} and $C$ is independent of $\theta$.  

Note that, we need an expansion independent of $\theta$. So, we replace $Q_\theta$ with $Q_{0}\equalscolon Q$, and $\sigma_\theta$ with $\bar\sigma$ (where appropriate), and use $\sup_{\mu \in \Omega}\|\rho_\mu\|_{\cC} < \infty$, to conclude,

\begin{align*}
   &\left| \EXP_\mu\left(e^{is\frac{S_{\theta,n}(g)-nA_{\theta,\mu}}{\sigma_\theta \sqrt{n}}}\right) -e^{-s^2/2}\left(1+\frac{Q(s)}{\sqrt{n}}\right)\right| \nonumber\\
   &\phantom{a}\leq e^{-s^2/2}\left|\frac{Q_\theta(s)}{\sqrt{n}} -\frac{Q(s)}{\sqrt{n}} \right| + e^{-s^2/4} \frac{|s|^3}{\bar \sigma^3\sqrt{n}}\left|\psi_\theta\left(\frac{is}{\sigma_\theta\sqrt{n}}\right)\right|+ e^{-s^2/4} |(\log\lambda_\theta)^{(3)}(0)|^2\frac{|s|^6}{36\bar \sigma^6 n } \\&\phantom{aaa}+| m(\Pi^{(1)}_\theta\rho_\mu)| \sup_{\gamma\in [0,1]} |(\log \lambda_\theta)^{(3)}(i\gamma s)|\frac{|s|^4}{6\bar \sigma^4n} + \frac{s^2}{\bar \sigma^2n} \sup_{\gamma \in [0,1]}\left|m\Big(\Pi^{(2)}_{\theta,\frac{i\gamma s}{\sigma_\theta\sqrt{n}}}\rho_\mu\Big)\right| + \frac{C|s|}{\sqrt{n}} \|\rho_\mu\|_{\cC} \bar\kappa^n.
\end{align*}
Due to the continuity of $\lambda_\theta(is)$ and $\Pi_{\theta,is}$ and their derivatives with respect to $s$, at $(\theta,s)=(0,0)$,  the third and the fourth terms on the right are $\cO(n^{-1})$. Note that for the same reason, by choosing $\theta$ and $|s|$ small, 
$|\psi_\theta(s)|$ can be made small. So, the second term is $o(n^{-1/2})$ as $\theta \to 0$. Finally, 
\begin{align}\label{eq:1stPoly}
    \left|\frac{Q_\theta(s)}{\sqrt{n}} -\frac{Q(s)}{\sqrt{n}} \right| &\leq \left|(\log\lambda_\theta)^{(3)}(0)\sigma_0  -(\log\lambda_0)^{(3)}(0)\sigma_\theta\right|\frac{|s|^3}{6\bar \sigma^6 \sqrt{n}} \nonumber \\ &\phantom{aaaaaaaaaaaaaaaaa}+\left|m(\Pi_\theta^{(1)}\rho_\mu)\sigma_0-m(\Pi_0^{(1)}\rho_\mu)\sigma_\theta\right|\frac{|s|}{\bar \sigma^2 \sqrt{n}}.
\end{align}
Again, due to continuity of $(\log \lambda_\theta)^{(3)}$, $\sigma_\theta$ and $\Pi^{(1)}_\theta$ at $\theta = 0$, and the uniform boundedness of $\rho_\mu$, we have that the terms on the right are $o(n^{-1/2})$ as $\theta \to 0$. 
\end{proof}
\begin{rem}\label{rem:a&b}
From the proof, $$a=\frac{\log\lambda^{(3)}_0(0)}{(i\sigma_0)^3},\,\,b=\frac{m(\Pi^{(1)}_0\rho_\mu)}{i\sigma_0},\,\,\text{and}\,\, R(s)=C(s+s^6)$$ for $C>0$ sufficiently large. Hence, $R(|s|)/|s| = C(1+|s|^5)$ is bounded near $0$. This fact is used in the proof of \Cref{thm:CtsEdgeExp}. 
\end{rem}

\subsection{Edgeworth expansions for distribution functions}\label{sec:DistExp}

Now, we are ready to discuss the key theoretical result that ensures the asymptotic accuracy of the bootstrap for a large class of dynamical systems. We know from \cite{FernandoPene} that under the Assumptions \ref{asm:beta} and \ref{asm:gamma}(1), 
for each $\theta$, $\fN$ is the stable law of $S_{n,\theta}(h_\theta)$. That is, the following CLT holds: For any initial distribution $\mu$,
$$\frac{S_{\theta,n}(h_\theta)-nA_\theta}{\sigma_\theta\sqrt{n}}\implies Z  $$
where $Z$ is a standard normal random variable and $\implies$ denotes convergence in distribution. In the next theorem, a quantitative and a uniform version of this CLT is established using our spectral assumptions.

\begin{thm}\label{thm:CtsEdgeExp} Suppose that  the Assumptions \ref{asm:alpha}, \ref{asm:beta}, \ref{asm:gamma}, \ref{asm:delta}, and \ref{asm:ve} hold. Let $\fN(\cdot)$ and $\fn(\cdot)$ be as in \eqref{eq:standardGauss}.
 Then, there exists a quadratic polynomial $P$ such that the following asymptotic expansion holds:
\begin{equation}\label{eq:EdgeExp}
    \sup_{\mu \in \Omega, x\in \reals} \left|\Prob_{\mu}\left(\frac{S_{\theta,n}(h_\theta)-nA_\theta}{\sigma_\theta\sqrt{n}} \leq x\right) - \fN(x) - \frac{P(x)}{\sqrt{n}}\fn(x)\right| = o(n^{-1/2}),
\end{equation}
 as $\theta \to 0$ and $n\to \infty$.
\end{thm}

\begin{proof} Note that under the assumptions, we have the conclusions of \Cref{lem:CharExp}. Since we have an expansion for the characteristic functions, \eqref{eq:CharExpansion}, we could adapt the standard proof for iid sequences in \cite[Chapter XVI]{Feller}. 
Define $P$ to be the polynomial such that 
\begin{equation}\label{eq:P}
\cE_{n}(x)=\fN(x)+\frac{P(x)}{\sqrt{n}} \fn(x),
\end{equation}
where $\cE_{n}(s)$ is defined by
\begin{equation}\label{eq:cE}
    \widehat \cE_{n}(s) \coloneqq \int_{\mathbb R} e^{-isx}\, d\cE_{n}(x)\coloneqq e^{-s^2/2}+  \frac{e^{-s^2/2}}{\sqrt{n}}Q(s)
\end{equation}
where $Q$ is given by \eqref{EdgePolyFourier} with $\theta=0$. Such $P$ exists and can be written down explicitly as
\begin{equation}\label{eq:EdgePoly}
     P(x)=\frac{(\log\lambda_0)^{(3)}(0)}{6(i\sigma_0)^3 }(1-x^2)+\frac{m(\Pi^{(1)}_0\rho_\mu)}{i\sigma_0}.
\end{equation}

Next, from the Berry-Ess\'een inequality (see \cite[Chapter XVI.2]{Feller}), for each $T>0$, 
\begin{equation}\label{BerryEsseen1}
\left|\Prob_{\mu}\left(\frac{S_{\theta,n}(h_\theta)-nA_\theta}{\sigma_\theta\sqrt{n}} \leq x\right)-\cE_{n}(x)\right|\le \frac 1\pi\int_{-T}^{T}\bigg|\frac{\EXP_{\mu}\big(
e^{is\frac{S_{n,\theta}(h_\theta)-nA_{\theta}}{\sigma_\theta \sqrt{n}}}\big)-\widehat \cE_{n}(s)}{s}\bigg|\, ds+\frac{C_0}{T},
\end{equation}
where and $C_0$ is independent of $T$ and $\theta$. See \cite[Chapter XVI.3,4]{Feller} for a detailed discussion of the Berry-Ess\'{e}en inequality.

Now, we estimate the right hand side of \eqref{BerryEsseen1} for an appropriate choice of $T$. For $\ve >0$, choose $B> \max \{C_0\eps^{-1},\bar\sigma \delta\}$. Then 
\begin{align}\label{BEOrdr}
\frac{1}{\pi}\int_{-B\sqrt{n}}^{B\sqrt{n}}\bigg|\frac{\EXP_{\mu}\big(
e^{is\frac{S_{n,\theta}(h_\theta)-nA_{\theta}}{\sigma_\theta \sqrt{n}}}\big)-\widehat \cE_{n}(s)}{s}\bigg|\, ds+\frac{C_0}{B\sqrt{n}} 
 \leq I_1 + I_2+ I_3 + \frac{\ve}{\sqrt{n}} 
\, 
\end{align}
where
\begin{align*}
I_1 &=\frac 1\pi\int_{|s|<\delta\bar \sigma \sqrt{n}}\bigg|\frac{\EXP_{\mu}\big(
e^{is\frac{S_{n,\theta}(h_\theta)-nA_{\theta}}{\sigma_\theta \sqrt{n}}}\big)-\widehat \cE_{n}(s)}{s}\bigg|\, ds\, ,  \\
I_2 &=\frac 1\pi\int_{\delta\sqrt{n}\leq |s|\le B\sqrt{n}}\bigg|\frac{\EXP_{\mu}\big(
e^{is\frac{S_{n,\theta}(h_\theta)-nA_{\theta}}{\sigma_\theta \sqrt{n}}}\big)}{s}\bigg|\, ds\, , \\
I_3 &=\frac 1\pi\int_{|s|\geq \delta\sqrt{n}} \bigg|\frac{\widehat{\cE}_{n}(s)}{s}\bigg|\, ds.
\end{align*}

From \eqref{eq:CharExpansion} and \Cref{rem:a&b}, we have that 
$$I_1= \frac 1\pi\int_{|s|<\delta\bar \sigma \sqrt{n}}\frac{R(|s|)}{|s|} \left[e^{-s^2/4}\cdot o\Big(\frac {1}{\sqrt{n}}\Big)+ \cO(\bar\kappa^n)\right] \, ds = o(n^{-1/2})$$
where the implied constants are independent of $\theta$ and $\mu$.
Note that there exists a constant $C$ such that $|\widehat{\cE}_{n}(s)|\leq C e^{-s^2/4}$ for all $s\in \reals$. 
Therefore,
$$I_3 \leq \frac C\pi\int_{|s|\geq \delta\sqrt{n}} \frac{e^{-s^2/4}}{|s|}\, ds = \cO(e^{-c'n})\, ,$$
for some $c'>0$. Because our choice of $\ve >0$ is arbitrary, if $I_2=o(n^{-1/2})$, then the proof is complete. To show this, we change the variables in $I_2$ to obtain
\begin{align*}
I_2 &=\frac{1}{\pi}\int_{\delta \sigma^{-1}_\theta <|s|<B\sigma^{-1}_\theta}\bigg|\frac{\EXP_{\mu}\big(
e^{is(S_{n,\theta}(h_\theta)-nA_{\theta})}\big)}{s}\bigg|\, ds\, \leq\frac{1}{\pi} \int_{\delta\tilde \sigma^{-1} <|s|<B\bar \sigma^{-1}} \frac{|m(\cL^n_{\theta, is}(\rho_\mu))|}{|s|}\, ds.
\end{align*}
Note that $(\theta,s) \mapsto \cL_{\theta,is} \in \cL(\cB,\wt \cB)$ is continuous 
and hence, so is $\cL^n_{\theta,is}$. Since the spectral radius of $\cL_{\theta,is}$ is $\lim_{n\to \infty} \|\cL^n_{\theta,is}\|^{1/n}_{\cB,\wt \cB}$, the \Cref{asm:gamma}(4) 
gives us the existence of $n_{\theta,s}$ such that  $$\|\cL^{n_{\theta,s}}_{\theta,is}\|_{\cB,\wt \cB}^{1/n_{\theta,s}}<1.$$ By the continuity of $(\theta,s) \mapsto \|\cL_{\theta,is}\|_{\cB,\wt \cB}$, and the compactness of 
$[-\delta,\delta] \times [\delta\tilde \sigma^{-1}, B\bar \sigma^{-1}]$ there exists a $\eta \in (0,1)$ and $n_0$ such that 
$$\|\cL^{n}_{\theta,is}\|_{\cB,\wt \cB}<\eta^n$$ for all $n\geq n_0$ and for all $(\theta, s)\in [-\delta,\delta]\times [\delta\tilde \sigma^{-1}, B\bar \sigma^{-1}]$. 
As a result, 
\begin{align*}
I_2 \leq \frac{\eta^n}{\pi}\|\rho_\mu\|_{\cC} \int_{\delta\tilde \sigma^{-1} <|s|<B\bar \sigma^{-1}} \frac{1}{|s|}\, ds = \cO(\eta^n). 
\end{align*}
This gives the required asymptotics. 
\end{proof}

\begin{rem}\label{rem:EdgePoly}
We can express the coefficients of the polynomial $P(x)$ using the asymptotic moments of $S_{0,n}(h_0)$. Assume for simplicity that $A_0=0$. Then, from the discussion in \Cref{sec:Moments}, we have that 
\begin{align*}
    m(\Pi^{(1)}_0 \rho_\mu) = \lim_{n \to \infty} i\EXP_\mu(S_{n,0}(h_0)) \,\,\text{and}\,\, (\log \lambda_0)^{(3)}(0) &= \lambda^{(3)}_0(0) = i^3 M_{\nu_0,0}. 
\end{align*}
So, the first-order continuous Edgeworth expansion takes the form:
$$\cE_n(x) = \fN(x) + \frac{1}{\sqrt{n}}\left[\frac{M_{\nu_0,0}}{6\sigma^3_0}(1-x^2)+\frac{\wt M_{\mu,0}}{\sigma_0}\right]\fn(x)$$
where 
\begin{align*}
    \wt M_{\mu,0} = \lim_{n\to \infty} \EXP_\mu(S_{n,0}(h_0)),\,\,\,\sigma^{2}_0 = \lim_{n\to \infty} \EXP_\mu\left(\frac{S_{n,0}(h_0)}{\sqrt{n}}\right)^2 \,\,\text{and}\,\, M_{\nu_0,0} = \lim_{n\to \infty} \EXP_{\nu_0}\left(\frac{S_{n,0}(h_0)}{n^{1/3}}\right)^3.
\end{align*}
In general, the same holds with $h_0$ replaced by $\wt h_0 = h_0 - A_0$.
\end{rem}
\begin{rem}\label{rem:Symmetry}
If $\mu=\nu_0$, then $\wt M_{\mu,0} = 0$. So, the first-order Edgeworth expansion captures the deviation from equilibrium. In addition, it captures the asymmetry of the stationary distribution because $M_{\nu_0,0}\sigma^{-3}_0$ corresponds to the \textit{skewness} of $\nu_0$. As a result, if $\mu=\nu_0$ and $\nu_0$ is symmetric, then the Gaussian approximation is as good as the bootstrap. We will observe this in our simulations.
\end{rem}
\begin{rem}\label{rem:IntialMeasure}
In practice, we may replace the initial measures $\mu \in \Omega$, by a family of initial measures $\{\mu_\theta\}$, i.e., the quantity of interest will be
$$\Prob_{\mu_\theta}(S_{n,\theta}(h_\theta) \leq x\cdot \sigma_\theta\sqrt{n}),$$
and $\theta$ may depend on $n$.
In this case, we let
\begin{equation*}
     P(x)=\frac{(\log\lambda_0)^{(3)}(0)}{6(i\sigma_0)^3 }(1-x^2)+\frac{m(\Pi^{(1)}_0\rho_{\mu_0})}{i\sigma_0}
\end{equation*}
so that the expansion does not depend on $\{\mu_\theta\}$, and in turn, on $\theta$. Due to \eqref{eq:thetaCharExp}, \eqref{eq:1stPoly} and the proof of the \Cref{thm:CtsEdgeExp}, the first order continuous Edgeworth expansion holds if we assume that there exists a space $\cC$ in \eqref{SpaceChain} such that $\cC \hookrightarrow \cX_3 \hookrightarrow \cX^{(+)}_3$ $$\lim_{\theta \to 0}\|\rho_{\mu_\theta} - \rho_{\mu_0}\|_{\cC} =0.$$
\end{rem}

\section{Examples}\label{sec:Exmp}
Here, we discuss some examples for which our abstract results apply. In  Sections \ref{sec:SmoothExp} and \ref{sec:PiecewiseExp}, we describe a class of dynamical systems and verify our abstract Assumptions \ref{asm:alpha}, \ref{asm:beta} and \ref{asm:gamma} for that class of systems. In \Cref{sec:simulationresults}, we have included simulation results corresponding to particular examples of these classes of dynamical systems. Moreover, we verify our assumptions for $V-$geometrically ergodic Markov chains in \Cref{sec:Markov} for unbounded observables. This example shows that our techniques are neither limited to dynamical systems nor to bounded observables. 

For a summary of standard techniques to verify our assumptions for dynamical systems, we refer the reader to \cite[Section 4.2]{FernandoPene}. Also, an elementary discussion related to the conditions in the \Cref{asm:alpha} can be found in \cite[Section 3.2]{LiveraniInvariant}.

\subsection{$C^1$-perturbations of smooth expanding maps}\label{sec:SmoothExp}
This is the class of dynamical systems that we mentioned as an example at the beginning. We make a few extra assumptions to guarantee that all our spectral assumptions hold. 

Let $g \in C^2(\bbT, \bbT)$ be such that there exists $\eta>2$ such that $|g'| \geq \eta$ (in fact, we may assume $\eta  >1$ and consider an iterate of $g$, $g^k$, with $|(g^k)'| >\eta$). Let $g_\theta,\,\, \theta \in [0,1]$ with $g_0 \colonequals g$ be 
such that for all $\theta, \bar\theta \in [0,1]$, $$d_{C^1}(g_\theta,g_{\bar \theta})\leq |\theta-\bar \theta|\,\,\,\text{and}\,\,\sup_{\theta} d_{C^2}(g_\theta, g_{\bar \theta}) \leq M_1$$ for some fixed $M_{1}$. Also, let $m$ be the Lebesgue measure on $\bbT$. Then, taking $\cB \colonequals \BV(\bbT)$ to be the space of functions of bounded variation on $\bbT$ with the $\BV$ norm $$\|\cdot \|_{\BV}= \Var[\cdot]+\|\cdot\|_{L^1}$$ where $\Var[\cdot]$ refers to the total variation of on $\bbT$ 
and $\wt \cB \colonequals L^1$, we have $\cB \subset \wt \cB$ and $\|\cdot\|_{L^1} \leq \|\cdot\|_{\BV}$. Note also that $\BV(\bbT)$ is a Banach algebra.

Let $h_\theta\in \cB 
$ be such that $$\lim_{\theta \to 0}\|h_\theta - h_0\|_{\BV}=0.$$ 
We assume that, for small $\theta$, $h_\theta$ is \textit{non-arithmetic}, i.e.,
\begin{equation}\label{eq:nonarithmetic}
    \text{$h_\theta$ is not $g_\theta-$cohomologous to a lattice-valued function in $\BV(\bbT)$}.
\end{equation}
That is, there do not exist $a, b \in \reals$  and $\eta \in \BV$ such that $h_\theta - (\eta \circ g_\theta - \eta) \in a+b\integers$, $m-$almost surely. In particular, this implies, \eqref{eq:CoBd}. Also, \eqref{eq:nonarithmetic} is reminiscent of the theorem due to Ess\'een that, in the iid case, the first order Edgeworth expansion holds if and only if the common distribution is non-lattice; see \cite{Ess}. Since $\BV(\bbT)$ is a Banach algebra, multiplication by $h_\theta$ is a bounded linear operator on $\BV$ and the convergence of $h_\theta$ to $h$ gives $\|H_\theta-H_\theta\|_{\cB,\cB} \to 0$ as $\theta\to 0$. 

Note that for each $\theta$, $s \mapsto \cL_{\theta, is}$ is analytic from $\reals$ to $\cL(\cB, \cB)$. This follows from
$$\cL_{\theta,is}(\cdot)=\cL_{\theta}(e^{ish_{\theta}} \cdot)=\sum_{k=0}^{\infty} \frac{(is)^k}{k!} \cL_{\theta}((h_\theta)^k\,\cdot )$$ where $(h_\theta)^k = h_\theta \times \dots \times h_\theta$, $k-$times. 
Note that $\cL_\theta ((h_\theta)^k\, \cdot) = \cL_\theta \circ H^k_\theta$. So, $\|\cL_\theta ((h_\theta)^k\, \cdot)\|\leq \|\cL_\theta\|\|H_\theta\|^k$ where $\|\cdot\|$ is the operator norm induced by the $\BV$ norm. Therefore, the series expansion converges absolutely. Since $\cL(\cB,\cB)$ is complete, this means that the series converges in $\cL(\cB,\cB)$, and it follows that, for each $\theta$, $s \mapsto \cL_{\theta,is}$ is an analytic family in $\cL(\cB,\cB)$. 
Since $\|\cL_{\theta,is} \ve\|_{L^1} \leq \|\ve\|_{L^1}$, the operators are in $\cL(\wt\cB, \wt\cB)$ as well.

We follow \cite[Example 3.1]{LiveraniInvariant}. From the proof there, it follows that for all $\psi \in \cB$,
$$\|\cL_{\theta}(\psi)-\cL_{\bar \theta}(\psi)\|_{L^1} \lesssim \|\psi\|_{\BV}|\theta -\bar \theta| $$ where the implied constant depends only on $M_1$ and $\eta$.  Next, for all $\vp \in C^1$ and $\psi \in \cB$, 
\begin{align*}
    \int (\cL_{\theta, is}\psi - \cL_{\bar\theta, is}\psi) \cdot \vp {\, dm } &= \int (\cL_{\theta}(e^{ish_\theta}\psi) - \cL_0 (e^{ish_{\bar\theta}}\psi)) \cdot \vp {\, dm }\\ &= \int e^{ish_\theta} \psi\cdot  (\vp \circ g_\theta - \vp \circ g_{\bar\theta}) {\, dm }+ \int \psi\cdot \vp \circ g_{\bar\theta}\cdot (e^{ish_\theta}-e^{is h_{\bar\theta}}) {\, dm }.
\end{align*}
The argument in \cite[Example 3.1]{LiveraniInvariant} can be applied to estimate the first term, i.e., 
$$\left|\int e^{ish_\theta} \psi \cdot (\vp \circ g_\theta - \vp \circ g_{\bar\theta}){\, dm } \right| \lesssim  \|e^{is h_\theta}\psi\|_{\BV}\|\vp\|_\infty |\theta-\bar\theta| \leq \|\vp\|_\infty\|\psi\|_{\BV}\|e^{ish_\theta}\|_{\BV} |\theta-\bar\theta| $$
with a constant independent of $\theta, \bar\theta$ and $s$. Since $|e^{ish_\theta(x)}-e^{ish_\theta(y)}| \leq |s||h_\theta(x)-h_\theta(y)|$,
we have $\Var[e^{ish_\theta}]\leq |s|\Var[h_\theta]$.
Since $h_\theta \to h$ in $\BV$, we have that $\Var[e^{ish_\theta}] \lesssim |s|$ where the implied constant is independent of $\theta$. 
Therefore, 
$$\left|\int e^{ish_\theta} \psi \cdot (\vp \circ g_\theta - \vp \circ g_{\bar\theta}){\, dm } \right| \lesssim   \|\vp\|_\infty\|\psi\|_{\BV}(|s|+1)\cdot |\theta - \bar\theta|.$$
For the second term, note that, 
\begin{align*}
    \left|\int \psi\cdot \vp \circ g_{\bar\theta} \cdot (e^{ish_\theta}-e^{is h_{\bar\theta}}){\, dm } \right| \leq \|{\vp}\|_{\infty} |s|\int|\psi||h_\theta-h_{\bar \theta}|\, dm 
    &\leq\|{\vp}\|_{\infty} |s| \|\psi(h_\theta-h_{\bar \theta})\|_{\BV} \\
    &\leq \|{\vp}\|_{\infty} \|\psi\|_{\BV} |s| \|h_\theta-h_{\bar \theta}\|_{\BV}.
\end{align*}
In fact, the estimates hold for $\vp \in L^\infty$ because $C^1(\bbT, \bbT)$ is dense in $L^\infty$. Next, taking the supremum over $\vp$ with $\|\vp\|_\infty \leq 1$, we have, 
$$\|\cL_{\theta, is}\psi - \cL_{\bar\theta, is}\psi\|_{L^1} \lesssim (|s|+1) \|\psi\|_{\BV} \left(|\theta-\bar\theta| +\|h_\theta - h_{\bar \theta}\|_{\BV}\right).$$
Therefore, 
\begin{align*}
    \|\cL_{\theta,is}\psi - \cL_{\bar \theta, i \bar s} \psi\|_{L^1} &\leq \|\cL_{\theta,is}\psi - \cL_{\bar \theta, i s} \psi\|_{L^1} + \|\cL_{\bar \theta, i s}\psi - \cL_{\bar \theta, i \bar s} \psi\|_{L^1} \\ &\lesssim (|s|+1) \|\psi\|_{\BV} \left(|\theta-\bar\theta| +\|h_\theta - h_{\bar \theta}\|_{\BV}\right) +\|(\cL_{\bar \theta, i s}- \cL_{\bar \theta, i \bar s})\psi\|_{\BV}   \\ 
    &\leq (|s|+1) \|\psi\|_{\BV} \left(|\theta-\bar\theta| +\|h_\theta - h_{\bar \theta}\|_{\BV}\right) +\|\cL_{\bar \theta, i s}- \cL_{\bar \theta, i \bar s} \|_{\BV}\|\psi\|_{\BV} \\ 
    &\leq\|\psi\|_{\BV}  \left(  \left(|\theta-\bar\theta| +\|h_\theta - h_{\bar \theta}\|_{\BV}\right) (|s|+1) +\|\cL_{\bar \theta, i s}- \cL_{\bar \theta, i \bar s} \|_{\BV}\right),
\end{align*}
and hence, 
$$\lim_{(\theta,s)\to (\bar\theta, \bar s)}\|\cL_{\theta,is} - \cL_{\bar \theta, i \bar s} \|_{\BV,L^1} = 0.$$

Since $g \in C^2$ 
on $\bbT$ and $|g'|>\eta$ and $g_\theta$ is $C^1$ close to $g$, for sufficiently small $\theta$, $|g'_\theta|>\eta$, and due to \cite[Lemma 1]{ButterleyEslami}, the following Doeblin-Fortet inequality,
$$\forall \psi \in \cB,\,\,\, \|\cL_{\theta,is}\psi\|_{\BV} \leq 2\gamma^{-1}\|\psi\|_{\BV}+C_{\gamma,\theta}(1+|s|)\|\psi\|_{L^1}\,,$$
(where $C_{\gamma,\theta}$ is uniformly bounded in $\theta$) holds. Iterating this $n$-times, we obtain \eqref{eq:UniformDF0}. Note that the non-arithmeticity of $h_\theta$ is essential to apply \cite[Lemma 1]{ButterleyEslami}.  

So far, we have checked \Cref{asm:alpha} with  $\cB=\BV(\bbT)$ and $\wt \cB= L^1(m)$, \Cref{asm:beta} with $\cX_0=\cX^{(+)}_3=\BV(\bbT)$ and any $p_0 \geq 1$, and finally, we verify \Cref{asm:gamma}. 

From \cite[Theorem II.5]{HennionHerve}, we have that the essential spectral radius of $\cL_{\theta, is}$ is at most $\lambda^{-1}$ and $\cL_{\theta, is}$ is quasi compact. In particular, $\cL_\theta$ has finitely many eigenvalues $\lambda_i$ such that $\lambda < |\lambda_i| \leq 1$. Further, due to the exactness of the transformation $f$, $1$ is the only eigenvalue of $\cL_0$ on the unit circle and it is simple. This is also true for $\cL_\theta$ because $g_\theta$ is uniformly expanding. 
Note that \Cref{asm:gamma}(3) is equivalent to \Cref{asm:gamma}(1) because $\cX_0=\cX^{(+)}_3=\cB$. \Cref{asm:gamma}(4) follows directly from \cite[Lemma 4.5]{FernandoPene} due to the non-arithmeticity assumption \eqref{eq:nonarithmetic}.

\begin{rem}
Note that, in \ref{asm:gamma}(1), we do not require simplicity of the eigenvalue $1$ of $\cL_\theta, \theta \neq 0$ (it is part of the conclusion of our results), but, in this special case, we have it due to structural stability of expanding maps.
\end{rem}

\subsubsection{Periodic Cubic Spline Approximations}\label{sec:CubicPiecewise} As a special case, we discuss the cubic spline approximations of $g$ (considering $g$ as a periodic function on $\reals$). This will be relevant to our simulations. Let $0 = z_0 < z_1 < \dots < z_n < z_{n+1} = 1$ (we work on [0,1] and identify $0$ and $1$, so that $z_0$ and $z_{n+1}$ overlap) with $g(z_i)=y_i$ for $i=1,\dots,n$. Let $\mathfrak{s}:[0,1]\to [0,1]$ be the \textit{periodic} cubic spline approximation of $g$ with mesh $\{z_i\}$. Then, 
\begin{itemize}[leftmargin=25pt]
    \item $\mathfrak{s}(x)=a_i+b_ix+c_ix^2+d_ix^3$ on $[z_i,z_{i+1}]$ for $i=0,\dots,n$,
    \item For $j=1,2$, $\mathfrak{s}^{(j)}(z_i-)=\mathfrak{s}^{(j)}(z_{i}+)$ when $i=1,\dots,n$ and $\mathfrak{s}^{(j)}(0+)=\mathfrak{s}^{(j)}(1-)$,
    \item $\mathfrak{s}(z_i)=y_i$ and $\mathfrak{s}(0)=\mathfrak{s}(1)$.
\end{itemize}
Then $\|g^{(j)}-\mathfrak{s}^{(j)}\|_{L^\infty} = \cO(\theta^{2-j})$  $j=0,1,2$ where  $\theta= \max\{z_{i+1}-z_{i}\,|\, 0 \leq i \leq n\}$ is the mesh-size; see \cite[Theorem 2.3.2]{Spline}. Therefore, writing $\mathfrak{s}_{\theta}$ for a spline approximation of $g$ with mesh-size $\theta \in (0,1/N)$ where $N$ is large, we obtain a family of $C^1-$perturbations of the $C^2$ expanding map $g$. 
 
\subsection{Linear spline approximations of piecewise expanding maps}\label{sec:PiecewiseExp}

Let $g:[0,1]\to [0,1]$ be a piecewise expanding map as in \cite{Liverani}. That is,
\begin{itemize}[leftmargin=25pt]
    \item $g$ is piecewise $C^2$ on $[0,1]$: There is a mesh $0 = z_0 < z_1 < \dots < z_n < z_{n+1} = 1$ such that  $g|_{(z_i,z_{i+1})}$ extends to a $C^2$ function on $[z_i,z_{i+1}]$  for all $i=0,\dots,n$.  
    \item $g$ is uniformly expanding: There exists $\eta>1$ such that $|g'|\geq \eta$.
    \item $g$ is a covering:  Let $$\cA_n = \bigvee_{k=0}^n g^{-k}(\{[z_{i},z_{i+1}], i=0,\dots,n\}).$$
    Then there exists $N_n$ such that for all $I \in \cA_n$, $g^{N_n}(I) = [0,1]$. 
\end{itemize}

In this case, as in the previous example, we can let let $m$ be the Lebesgue measure on $[0,1]$, $\cB\colonequals\BV[0,1]$ 
to be the space of functions of bounded variation on $[0,1]$ with the $\BV$ norm $$\|\cdot \|_{\BV}= \Var[\cdot]+\|\cdot\|_{L^1}$$ where $\Var[\cdot]$ refers to the total variation of a function on $[0,1]$,  
and $\wt \cB \colonequals L^1(m)$. Then, we have $\cB \subset \wt \cB$ and $\|\cdot\|_{L^1} \leq \|\cdot\|_{\cB}$, and that $\BV[0,1]$ is a Banach algebra as in the previous example.

Consider $\mathfrak{L}:[0,1]\to[0,1]$ where $\mathfrak{L}|_{[z_i,z_{i+1}]}=\mathfrak{s}_i:[z_i,z_{i+1}] \to [0,1]$ is the linear spline approximation of $g$ on $[z_i,z_{i+1}]$ with mesh $\{w_{j,i}\}$. That is, 
$$z_i = w_{0,i} < w_{1,i} < \dots < w_{m,i} < w_{m+1,i}=z_{i+1}$$ and
$$\mathfrak{s}_i(x) = \frac{w_{j+1,i}-x}{w_{j+1,i}-w_{j,i}}g(w^{+}_{j,i})+\frac{x-w_{j,i}}{w_{j+1,i}-w_{j,i}}g(w^{-}_{j+1,i}),\,\,\, x \in [w_{j,i},w_{j+1,i}] \equalscolon P_{j,i}.$$
{Note that is possible that $w_{j,i} = z_i$ or $z_{i+1}$ and $g$ has a discontinuity at $w_{j,i}$s. So, we have distinguished between the left (\,-\,) and right (+) values at $w_{j,i}$s. In what follows, this distinction is understood even when it is not explicitly stated.}

Since $\mathfrak L$ it is piecewise linear, it is piecewise $C^2$ on $[0,1]$ with mesh $\cup_{i} \{w_{j,i}\}$. Also,
$$\mathfrak s'_i(x)= \frac{g(w_{j+1,i}) - g(w_{j,i})}{w_{j+1,i}-w_{j,i}} = g'(\xi_{j,i})$$
for some $\xi_{j,i} \in [w_{j,i},w_{j+1,i}]$ due to the mean value theorem. Therefore, $|\mathfrak{L}'|\geq \eta > 1$, i.e., $\mathfrak L$ is uniformly expanding.
So, taking $\theta = \max\{w_{j+1,i}-w_{j,i}| i,j\}$, and writing $\mathfrak{L}_{\theta}$ for a linear spline approximation of $g$ with mesh-size $\theta \in (0,1/N)$ where $N$ is large, we obtain a family of piecewise expanding dynamical systems. 

We choose $h_\theta$, as in the previous example.  Let $h_\theta\in \BV[0,1]$ be such that $\|h_\theta - h_0\|_{\BV} \to 0$ as $\theta\to 0$. We assume that $h_\theta$ for small $\theta$ are non-arithmetic so that \eqref{eq:CoBd} is true. We also have $H_\theta  \in \cL(\BV[0,1],\BV[0,1])$, because $\BV[0,1]$ is a Banach algebra, and  $\|H_\theta-H_0\|_{\BV,\BV} \to 0$ as $\theta\to 0$ due to the strong convergence of $h_\theta$. Also, as before, for each $\theta$, $s \mapsto \cL_{\theta, is}$ is analytic from $\reals$ to $\cL(\BV[0,1],\BV[0,1])$.

To show that $(\theta,s)\mapsto \cL_{\theta,is}$ is continuous from $(0,1/N)\times \reals \to \cL(\cB,\wt\cB)$ we use \cite[Lemma 13]{Keller1982}. First, we introduce a distance $\mathfrak d$ on the class of piecewise expanding maps on $[0,1]$:
\begin{multline*}
    \mathfrak{d}(g,\wt g) = \inf \bigg\{\eps>0\, \Big|\, B \subseteq [0,1],\, d:[0,1] \to [0,1]\, \text{is a diffeomorphism}, \\ g|_{A}=\wt g\circ d|_A,\, m(B)>1-\eps,\, |d(x)-x|<\eps,\, \Big|\frac{1}{d'(x)}-1\Big|<\eps\bigg\}
\end{multline*}
Then, 
\begin{equation}\label{eq:SkohorodEst}
   \| \cL - \wt\cL\|_{\cB,\wt \cB} \leq 12 \cdot \mathfrak d (g,\wt g).
\end{equation}
where $\cL$ and $\wt\cL$ are the transfer operators corresponding to $g$ and $\wt g$, respectively. 

We claim that $\mathfrak d(g,\mathfrak L_\theta) \to 0$ as the mesh size $\theta \to 0$. To see this, it is sufficient to construct a set $B_\theta \subset[0,1]$ and a diffeomorphism $d_\theta:[0,1]\to [0,1]$ such that  
$\mathfrak L_\theta|_{B_\theta}= g \circ d_\theta|_{B_\theta}$ and $m(B_\theta) \to 1$ and $d_\theta \to \text{id}$ (in $C^1$) as $\theta \to 0$. 

First, we restrict our attention to $P_{j,i}$s, and solve for $d$ by solving $g|_{B}=\mathfrak L \circ d|_B$
\begin{align*}
    g(x)=(\mathfrak L \circ d)(x)= \mathfrak{s}_i(d(x)) = \frac{w_{j+1,i}-d(x)}{w_{j+1,i}-w_{j,i}}g(w_{j,i})+\frac{d(x)-w_{j,i}}{w_{j+1,i}-w_{j,i}}g(w_{j+1,i})
\end{align*}
which gives us,
\begin{align*}
    d(x)= \frac{g(w_{j+1,i})-g(x)}{g(w_{j+1,i})-g(w_{j,i})}w_{j,i}+\frac{g(x)-g(w_{j,i})}{g(w_{j+1,i})-g(w_{j,i})}w_{j+1,i}
\end{align*}
which is well-defined because $g$ is strictly monotonic on $P_{j,i}$ (Note that, since $|g'|>1$ and $g'$ is continuous on each $[z_i,z_{i+1}]$, $g'$ should remain either strictly negative or strictly positive on $[z_i,z_{i+1}]$, and hence, the strict monotonicity). Also,
\begin{align*}
    d^\prime(x) = g'(x) \frac{w_{j+1,i}-w_{j,i}}{g(w_{j+1,i})-g(w_{j,i})} = \frac{g'(x)}{g'(\zeta_{j,i})}
\end{align*}
for some $\zeta_{j,i} \in \mathring{P}_{j.i}$ ($\colonequals$ the interior of $P_{j.i}$), 
and hence, 
\begin{align*}
    \left|\frac{1}{d^\prime(x)} - 1\right| =\left| \frac{g'(\zeta_{j,i})}{g'(x)} - 1\right| = \frac{|g'(\zeta_{j,i})-g'(x)|}{|g'(x)|} \leq \eta^{-1}\|g''\|_\infty |\zeta_{j,i}-x| \leq  \eta^{-1}\|g''\|_\infty \theta 
\end{align*}
since $g\in C^2(P_{j,i})$. Without loss of generality, assume that $\theta$ is small enough so that $\eta^{-1}\|g''\|_\infty \theta< c \sqrt{\theta}$ for suitable $c>0$ (chosen later). Also, because $g'$ has the same sign in $\mathring{P}_{j.i}$, $d^{\prime}>0$ on $\mathring{P}_{j.i}$. So, $d$ is strictly increasing on $\mathring{P}_{j.i}$. 
Also, it is easy to see that 
\begin{align*}
   \Big| d(x) - x \Big| \leq \frac{g(w_{j+1,i})-g(x)}{g(w_{j+1,i})-g(w_{j,i})} (x-w_{j,i})+\frac{g(x)-g(w_{j,i})}{g(w_{j+1,i})-g(w_{j,i})}(w_{j+1,i}-x) \leq w_{j+1,i} - w_{j,i} \leq \theta
\end{align*}

Next, let $\wt \theta = \min\{w_{j+1,i}-w_{j,i}| i,j\}$. To construct $B_\theta$, from each $[z_i,z_{i+1}]$ delete intervals of length $\wt \theta^2$ centered at each $w_{j,i} \neq z_i, z_{i+1}$ and at $z_i$ and $z_{i+1}$ remove intervals of length $\wt \theta^2/2$ with $z_i$ or $z_{i+1}$ as an endpoint and lying inside $[z_i,z_{i+1}]$. Call the remaining subset of $[0,1]$ as $B_\theta$. Since we remove $\cO(\wt \theta^{-1})$ number of such intervals, the total length of the intervals deleted is of order $\wt \theta^2 \times \wt \theta^{-1} = \wt \theta \leq \theta$ so that $m(B_\theta)> 1 - \cO(\theta)$. We assume, without loss of generality, that $\theta$ is small enough so that $m(B_\theta) > 1- c \cdot  \sqrt{\theta}$. 

Finally, we extend $d$ from $[0,1]\setminus B_\theta$ to $[0,1]$ such that $d$ is a diffeomorphism such that {for all $x\in [0,1]$,}
\begin{equation}\label{eq:DiffeoBds}
    \Big| d(x) - x \Big| <  c\sqrt{\theta},\,\,\,\text{and}\,\,\, \left|\frac{1}{d^\prime(x)} - 1\right| < c\sqrt{\theta}.
\end{equation}
Since $d^\prime>0$ on $\cup \mathring{P}_{j.i}$, it is enough to define $d$ on $B_\theta$ 
so that it is differentiable and $d^\prime>0$ on $[0,1]$ (in order to ensure that $d$ has a differentiable inverse). It is easy to see that a continuous and positive extension of $d'$ yields a unique $d$ that is strictly increasing and differentiable. In fact, we extend $d^\prime$ to $[\wt \zeta, \wh \zeta]$ where $\wt \zeta =w_{j,i}-\wt \theta^2/2$ and $\wh \zeta = w_{j,i}+\wt \theta^2/2$ are the end points of the removed interval centred at $w_{j,i}$ as a continuous (positive) curve joining $d^\prime(\wt \zeta)$ and $d^\prime(\wh \zeta)$. The only extra condition this extension of $d'$ should satisfy is 
\begin{equation}\label{eq:AreaCnd}
    \int_{\wt\zeta}^{\wh \zeta} d^\prime (x) \, dx =  d(\wh \zeta) - d(\wt \zeta).
\end{equation}

To show that this is possible, we compute 
$d(\wh \zeta ) - d(\wt \zeta )$. To make the notation simpler we denote the three consecutive mesh points in the expressions for $d(\wh \zeta )$ and $d(\wt \zeta )$ by $w_1, w_2(=w_{j,i})$ and $w_3$ with  $w_1< w_2 < w_3$. 

First, note that 
$d(\wt \zeta) = w_2 - \alpha (w_2-w_1)$ and $d(\wh \zeta_{j,i}) = w_2+\beta(w_3-w_2)$ where
$$\alpha=\frac{g( w_2^-)-g(\wt \zeta)}{g(w_2^-)-g(w_1)}=\frac{g( w_2^-)-g(\wt \zeta)}{g'(\wt \eta)(w_2-w_1)},\,\,\, \beta = \frac{g(\wh \zeta )-g(w_2^+)}{g(w_3)-g(w_2^+)}=\frac{g(\wh \zeta )-g(w_2^+)}{g'(\wh \eta)(w_3-w_2)}$$
for some $\wt\eta \in (w_1,w_2) $ and $\wh \eta \in (w_2,w_3)$.
Then
\begin{align*}
    0< d(\wh \zeta)-d(\wt \zeta) &= \beta(w_3-w_2)+ \alpha (w_2-w_1)\\ &= \frac{g( w_2^-)-g(\wt \zeta)}{g'(\wt \eta)}+ \frac{g(\wh \zeta )-g(w_2+)}{g'(\wh \eta)}= \left(\frac{g'(\wt \xi\,)}{g'(\wt \eta)}+ \frac{g'(\wh \xi\,)}{g'(\wh \eta)}\right)\frac{\wt\theta^2}{2} 
\end{align*}
for some $\wt \xi \in (\wt \zeta, w_2)$, $\wh \xi \in (w_2,\wh \zeta)$. Note that
$$\left|\frac{g'(\wt \xi\,)}{g'(\wt \eta)} - 1 \right|= \frac{|g'(\wt \xi\,)-g'(\wt \eta)|}{|g'(\wt \eta)|} = \frac{|g''(\wt \chi)|}{|g'(\wt \eta)|} |\wt \xi-\wt \eta|$$
for some $\wt \chi$ between $\wt \xi$ and $\wt \eta$ (and similarly, for $ \sim $ replaced with $ \wedge$) because $g$ extends to a function in $C^2[z_i,z_{i+1}]$. So, there exists $\wt c>1$ (independent of $z_i$s) such that 
$$ 1 - \wt c \cdot \theta  < \frac{1}{2}\left(\frac{g'(\wt \xi\,)}{g'(\wt \eta)}+ \frac{g'(\wh \xi\,)}{g'(\wh \eta)}\right) < 1+ \wt c \cdot \theta.$$
Therefore,
$$  1 - \wt c \cdot \theta < \frac{d(\wh \zeta) - d(\wt \zeta)}{\wh \zeta - \wt \zeta} <  1 + \wt c \cdot \theta$$
where $\wt c>0$ works for all removed intervals in $[0,1]$.

We have to maintain that 
$$(1+c\sqrt{\theta})^{-1} < d^\prime(x)  < (1-c\sqrt{\theta})^{-1},\,\, x \in [\wt \zeta, \wh \zeta].$$
in order to guarantee the second inequality in $\eqref{eq:DiffeoBds}$. 
By choosing $d'$ to be close to the constant value $(1+c\sqrt{\theta})^{-1}$ or $(1-c\sqrt{\theta})^{-1}$ on most of $[\wt \zeta,\wh \zeta]$, we can make sure the function
$$ d' \mapsto \text{AVG}(d')\colonequals\frac{1}{\wh \zeta - \wt \zeta} \int_{\wt\zeta}^{\wh \zeta} d^\prime (x) \, dx $$
takes values as close as we want to $(1+c\sqrt{\theta})^{-1}$ or $(1-c\sqrt{\theta})^{-1}$, respectively. Also, by choosing $c$ sufficiently large, 
$$ (1+c\sqrt{\theta})^{-1}< 1 - \wt c \cdot \theta < 1 + \wt c \cdot \theta < (1-c\sqrt{\theta})^{-1}. $$
Since, $\text{AVG}(\cdot): C[\wt \zeta, \wh \zeta] \to \reals$ is continuous, the intermediate value theorem \cite[Theorem 4.22]{Rudin} implies 
that there is a continuous function $d^{\prime}$ such that \eqref{eq:AreaCnd} is satisfied. As the final step, we have to check whether the first inequality in \eqref{eq:DiffeoBds} holds. To this end, note that
\begin{align*}
    d(x)-x = \int_{\wt \zeta}^{x} d'(y) \, dy  + (d(\wt \zeta) - \wt \zeta) + (\wt\zeta - x),
\end{align*}
and hence,
\begin{align*}
    |d(x)-x| 
    &\leq  (1-c\sqrt{\theta})^{-1}|\wt\zeta - x|  + |d(\wt \zeta) - \wt \zeta)| + |\wt\zeta - x|\\ &\leq (1-c\sqrt{\theta})^{-1}\wt \theta^2 + \theta + \wt \theta^2 < c\sqrt{\theta}
\end{align*}
for $\theta$ sufficiently small. Therefore, $d^\prime$ extends to a continuous and strictly increasing function on $[0,1]$ in such a way that $d$ is continuous and satisfies \eqref{eq:DiffeoBds}. Since $d$ is a diffeomorphism such that $$g|_{B_\theta}=\wt g\circ d|_{B_\theta},\, m(B_\theta)>1- c \sqrt{\theta},\, |d(x)-x|<c \sqrt{\theta},\, \Big|\frac{1}{d'(x)}-1\Big|<c  \sqrt{\theta},$$
we have $\mathfrak{d}(g,\mathfrak L_\theta) \leq c \cdot  \sqrt{\theta}.$
Hence, taking $\cL_\theta$ and $\cL_0$ to be the transfer operators of $\mathfrak L_\theta$ and $g$, respectively, we have from \eqref{eq:SkohorodEst},
\begin{equation}\label{eq:NormEstimatExp}
    \| \cL_\theta - \cL_0\|_{\cB,\wt \cB}  \leq 12 c \cdot \sqrt{\theta}.
\end{equation}

As a result, we have the continuity at $(0,0)$: 
\begin{align*}
    \|\cL_{\theta,is}-\cL_{0}\|_{\BV, L^1} &\leq \|\cL_{\theta,is}-\cL_\theta\|_{\BV, L^1}+\|\cL_\theta - \cL_0\|_{\BV, L^1} \to 0,\,\,\,\text{as}\,\,\, (\theta,s)\to (0,0).
\end{align*}
In fact, we have continuity in a neighbourhood of $(0,0)$:
\begin{align*}
    \|&\cL_{\theta,is}\psi - \cL_{\bar \theta, i \bar s} \psi\|_{L^1} \\&\leq \|\cL_{\theta,is}\psi - \cL_{\bar \theta, i s} \psi\|_{L^1} + \|\cL_{\bar \theta, i s}\psi - \cL_{\bar \theta, i \bar s} \psi\|_{L^1} \\ &\leq \|\cL_\theta(e^{is h_\theta} \psi) -\cL_{\bar \theta}(e^{is h_{\bar \theta}}\psi) \|_{L^1} +  \|\cL_{\bar \theta, i s}\psi - \cL_{\bar \theta, i \bar s} \psi\|_{\BV} 
   \\  &\leq \|(\cL_\theta - \cL_{\bar \theta})(e^{is h_\theta} \psi) \|_{L^1} +  \|\cL_{\bar \theta}((e^{is h_\theta} - e^{is h_{\bar \theta}})\psi) \|_{L^1}+ \|(\cL_{\bar \theta, i s}- \cL_{\bar \theta, i \bar s}) \psi\|_{\BV}
   \\ &\leq \|\cL_\theta - \cL_{\bar \theta}\|_{\BV, L^1}\|e^{ish_\theta}\psi\|_{\BV} + \|\cL_{\bar\theta}\|_{L^1,L^1}\|\|(e^{is h_\theta} - e^{is h_{\bar \theta}})\psi\|_{L^1}+\|\cL_{\bar \theta, i s}- \cL_{\bar \theta, i \bar s} \|_{\BV}\|\psi\|_{\BV}\\
   &\leq \mathfrak d(\mathfrak L_\theta,\mathfrak L_{\bar \theta})|s|\|h_\theta\|_{\BV}\|\psi\|_{\BV} + \|\cL_{\bar\theta}\|_{L^1,L^1}|s|\|(h_\theta - h_{\bar \theta})\psi\|_{L^1}+\|\cL_{\bar \theta, i s}- \cL_{\bar \theta, i \bar s} \|_{\BV}\|\psi\|_{\BV}\\
   &\leq \mathfrak d(\mathfrak L_\theta,\mathfrak L_{\bar \theta})|s|\|h_\theta\|_{\BV}\|\psi\|_{\BV} + \|\cL_{\bar\theta}\|_{L^1,L^1}|s|\|h_\theta - h_{\bar \theta}\|_{\BV}\|\psi\|_{\BV}+\|\cL_{\bar \theta, i s}- \cL_{\bar \theta, i \bar s} \|_{\BV}\|\psi\|_{\BV}
\end{align*}
Therefore, 
\begin{align}\label{eq:ExpNorm2}
    \|\cL_{\theta,is} - \cL_{\bar \theta, i \bar s}\|_{\BV,L^1} &= \sup_{\|\psi\|_{BV}\leq 1}\|\cL_{\theta,is}\psi - \cL_{\bar \theta, i \bar s} \psi\|_{L^1} \nonumber \\ &\leq  \mathfrak d(\mathfrak L_\theta,\mathfrak L_{\bar \theta})|s|\|h_\theta\|_{\BV} + \|\cL_{\bar\theta}\|_{L^1,L^1}|s|\|h_\theta - h_{\bar \theta}\|_{\BV}+\|\cL_{\bar \theta, i s}- \cL_{\bar \theta, i \bar s} \|_{\BV}
\end{align}
Assume $\theta>\bar \theta$. Then, $\mathfrak L_{\bar \theta}$ is a piece-wise expanding map and $\mathfrak L_{\theta}$ is a spline approximation of $\mathfrak L_{\bar \theta}$. So, from the general analysis we did, we know that $\mathfrak d (\mathfrak L_\theta,\mathfrak L_{\bar \theta}) \to 0$ as $\theta \to \bar\theta^+$. When $\bar \theta > \theta$, $\mathfrak L_{\bar\theta}$ is a spline approximation of the piecewise expanding map $\mathfrak L_{\theta}$. So, when $\bar \theta$ and $\theta$ are close, $\mathfrak d (\mathfrak L_\theta,\mathfrak L_{\bar \theta})$ is close to $0$ as well. So,  $\mathfrak d (\mathfrak L_\theta,\mathfrak L_{\bar \theta}) \to 0$ as $\theta \to \bar\theta ^-$. So, from \eqref{eq:ExpNorm2}, we have,
$$\lim_{(\theta,s)\to (\bar\theta,\bar s)}\|\cL_{\theta,is} - \cL_{\bar \theta, i \bar s}\|_{\BV,L^1} = 0$$

\Cref{asm:beta} with $\cX_0=\cX^{(+)}_3=\BV([0,1])$ and any $p_0 \geq 1$, and \Cref{asm:gamma} follows as in the previous example from the results in \cite{ButterleyEslami, FernandoPene}. We essentially use the non-arithmeticity of $h_\theta$ here. However, $\cL_\theta$ need not be a covering for \cite[Lemma 1]{ButterleyEslami} to hold. 

\subsection{Markov models}\label{sec:Markov}

In addition to dynamical systems, our continuous Edgeworth expansion result in \Cref{thm:CtsEdgeExp} is also applicable in the Markovian setting. Several ideas in \cite{Ferre, Francoise, HervePene, FernandoPene} give us conditions to establish first-order Edgeworth expansions under the optimal moment conditions of the iid case. In the discussion below, using Markov integral operators in place of transfer operators, we show that first-order continuous Edgeworth expansions hold for $V-$geometrically ergodic Markov processes, and hence, the bootstrap has the desired asymptotic accuracy in that setting.

Let $(X,\cF, m)$ be a Borel probability space and $J \subset \reals $ be a neighbourhood of $0$. Recall that $\cM_1(X)$ is the set of Borel probability measures on $X$. For each $\theta \in J$, let $\{x^\theta_n\}_{n\geq 0}$ be a time homogeneous Markov chain on $X$ with initial measure $\mu_\theta \in \cM_1(X)$ and the Markov transition operator
$$\cL_\theta(\vp)(x)=\int_{X} \vp(y) d\bP_\theta(x,dy) = \EXP(\vp | x_0=x),\,\,\, \vp \in L^1(X). $$

\subsubsection{$V-$geometrically ergodic Markov chains}
Let $V:X\to [1,\infty)$ be a measurable function with  
$\EXP_{m}(V)<\infty$. Let $\|f\|_{L^\infty_V} = \sup_{x \in X}|f(x)|/V(x)$ for $f:X\to \complex$, let $$L^\infty_V \colonequals  \left\{f:X \to \complex\Big| \|f\|_{L^\infty_V} <\infty\right\},$$
and let $\|\cdot\|_{L^\infty_V, L^\infty_V}$ be the operator norm. Now suppose $\{x^0_n\}$ is irreducible and aperiodic. In addition, suppose $\{x^0_n\}$ is $V-$geometrically ergodic, i.e., there exist $\nu \in \cM_1(X)$, $C>0$ and $\kappa \in (0,1)$ such that $\EXP_\nu(V)<\infty$ and
\begin{equation}\label{GeomErg}
    \|\cL^n -\EXP_{\nu}[\,\cdot\,] {\bf 1}_{X} \|_{L^\infty_V, L^\infty_V} \leq C\kappa^n. 
\end{equation}
As in \cite{Ferre}, we assume that 
\begin{displayquote}
\begin{enumerate}[leftmargin=*,itemsep=5pt]
    \item[(i)] there exist $N, L>0$ and $\kappa \in (0,1)$ such that $\cL^N_\theta V \leq \kappa^N V + L {\bf 1}_{X},\,\,\theta \in J$, 
    \item[(ii)] $\lim_{\theta \to 0} \|\cL_\theta - \cL\|_{L^\infty,L^\infty_V}=0.$
\end{enumerate}
\end{displayquote}

Let $h_\theta \in L^\infty_V,\, \theta \in J$ be such that $h_\theta \to h_0$ in $L^\infty_V$ as $\theta \to 0$ (often, this is vacuous in practice because we take $h_\theta = h_0$ for all $\theta$) and $h_\theta \in L^{3}$
for all $\theta$. We further assume as in \cite[Section 7]{Francoise} that
\begin{equation}\label{GeomRegular}
    \sup_{a=1,2,3}\sup_{j=0,\dots,a} \|\cL_\theta(h^{a-j}_\theta V^{j/3})\|_{L^\infty_{V^{a/3}}} < \infty.
\end{equation}
We note from \cite[Remark 7.6]{Francoise} that \eqref{GeomRegular} is true as soon as $\|h^3_\theta\|_{L^{\infty}_V}<\infty$ when the Markov chains are stationary.
Finally, we assume the following non-lattice condition in \cite[p.~435]{HervePene}: 
\begin{displayquote}
It is not the case that there exist $a\in \reals$ and $b>0$, a $\nu-$full $\cL-$absorbing set $U \in \cF$ and $\zeta \in L^\infty$ such that for all $x \in U, $ $h(y)+\zeta(y)-\zeta(x) \in a + b\integers.$
\end{displayquote}

\begin{rem}

Assuming a weak form of convergence in the Assumption (ii) as opposed to convergence in $\|\cdot\|_{L^\infty_V,L^\infty_V}$, 
allows us to for many well-studied family of Markov chains, for example, the
\textit{auto-regressive processes} in \Cref{eg:autoregressive} below; see \cite{Ferre} for details. 
\begin{exmp}\label{eg:autoregressive}
Let $X = \reals$ and $\{x^\theta_n\}_{n\geq 0}$ be defined by
$x^\theta_n = \theta \cdot x_{n-1} + \xi_n, \,\,\, n \geq 1$
where $x_0$ is a $\reals$-valued random variable, $\theta\in (-1,1)$ and $\xi_n$ is a sequence of absolutely continuous iid random variables that are independent of $x_0$. Then, one can show that $\{x^\theta_n\}$  is $V-$geometrically ergodic with $V(x)=1+|x|$, and $\|\cL_\theta - \cL_{\theta_0} \|_{L^\infty_V,L^\infty_V} \not\to 0$ but $\|\cL_\theta - \cL_{\theta_0} \|_{L^\infty,L^\infty_V} \to 0$ as $\theta \to \theta_0$. 
\end{exmp}
\end{rem}

Now, we verify our spectral assumptions. To this end, let $\cB=\cX_0=\complex \cdot {\bf 1}_X$, $\cX^{(+)}_a=\cX_{a+1}= L^\infty_{V^{(a+1)/3}}$ for $a\in \{0,1,2\}$ and $\cX^{(+)}_3=\cX_4 = \wt \cB= L^1$ as in \cite[Section 7]{Francoise} with $r=2$ there. Then, it is easy to see that $\cX_a, a=0,1,2,3,4$ forms a chain of Banach spaces of the form \eqref{SpaceChain} (here we use essentially the assumptions $V\geq 1$, $\EXP_m(V)<\infty$ and $\EXP_\nu(V)<\infty$). Recall that $\cL_{\theta,is}$ given by \eqref{eq:cL_theta_is}. It is straightforward that $\cL_{\theta,is}$ are linear operators on $\cB$. Applying \eqref{GeomRegular} 
with $a=j$, we have for all $\vp\in L^\infty_{V^{a/3}}$, 
\begin{align*}
    \|\cL_{\theta,is}(\vp)\|_{L^\infty_{V^{a/3}}} = \sup_x |(V(x))^{-a/3}\cL_{\theta,is}(\vp)(x)| \leq \|\vp\|_{L^\infty_{V^{a/3}}}\sup_x |(V(x))^{-a/3}\cL_{\theta}(V^{a/3})(x)|\lesssim \|\vp\|_{L^\infty_{V^{a/3}}}.
\end{align*}
So, $\cL_{\theta,is}$ are bounded linear operators on $L^\infty_{V^{a/3}}$ for all $a=1,2,3$. 

Now, under (i) and (ii), we have the following (see \cite[Theorem 1]{Ferre}) with $\nu_0\colonequals \nu$:
\begin{displayquote}
\begin{enumerate}[leftmargin=*,itemsep=5pt]
    \item For each $\bar \kappa \in (\kappa,1)$, there exists $\eps$ such that for all $|\theta|<\eps$, $\{x^\theta_n\}$ has a unique invariant probability measure, $\nu_\theta$, such that $\nu_\theta (V)<\infty$ and
    \begin{equation*}
    \sup_{|\theta|<\eps}\|\cL^n_\theta-\EXP_{\nu_\theta}[\,\cdot\,] {\bf 1}_{X} \|_{L^\infty_V, L^\infty_V} \leq C\bar \kappa^n, 
    \end{equation*}
    \item $\lim_{\theta \to 0} \sup_{\|\vp\|_{L^\infty} \leq 1} |\nu_\theta(\vp) - \nu_0(\vp)| =0.$
\end{enumerate}
\end{displayquote}
So, for all $\theta$ sufficiently close to $0$, $\{x^\theta_n\}$ is $V-$geometrically ergodic with the same rate $\bar \kappa$. Now, we fix $\bar \kappa$, and this fixes $\eps$, and we reduce the $\theta$ range from $J$ to $J \cap (-\eps,\eps)$. 

Next, for all $\vp \in L^\infty$,
\begin{align*}
    \|\cL_{\theta, is}\vp - \cL_{0,i\bar s}\vp\|_{L^\infty_V}
    &\leq \|\cL_{\theta}-\cL\|_{L^\infty,L^\infty_V}\|e^{ish_\theta}\vp\|_{L^\infty} + \|\cL_0\|_{L^\infty_V,L^\infty_V}\|(e^{ish_\theta}-e^{i\bar s h_0}) \vp\|_{L^\infty_V}\\
    &\leq \|\cL_{\theta}-\cL\|_{L^\infty,L^\infty_V}\|\vp\|_{L^\infty} + \|\cL_0\|_{L^\infty_V,L^\infty_V}\left(\|h_\theta - h_0\|_{L^\infty_V}+|s-\bar s|\|h_0\|_{L^\infty_V}\right)\|\vp\|_{L^\infty},
\end{align*}
and hence, $$\lim_{(\theta,s)\to (0,\bar s)}\|\cL_{\theta, is} - \cL_{0,i\bar s}\|_{L^\infty,L^\infty_V} = 0.$$
A similar argument gives that, for all $|\theta|<\eps$, $$\lim_{s \to \bar s} \|\cL_{\theta, is} - \cL_{\theta,i\bar s}\|_{L^\infty,L^\infty_V} = 0.$$
Due to \Cref{WeakerAlpha}, \Cref{rem:AltAssump1} and \Cref{rem:AltAssump2}, this weaker assumption is a sufficient replacement for \Cref{asm:alpha}(1) where $\cB=\complex \cdot {\bf 1}_{X}$ equipped with $\|\cdot\|_{L^\infty}$ and $\wt \cB=L^1$. \Cref{asm:alpha}(2) follows from (1) because 
\begin{align*}
    \|\cL^n_{\theta,is}(c\cdot {\bf 1}_{X})\|_{L^\infty} \leq |c|\|\cL^n_\theta({\bf 1}_{X})\|_{L^\infty} \leq |c|(C\bar \kappa^n + \|{\bf 1}_{X}\|_{L^\infty_V}) &= C\bar \kappa^n \|c\cdot {\bf 1}_{X}\|_{L^\infty}+\|c\cdot {\bf 1}_{X}\|_{L^\infty_V}\\
    &\leq C\bar \kappa^n \|c\cdot {\bf 1}_{X}\|_{L^\infty}+\EXP_m(V)\|c\cdot {\bf 1}_{X}\|_{L^1}
\end{align*}
for all $c \in \complex$.

Due to \cite[Theorem 7.5]{Francoise}, for $a=1,2,3$, 
$$\|\cL_{\theta,is}(\vp)\|_{L^\infty_{V^{a/3}}} \leq C \kappa^n \|\vp\|_{L^\infty_{V^{a/3}}} + \|\vp\|_{L^1},$$
for all $\vp \in L^\infty_{V^{a/3}}$. So, we have \Cref{asm:beta}(3). The same theorem gives the regularity required in \Cref{asm:beta}(1,2). From $$\|(h_\theta)^j \vp\|_{L^{(a+j)/3}} = \|(h_\theta)^j\|_{L^\infty_{V^{j/3}}}\|\vp\|_{L^\infty_{V^{a/3}}},$$ we note that $\cL_{\theta, 0}^{(j)}(\cdot)= \cL_\theta((ih_\theta)^j\,\cdot)$ is a linear operator from $L^{\infty}_{V^{a/3}}$ to $L^{\infty}_{V^{(a+j)/3}}$. So, due to \eqref{GeomRegular}, for all $\theta$, for all $a=0,1,2$ and for all $\vp \in L^\infty_{V^{a/3}}$, the Taylor expansion 
$$\cL_{\theta,is}(\vp) = \sum_{k=0}^{j}\frac{\cL_\theta((ih_\theta)^k \vp)}{k!}s^k + \|\vp\|_{L^\infty_{V^{a/3}}}o(s^j),\,\,\, s \in \reals$$
holds in $L^{\infty}_{V^{(a+j)/3}}$ for $j=1,\dots,3-a$ with $L^\infty_{V^{0}}$ is understood to be $\complex \cdot {\bf 1_X}$. 

We recall from \cite[Lemma 10.1]{HervePene} that \eqref{GeomErg} and (1) above implies that for all $\eta \in (0,1]$,
\begin{equation*}
    \|\cL_\theta^n -\EXP_{\nu}[\,\cdot\,] {\bf 1}_{X} \|_{L^\infty_{V^\eta}, L^\infty_{V^\eta}} \leq C\kappa^n,
\end{equation*}
and we have \Cref{asm:gamma}(1,3). \Cref{asm:gamma}(4) follows from the non-lattice assumption; see \cite[Section 5.2]{HervePene}. \Cref{asm:gamma}(2) is a given because $\bP_\theta(x,\cdot)$ are probability measures. In fact, we have the stronger conclusion that $\cL_\theta$ has a spectral gap of $1-\bar \kappa$ on $L^\infty_{V^{\eta}_\infty}$ for all $\eta \in (0,1]$ and $|\theta|<\eps$. 

\begin{rem}
In the special case of $V\equiv 1$, we obtain the case of uniformly ergodic Markov chains. We consider a single space $\cB = \wt \cB = L^\infty$ and most of our assumptions at the beginning of this section become vacuous, and we recover the results of \cite{DM} for \textit{bounded} observables. 
\end{rem}
\begin{rem}  
We could have considered the more general case of $(Y,\cG,\wt m)$ being another Borel probability space, $\{y_n\}_{n\geq 0}$ being a sequence of iid sequence of random variables on $Y$ with common distribution $\wt m$ and independent of $\{x^\theta_n\}$ for all $\theta \in J$,  
$h_\theta:X \times X \times Y \to \complex$, $X^\theta_n=h_\theta(x^\theta_{n-1},x^\theta_{n}, y_n)$. 
\end{rem}

\part{The Bootstrap}\label{part:Bootstrap}

The key idea behind the classical bootstrap is that the process of repeatedly resampling from an observed sample mimics the process of the original iid sampling from the whole population. However, iid sampling in the classical bootstrap cannot mimic the dependence structure within the dynamically generated data and is not applicable in the our setting. The goal of this part of the paper is to develop a generic bootstrap algorithm for dynamical systems and study its asymptotic accuracy.

\section{The Bootstrap Algorithm}\label{sec:BootstrapAlgo}
Consider the dynamical system characterized by transformation function $g:X\to X$, initial measure $\mu$, and an observable $h$, as in \eqref{eq:DataGenerating1}, \eqref{eq:DataGenerating2}, and \eqref{eq:estimator}. We assume $h$ is known, which is the case in most practical situations. However, $g$ and $\mu$ are usually unknown in practice.

Let $S_{n}(h)$, $A$, and $\sigma$ be defined by, as in \eqref{eq:BirkhoffSum}, \eqref{eq:Mean}, and \eqref{eq:Variance},
\begin{align*}
    S_{n}(h) = \sum_{k=0}^{n-1} h \circ g^{k}(x_{0}), \,\,
    A = \lim_{n \to \infty} \EXP_\mu\left(\frac{S_{n}(h)}{n}\right),\,\, \sigma =\sqrt{\lim_{n\to\infty} \EXP_{\mu}\left(\frac{S_{n}(h)-nA}{\sqrt{n}}\right)^2}.
\end{align*}
Notice that, {by \Cref{ContVar}}, $A$ is equal to the space average of $h$ with respect to $\nu_0$, i.e., $A=\nu_0(h).$ To approximate the distribution of the estimators of $A$, we develop bootstrap algorithms; specifically, to mimic the dependence structure of the dynamical system, in our algorithms we replace $g$ and $\mu$ by the transformation estimator $\wh g$ and the bootstrap initial distribution $\mu^{*}$, respectively. Below we detail the algorithms of the pivoted and  non-pivoted bootstrap in \Cref{sec:pivot} and \Cref{sec:nonpivot}, respectively. Afterwards, we will briefly describe the construction of $\wh g$ and $\mu^{*}$ in \Cref{sec:EstConstr} and discuss the theoretical suitability of these choices of $\wh g$ and $\mu^{*}$ in \Cref{sec:Suitability}.

\subsection{Pivoted bootstrap}\label{sec:pivot}

First, we state the pivoted bootstrap algorithm which is used when $\sigma$ is known or can be well-approximated. Its asymptotic accuracy is proved in \Cref{sec:AsympAccuPivot} and is $o_\text{a.s.}(n^{-1/2})$ in general. Theoretically, an accuracy of $\cO_\text{a.s.}(n^{-1})
$ is possible but we need stronger spectral assumptions to achieve this. 
\vs

{\footnotesize
\normalem
\begin{algorithm}[H]
	\caption{Pivoted Bootstrap}
	\label{alg:pivot}
	\DontPrintSemicolon
	\SetAlgoLined
	\KwIn{\parbox[t]{0.39\linewidth}{Data $x_{0}, \dots, x_{n-1}$ \\
							$h$: observable function\\
							$\wh g$: estimator of transformation $g$\\
							$\wh{\sigma}^{2}$: estimator of long-run variance $\sigma^{2}$ \\
							}
        	 \parbox[t]{0.45\linewidth}{
        	 $\alpha$: significance level, e.g., $5\%$\\
        	 $B$: number of bootstrap iterations\\
        	 $\mu^{*}$: bootstrap initial distribution \\
    }}
    \BlankLine
    \KwOut{$1-\alpha$ confidence interval of $A$} 
	\BlankLine
	$S_{n}(h) \leftarrow \sum_{i=0}^{n-1} h(x_{i})$\;
	$\wh{\sigma} \leftarrow \wh{\sigma}(x_{0}, \dots, x_{n-1}) $\; 
	\BlankLine
	\For{$b \leftarrow 1$ \KwTo $B$}{
		\BlankLine
		Sample $x_{0}^{*,b}\sim \mu^{*}$\; \label{lin:bootInit_pivoted}
		\BlankLine
		\For{$k \leftarrow 1$ \KwTo $n-1$}{
		\BlankLine
		$x_{k}^{*,b} \leftarrow \wh g^{k}(x_{0}^{*,b})$\;
		\BlankLine
		}
	    $S_{n}^{*,b}(h) \leftarrow \sum_{i=0}^{n-1} h(x_{i}^{*,b})$ \;
	    $\wh{\sigma}^{*,b}\leftarrow \wh{\sigma}(x_{0}^{*,b}, \dots, x_{n-1}^{*,b})$\; 
	}
	\BlankLine
	$\bar{S}_{n}^{*}(h)  \leftarrow \frac{1}{B}\sum_{b=1}^{B}S_{n}^{*,b}(h)$ \;
	\For{$b \leftarrow 1$ \KwTo $B$}{
	$T_{n}^{*,b} \leftarrow \frac{S_{n}^{*,b}(h)-\bar{S}_{n}^{*}(h)}{\sqrt{n}\wh{\sigma}^{*,b}}$ \label{lin:T_pivoted}
	}
	$F_{n}^{*}(x) \leftarrow \frac{1}{B}\sum_{b=1}^{B}{\bf 1}\{T_{n}^{*,b}\leq x\}$
	
	\Return{$\left\{A\in \reals \,\Big|\,\frac{S_{n}(h)-nA}{\sqrt{n}\wh{\sigma}}\in \big((F_{n}^{*})^{-1}(\alpha/2), (F_{n}^{*})^{-1}(1-\alpha/2)\big)\right\}$.}
\end{algorithm}
\ULforem 
}

\subsection{Non-pivoted bootstrap}\label{sec:nonpivot}

Next, we present the non-pivoted bootstrap algorithm to be used when $\sigma$ is unknown and cannot be approximated reliably. We will see in \Cref{sec:AsympAccuNonpivot} that, in general, its asymptotic accuracy is $o_\text{a.s.}(1)$. However, an accuracy of $\cO_\text{a.s.}(n^{-1/2})$ is theoretically possible. Under some reasonable assumptions, we provide an example with accuracy $\cO_\text{a.s.}(n^{-1/3})$.
\vs

{\footnotesize
\normalem
\begin{algorithm}[H]
	\caption{Non-pivoted Bootstrap}
	\label{alg:nonpivot}
	\DontPrintSemicolon
	\SetAlgoLined
	\KwIn{\parbox[t]{0.39\linewidth}{Data $x_{0}, \dots, x_{n-1}$ \\
								 $h$: observable function\\
							 $\wh g$: estimator of transformation $g$\\
							}
        	 \parbox[t]{0.45\linewidth}{
        	 $\alpha$: significance level, e.g., $5\%$\\
        	 $B$: number of bootstrap iterations\\
        	 $\mu^{*}$: bootstrap initial distribution \\
    }}
    \BlankLine
    \KwOut{$1-\alpha$ confidence interval of $A$} 
	\BlankLine
	$S_{n}(h) \leftarrow \sum_{i=0}^{n-1} h(x_{i})$\;
	\BlankLine
	\For{$b \leftarrow 1$ \KwTo $B$}{
		\BlankLine
		Sample $x_{0}^{*,b}\sim \mu^{*}$\; \label{lin:bootInit_non_pivoted}
		\BlankLine
		\For{$k \leftarrow 1$ \KwTo $n-1$}{
		\BlankLine
		$x_{k}^{*,b} \leftarrow \wh g^{k}(x_{0}^{*,b})$\;
		\BlankLine
		}
	    $S_{n}^{*,b}(h) \leftarrow \sum_{i=0}^{n-1} h(x_{i}^{*,b})$ \;
	}
	\BlankLine
	$\bar{S}_{n}^{*}(h)  \leftarrow \frac{1}{B}\sum_{b=1}^{B}S_{n}^{*,b}(h)$ \;
	\For{$b \leftarrow 1$ \KwTo $B$}{
	$T_{n}^{*,b} \leftarrow \frac{S_{n}^{*,b}(h)-\bar{S}_{n}^{*}(h)}{\sqrt{n}}$\label{lin:T_non_pivoted}
	}
	$F_{n}^{*}(x) \leftarrow \frac{1}{B}\sum_{b=1}^{B}{\bf 1}\{T_{n}^{*,b}\leq x\}$
	\BlankLine
	
	\Return{$\left\{A\in \reals\,\Big|\, \frac{S_{n}(h)-nA}{\sqrt{n}}\in \big((F_{n}^{*})^{-1}(\alpha/2), (F_{n}^{*})^{-1}(1-\alpha/2)\big)\right\}$.}
\end{algorithm}
\ULforem 
}
\vs 

\begin{rem} \label{rem:truebootexp}
Since we would like to approximate the distribution of the rescaled version of $S_{n}(h)-nA$,
theoretically, we want to in Line \ref{lin:T_pivoted} of Algorithm \ref{alg:pivot} and Line \ref{lin:T_non_pivoted} of Algorithm \ref{alg:nonpivot} subtract ${n}A^*$ from $S_{n}^{*,b}(b)$, where $A^*$ is defined in \eqref{eq:bootQuantities}. However, it is in general difficult to find the closed form of $A^*$. So, we substitute ${n}A^*$ with $\bar{S}_{n}^{*}(h)$. Indeed, when both $B$ and $n$ are large
\[
\bar{S}_{n}^{*}(h) \approx
\EXP^{*} \left(S_{n}^{*}(h)\right) \approx {n}A^{*}.
\]
\end{rem}

\section{Asymptotic Accuracy of the Bootstrap} \label{sec:Accu}

Let $\EXP^*$ be the expectation conditional on data $x_{0},\dots, x_{n-1}$.
Let $\mu^{*}$ be a random initial distribution depending on $x_{0},\dots, x_{n-1}$ and $n^{*}$ be the size of the bootstrap data we generated (from an original data of size $n$). Let the rescaled summation, the spatial average, and the long-run variance for the bootstrap sample be defined by, 
\begin{equation}\label{eq:bootQuantities}
\begin{aligned}
S_{n^{*}}^{*}(h)= \sum_{k=0}^{n^{*}-1} h \circ \wh g^{k}(x_{0}^{*}),\,\,
A^*= \lim_{n^{*} \to \infty} \EXP_{\mu^{*}}^*\left(\frac{S_{n^{*}}^{*}(h)}{n^{*}}\right),\,\,\sigma^* =\sqrt{\lim_{n^{*}\to\infty} \EXP_{\mu^{*}}^*\left(\frac{S_{n^{*}}^{*}(h)-n^{*}A^*}{\sqrt{n^{*}}}\right)^2},
\end{aligned}
\end{equation}
respectively, with the limit ${n^{*} \to \infty}$ indicating the almost-sure limit as $n^{*}$ goes to infinity. 
By \Cref{ContVar}, when \Cref{asm:beta} and \Cref{asm:gamma}(1) hold, 
we have that $A^*$ and $\sigma^{*}$ do not depend on the choice of $\mu^{*}$, $A^*-A = o_{\text{a.s.}}(1)$, and $\sigma^{*}-\sigma = o_{\text{a.s.}}(1)$.

To be consistent with Algorithms \ref{alg:pivot} and \ref{alg:nonpivot}, from now on we set $\mu^{*}$ as the bootstrap initial distribution and $n^*=n$. Also, to
be consistent with our usual notation, treat $g$ as $g_0$, $h$ as $h_0$, $\sigma$ as $\sigma_0$, $A$ as $A_0$, $\wh g$ as $g_{\theta_n}$, $\mu^{*}$ (with density $\rho_{\mu^*}$) as $\mu_{\theta_n}$, 
$A^*$ as $A_{\theta_n}$, and $\sigma^*$ as $\sigma_{\theta_n}$, 
where $\{\theta_n\}$ is a sequence such that $\textstyle \lim_{n\to \infty}\theta_n=0$ and $n$ is the sample size.
Since we consider the asymptotics as the sample size becomes large, i.e., $n \to \infty$, we have $\theta_n \to 0$ automatically. So, we can apply results from \Cref{part:EdgeExp} about $g_{\theta}$, $\mu_{\theta}$, $A_{\theta}$, and $\sigma_{\theta}$ when $\theta\to 0$.
\subsection{Pivoted bootstrap}\label{sec:AsympAccuPivot}
Suppose that the asymptotic variance $\sigma^2$ is known. Then let 
\begin{equation}\label{eq:pivot1}
    T_n = \frac{1}{\sqrt{n}\sigma}\left(S_n(h)- nA\right) 
\end{equation}
Using Algorithm \ref{alg:pivot}, we generate bootstrap samples $\{x^*_k\}_{0 \leq k \leq n}$. Here, the bootstrap estimator of $T_n$ is 
\begin{equation}\label{eq:pivot1est}
    T^*_n = \frac{1}{\sqrt{n} \sigma^*}\left(S_n^*(h)- n A^*\right). 
\end{equation}

\begin{rem}\label{rem:VarEst}
In Algorithm \ref{alg:pivot}, we have used $\hsigma$, the estimator based on the original data, and $\hsigma^{*}$, which is based on the bootstrap data. This makes the algorithms more generic and more applicable. So, in simulations we replace $\sigma$ by $\hsigma$ and $\sigma^{*}$ by $\hsigma^{*}$ in \eqref{eq:pivot1} and \eqref{eq:pivot1est}, respectively. 
\end{rem}

\begin{thm}\label{thm:PivotedBootstrap}
Suppose that, almost surely, 
there exists $N$ such that the family $\{\wh g\,|\, n\geq N\} \cup \{g\}$ satisfies the Assumptions \ref{asm:alpha}, \ref{asm:beta}, \ref{asm:gamma}, \ref{asm:delta} and \ref{asm:ve}. Let $\cC$ be a space in \eqref{SpaceChain} such that $\cC \hookrightarrow \cX_3 \hookrightarrow \cX^{(+)}_3$ and assume that 
$\|\rho_{\mu^*} -\rho_{\mu}\|_{\cC} \to 0$ as $n \to \infty$. 
Then, 
$$\sup_{z \in \reals }\Big|\Prob_{\mu}(T_n \leq z) - \Prob 
(T^*_n \leq z\,|\,x_{0},\dots, x_{n-1})\Big| = o_{\text{a.s.}}(n^{-1/2}).$$
\end{thm}

\begin{proof}
From \Cref{thm:CtsEdgeExp}, we have
$$\cE_n(z)= \fN(z) +\frac{P(z)}{\sqrt{n}}\fn(z),\,\, z \in \reals$$ which is the first-order Edgeworth expansion for the dynamical system $g:X \to X$ with $h$ as the observable and $\mu$ as the initial measure. Under the assumptions, 
we apply \Cref{thm:CtsEdgeExp} along with \Cref{rem:IntialMeasure} to $\{\wh g\,|\, n\geq N\} \cup \{g\}$ to obtain 
$$\sup_{z\in \reals}\Big|\Prob\left(T^*_n\leq z\,|\, x_1,\dots, x_n\right) - \cE_n(z)\Big| =  o_{\text{a.s.}}(n^{-1/2}), \,\, n \to \infty.$$
By another application of \Cref{thm:CtsEdgeExp},
we have
\begin{align*}
    \sup_{z\in \reals} \Big|\Prob_{\mu}\left(T_n\leq z\right) - \cE_n(z)\Big| =\sup_{z\in \reals} \left|\Prob_{\mu}\left(\frac{S_{n}(h)-nA}{\sigma\sqrt{n}} \leq z\right) - \cE_n(z)\right|=  o(n^{-1/2}), \,\, n \to \infty.
\end{align*}
Finally, using the triangle inequality,
\begin{align*}
    \sup_{z\in \reals} &\Big|\Prob_{\mu}\left(T_n\leq z\right) - \Prob
    \left(T^*_n\leq z|\, x_1,\dots, x_n\right) \Big| \\ &\leq \sup_{z\in \reals} \Big|\Prob_{\mu}\left(T_n\leq z\right) - \cE_n(z)\Big| + \sup_{z\in \reals}\Big|\Prob
    \left(T^*_n\leq z|\, x_1,\dots, x_n\right) - \cE_n(z)\Big| = o_\text{a.s.}(n^{-1/2}), \,\, n \to \infty.
\end{align*}
\end{proof}

\begin{rem}\label{hUnkown}
When $h$ is unknown, we can replace $h$ in $T^*_n$ by a suitable estimator $\wh h$, take $\wh h$ as $h_{\theta_n}$, and denote the corresponding multiplication operator by $H_{\theta_n}$. Since \Cref{thm:CtsEdgeExp} established an Edgeworth expansion uniformly with respect to $h$, as long as $H_{\theta_n}$ satisfies \Cref{asm:delta}, \Cref{thm:PivotedBootstrap} still holds. Therefore, even when $h$ is unknown, we can still develop a bootstrap algorithm that is second-order efficient. 
\end{rem}

If $\cB \hookrightarrow L^2(m)$, then due to \Cref{rem:0avg}, the second condition of \Cref{asm:ve} is automatic. 
Hence, we have the following Corollary.
\begin{cor}\label{cor:PivotedBootstrap}
Suppose that $\cB \hookrightarrow L^2(m)$ 
and, almost surely, there exists $N$ such that the family $\{
\wh g\,|\, n\geq N\} \cup \{
g\}$ satisfies the Assumptions \ref{asm:alpha}, \ref{asm:beta}, \ref{asm:gamma}, and \ref{asm:delta} with $p_0 \geq 2$, 
and $h$ is not $g-$cohomologous to a constant in $L^2(m)$. Let $\cC$ is a space in \eqref{SpaceChain} such that $\cC \hookrightarrow \cX_3 \hookrightarrow \cX^{(+)}_3$ and assume that 
$\|\rho_{\mu^*} -\rho_{\mu}\|_{\cC} \to 0$ as $n \to \infty$. 
Then, 
$$\sup_{z \in \reals }\Big|\Prob_{\mu}(T_n \leq z) - \Prob(T^*_n \leq z \,|\,x_{0},\dots, x_{n-1})\Big| = o_{\text{a.s.}}(n^{-1/2}).$$
\end{cor}

\begin{rem}
We recall that in the iid setting, if the common distribution has three finite moments and is non-lattice, then by \cite[Chapter 3]{Hall}, we have 
$$\sup_{z \in \reals }\Big|\Prob_{\mu}(T_n \leq z) - \Prob(T^*_n \leq z \,|\,x_{0},\dots, x_{n-1})\Big| = o_{\text{a.s.}}(n^{-1/2}).$$
 Our \Cref{thm:PivotedBootstrap} and \Cref{cor:PivotedBootstrap} generalize this result where the existence of the first three finite moments follow from \Cref{asm:beta} and the distribution being non-lattice follows from \cref{asm:gamma}(4).
However, if we have stronger assumptions on the common distribution, we can say more. In particular, if, in the iid setting, the common distribution has four finite moments and satisfies Cram\'er's condition, 
 $$\limsup_{|t|\to \infty} |\EXP(e^{isX})|<1\,,$$
then the second order Edgeworth expansion exists, and it follows that, 
$$\sup_{z \in \reals }\Big|\Prob_{\mu}(T_n \leq z) - \Prob(T^*_n \leq z \,|\,x_{0},\dots, x_{n-1})\Big| = \cO_{\text{a.s.}}(n^{-1}).$$
See, for example, \cite[Section 3.3]{Hall}. So, in order to guarantee to improve the asymptotic accuracy of our bootstrap, we need stronger spectral assumptions that guarantee existence of at least four asymptotic moments and better control of the characteristic function away from the origin. We plan to explore this in a separate work.
\end{rem}

\subsection{Nonpivoted bootstrap}\label{sec:AsympAccuNonpivot} Suppose the asymptotic variance $\sigma^2$ is unknown and cannot be estimated easily. Then write 
\begin{equation}\label{eq:nonpivot1}
    T_n = \frac{1}{\sqrt{n}}\left(S_n(h) - nA\right).
\end{equation}
Using Algorithm \ref{alg:nonpivot}, we generate the bootstrap samples $\{x^*_k\}_{0 \leq k \leq n}$.
The bootstrap estimator of $T_n$ is given by 
\begin{equation}\label{eq:nonpivot1est}
    T^*_n = \frac{1}{\sqrt{n}}\left(S_{n}^{*}(h) - n A^*\right). 
\end{equation}

\begin{thm}\label{thm:NonPivotedBootstrap}
Suppose that, almost surely, there exists $N$ such that $\{\wh g\,|\, n\geq N\} \cup \{g\}$ satisfy the Assumptions \ref{asm:alpha}, \ref{asm:beta}, \ref{asm:gamma}, \ref{asm:delta}, and \ref{asm:ve}. Let $\cC$ is a space in \eqref{SpaceChain} such that $\cC \hookrightarrow \cX_3 \hookrightarrow \cX^{(+)}_3$ and assume that 
$\|\rho_{\mu^*} -\rho_{\mu}\|_{\cC} \to 0$ as $n \to \infty$.
Then, 
$$\sup_{z \in \reals }\big|\Prob_{\mu}(T_n \leq z) -\Prob(T^*_n \leq z\,|\,x_{0},\dots, x_{n-1})\Big| =\cO_{\text{a.s.}}(|\sigma^*-\sigma|)+\cO_{\text{a.s.}}(n^{-1/2}). 
$$
\end{thm}
\begin{proof}

Even though the Birkhoff sum $S_{n}^{*}(h)$ is not normalized by $\sigma^*$, due to uniformity in $x$ of we can still use the conclusion of \Cref{thm:CtsEdgeExp} to obtain
\begin{align*}
    \sup_{z \in \reals }\bigg|\Prob(T^*_n \leq z \, &|\,x_{0},\dots, x_{n-1}) - \cE_n\left(\frac{z}{\sigma^*}\right)\bigg| \\
    &= \sup_{z \in \reals }\left|\Prob\left(\frac{S_{n}^{*}(h)- n A^*}{\sigma^*\sqrt{n}} \leq \frac{z}{\sigma^*} \, \Big|\,x_{0},\dots, x_{n-1}\right) - \cE_n\left(\frac{z}{\sigma^*}\right)\right|\\ 
    &\leq \sup_{ x\in \reals} \left|\Prob \left(\frac{S^*_{n}(h)-nA^*}{\sigma^*\sqrt{n}} \leq x\, \Big |\, \,x_{0},\dots, x_{n-1}\right) - \cE_n\left(x\right) \right| = o_{\text{a.s.}}(n^{-1/2}).
\end{align*}
Similarly, we have 
$$
\sup_{z \in \reals }\left| \Prob_{\mu_{0}}(T_n \leq z) - \cE_n\left(\frac{z}{\sigma}\right) \right| = o(n^{-1/2}).
$$
As a result,
\begin{align*}
    \sup_{z \in \reals }\Big|\Prob_{\mu}(T_n \leq z) - \Prob&(T^*_n \leq z \, |\,x_{0},\dots, x_{n-1})\Big| \\ =
&\sup_{z \in \reals }\left|\cE_n\left(\frac{z}{\sigma}\right)-\cE_n\left(\frac{z}{\sigma^*}\right)\right|
+ o_{\text{a.s.}}(n^{-1/2})
\end{align*}
Note that
\begin{align*}
    \cE_n\left(\frac{z}{\sigma}\right)-\cE_n\left(\frac{z}{\sigma^*}\right)  &= \frac{1}{\sqrt{2\pi} \sigma^*}\int_{-\infty}^z e^{-\frac{y^2}{2(\sigma^*)^2}}\, dy - \frac{1}{\sqrt{2\pi}  \sigma}\int_{-\infty}^z e^{-\frac{y^2}{2\sigma^2}}\, dy \\ &\phantom{aaaaaaaaaaaaaaa}+ \frac{1}{\sqrt{n}}\left(P\left(\frac{x}{\sigma^*}\right)\fn\left(\frac{x}{\sigma^*}\right) -P\left(\frac{x}{\sigma}\right)\fn\left(\frac{x}{\sigma}\right)\right).
\end{align*}
Since the first term above is $\cO_{\text{a.s.}}(|\sigma^*-\sigma|)$ and the second is $\cO_{\text{a.s.}}(n^{-1/2})$, together we have  $$\cO_{\text{a.s.}}(|\sigma^*-\sigma|)+\cO_{\text{a.s.}}(n^{-1/2}),$$
which dominates the $o_{\text{a.s.}}(n^{-1/2})$ term.
\end{proof}

\begin{rem} 
We recall that the non-pivoted bootstrap does not depend on the knowledge of $\sigma$ while the Gaussian approximation does (in fact, it is impossible without $\sigma$). 
However, by \Cref{thm:NonPivotedBootstrap}, provided that $|\sigma^* - \sigma|\lesssim n^{-1/2}$, the non-pivoted bootstrap achieves an asymptotic accuracy of $\cO_{\text{a.s.}}(n^{-1/2})$, which is equal to that of the Gaussian approximation. Hence, in a somewhat oracular way, the non-pivoted bootstrap ``knows'' $\sigma$ without paying any price.
\end{rem}
In general, we only know that $|\sigma^*-\sigma|=o_{\text{a.s}}(1)$.
In this case, by \Cref{rem:0avg},
we have the following corollary.
\begin{cor}
Suppose that, almost surely, the Assumptions \ref{asm:alpha}, \ref{asm:beta}, \ref{asm:gamma}, and \ref{asm:delta} hold with $p_0 \geq 2$, and $h$ is not $g-$cohomologous to a constant in $L^2(m)$. Let $\cC$ is a space in \eqref{SpaceChain} such that $\cC \hookrightarrow \cX_3 \hookrightarrow \cX^{(+)}_3$ and assume that
$\|\rho_{\mu^*} -\rho_{\mu}\|_{\cC} \to 0$ as $n \to \infty$. Then, 
$$\sup_{z \in \reals }\big|\Prob_{\mu}(T_n \leq z) -\Prob(T^*_n \leq z\,|\,x_{0},\dots, x_{n-1})\Big| = o_{\text{a.s.}}(1).$$
\end{cor}

\section{Application of the Bootstrap to our Examples}

In this section, we describe some of the standard techniques available to apply the Algorithms \ref{alg:pivot} and  \ref{alg:nonpivot} in the context of dynamical systems. 

\subsection{Choice of $\wh g$ and $\mu^{*}$}\label{sec:EstConstr} 
To approximate $g$, there are a two key methods: Either one can use (8), (9), (12), and (13) of \cite{NA}, or one can use standard spline approximations; which method is more appropriate depends on the context.

For the choice of $\mu^{*}$, one may use the kernel estimator of $\nu_{0}$ on \cite[page 22]{HS}, defined by
\[
\wh \rho_{\nu_0}(x)=\frac{1}{n \fb}\sum_{k=1}^{n}K\Big(\frac{x-x_{k}}{\fb}\Big), 
\]
where $K:\reals\to \reals$ is the density function for a standard Gaussian random variable and $\fb$ is a bandwidth. 
Alternatively, one can let $\mu^{*}$ (having density $\rho_{\mu^*}$) be a fixed distribution that depends on neither $x_0,\dots,x_{n-1}$ nor $n$.

\subsection{Suitability of $\wh g$ and $\mu^* $}\label{sec:Suitability}

Now, we discuss the suitability of the choices of $\wh g$ and $\mu^{*}$ to align with our abstract setting in \Cref{part:EdgeExp}, and hence, ensure the asymptotic accuracy of the bootstrap. We require the following two conditions to be satisfied.
\begin{itemize}[leftmargin=15pt]
 \item[--] Almost surely, there is $N$ such that $\{\wh g\,|\, n\geq N\} \cup \{g\}$ satisfy the Assumptions \ref{asm:alpha}, \ref{asm:beta}, and \ref{asm:gamma}.  
   \item[--] There exists a space $\cC$  in \eqref{SpaceChain} such that $\cC \hookrightarrow \cX_3 \hookrightarrow \cX^{(+)}_3$ and $\| \rho_{\mu^*} - \rho_{\mu}\|_{\cC} \to  0$ as $n \to \infty$. 
\end{itemize}

In both smooth expanding and peicewise expanding maps, the \Cref{asm:gamma}(1) holds. So, $g$ possesses a unique acip $\nu_0$ (with density $\rho_{\nu_0}$); moreover, $(g,\nu_0)$ is strong mixing, and hence, ergodic. As a result, $\nu_0-$almost every orbit of $g$ is dense in the support of $\nu_0$; cf.~\cite[p.~29]{walters}. So, $\nu_0-$almost surely, $x_0,\dots,x_{n-1},\dots$ is a dense orbit. Therefore, we can construct spline approximations $\wh g$. We have already shown that spline approximations of expanding maps satisfy the Assumptions \ref{asm:alpha}, \ref{asm:beta}, and \ref{asm:gamma}; see Sections \ref{sec:SmoothExp} and  \ref{sec:PiecewiseExp}, and hence, they are an ideal choice for our simulations.

When a kernel estimators that changes with the sample size $n$, the above suitability condition for stationary processes can be checked using ideas in \cite[Section 3]{Yu} where the convergence properties of of $\wh\rho_{\nu_0}$ to $\rho_{\nu_0}$ are discussed. 
In particular, if $\{X_n\}$ is stationary and $\beta-$mixing, and $\rho_{\nu_0}$ is sufficiently regular, then $\wh\rho_{\nu_0}$ and $\wh\rho^{\prime}_{\nu_0}$ converge to $\rho_{\nu_0}$ and $\rho^{\prime}_{\nu_0}$, respectively and the convergence is uniform. 
In particular, for $V-$geometrically ergodic Markov chains, which by definition are $\beta-$mixing, we have $\|\wh\rho_{\nu_0}-\rho_{\nu_0}\|_{L^\infty} \to  0$. 

Since the kernel estimators in dynamical systems and $\beta-$mixing setting share many common properties \cite{HS}, we conjecture that, in the dynamical systems examples in \Cref{sec:Exmp}, we also have the uniform convergence of $\wh\rho_{\nu_0}$ and $\wh\rho^{\prime}_{\nu_0}$ to $\rho_{\nu_0}$ and $\rho^{\prime}_{\nu_0}$, which implies $\|\wh\rho_{\nu_0}-\rho_{\nu_0}\|_{\BV} \to  0$. However, this is still an open problem. In general, it is interesting to study the convergence of derivatives of kernel estimators of acips of dynamical systems but this has not been pursued in the previous literature. 

When the bootstrap initial measure, $\mu^*$, is fixed, the only condition it should satisfy is being absolutely continuous, and we opt for this option.  

\subsection{Improved accuracy of the non-pivoted bootstrap}
Now, we explain why we our non-pivoted bootstrap algorithm exhibits better asymptotic accuracy than $o_{\text{a.s.}}(1)$ when applied to our examples in \Cref{sec:SmoothExp} and \Cref{sec:PiecewiseExp}.

\begin{exmp}[Piecewise $C^2$ expanding maps, a continuation of \Cref{sec:PiecewiseExp}]
Since we assume $x_0,\dots,x_{n-1}$ is a sufficiently dense partial orbit,
we may assume that the mesh size, $\theta$, of the spline approximation is $\cO(n^{-1})$. So, $$\|\rho_{\nu_\theta}-\rho_{\nu_0}\|_{L^1} \lesssim \|\Pi_\theta - \Pi_0\|_{\BV,L^1}\lesssim \|\cL_\theta - \cL_0\|_{\BV,L^1} \lesssim \sqrt{\theta}=n^{-1/2}$$ 
where $\nu_\theta$ is the unique acip of $g_\theta$ and the implied constants are independent of $\theta$. The estimates (from left to right) follow from the proof of \Cref{cor:CtsDensity}, the Cauchy integral representation of $\Pi_\theta$ given in \Cref{sec:Moments} and \eqref{eq:NormEstimatExp}, respectively.

We recall from \Cref{sec:Moments} that for all $k>0$, 
\begin{equation*}
   \frac{1}{k} \EXP_\mu(S_{\theta,k}(h))=A_\theta + \frac{1}{ik} m(\Pi^{(1)}_\theta \rho_\mu) + \cO(\kappa^k)
\end{equation*}
Since the convergence rate of $A_\theta \to A_0$ is independent of the choice of the initial measure, we write,
\begin{align*}
    A_\theta - A_0 
     &= \frac{1}{k}\left(\EXP_{\nu_\theta}(S_{k,\theta}(h)) - \EXP_{\nu_0}(S_{k,0}(h)) \right)+o(k^{-1})\\
     &= \frac{1}{k}\sum_{j=1}^{k-1} \left[ \EXP_{\nu_\theta}(h \circ g_\theta^j ) - \EXP_{\nu_0}(h \circ g_0^j)\right]+o(k^{-1})\\
     &= \frac{1}{k}\sum_{j=1}^{k-1} \left[ \EXP_{\nu_\theta}(h) - \EXP_{\nu_0}(h)\right]+o(k^{-1})\\
|A_\theta - A_0|
     &\leq \frac{1}{k}\sum_{j=1}^{k-1}\left|\int h(\rho_{\nu_\theta}-\rho_{\nu_0})\, dm\right|+o(k^{-1})\\
     &\leq  \|h\|_{L^\infty}\|\rho_{\nu_\theta}-\rho_{\nu_0}\|_{L^1} +o(k^{-1})
     \lesssim n^{-1/2}+o(k^{-1})
\end{align*}
Above, we use that $\nu_\theta$ is $g_\theta-$invariant. Next, choosing $k= \cO(n^{1/2})$, we can obtain the best possible rate of convergence of the asymptotic means:
$$|A^*-A| = |A_\theta - A_0| \lesssim n^{-1/2}.$$
Now, we focus on the rate of convergence of the variance. We recall that for all $k>0$, 
\begin{equation*}
    \frac{1}{k}\EXP_\mu([S_{\theta,k}(h_\theta)- k A_\theta]^2) = \sigma^2_\theta  -\frac{1}{k}m(\Pi^{(2)}_\theta \rho_\mu)  + \cO(\kappa^n)
\end{equation*}
Then,
\begin{align*}
   \sigma^2_\theta - \sigma^2_0 
    &= \frac{1}{k}\left(\EXP_{\nu_\theta}([S_{k,\theta}(h)-kA_\theta]^2)-\EXP_{\nu_0}([S_{k,0}(h)-kA_0]^2) \right)+ o(k^{-1})\\
    &=\frac{1}{k}\left(\EXP_{\nu_\theta}([S_{k,\theta}(h)]^2)-\EXP_{\nu_0}([S_{k,0}(h)]^2\right) -k [A^2_\theta - A^2_0]+o(k^{-1})
\end{align*}
Since, $|A^2_\theta - A^2_0| \lesssim |A_\theta-A_0|\lesssim n^{-1/2}$ and the second term is $\cO(kn^{-1/2})$.

To estimate the first term, note that, 
\begin{align*}
    |\EXP_{\nu_\theta}([S_{k,\theta}(h)]^2)-\EXP_{\nu_0}([S_{k,0}(h)]^2)|&=  \left|\sum_{j=0}^{k-1}\sum_{l=0}^{k-1} \EXP_{\nu_\theta}(h \circ g_\theta^j \cdot h \circ g_\theta^l) - \EXP_{\nu_0}(h \circ g_0^j \cdot h \circ g_0^l)\right|\\
    &= 2 \left|\sum_{j=1}^{k-1}(k-j)[\EXP_{\nu_\theta}(h \circ g_\theta^j \cdot h) - \EXP_{\nu_0}(h \circ g_0^j \cdot h)]\right|\\
    &\leq 2 \sum_{j=1}^{k-1}(k-j)\left|\EXP_{m}(h \cdot \cL^j_\theta(h\rho_{\nu_\theta})) - \EXP_{m}(h \cdot \cL^j_0(h\rho_{\nu_0}))\right|\\
    &\leq 2\|h\|_{L^\infty}\sum_{j=1}^{k-1}(k-j) \|\cL^j_\theta(h\rho_{\nu_\theta})-\cL^j_0(h\rho_{\nu_0})\|_{L^1},
\end{align*}
\begin{align*}
    \|\cL^j_\theta(h\rho_{\nu_\theta})-\cL^j_0(h\rho_{\nu_0})\|_{L^1} &\leq \|\cL^j_\theta(h(\rho_{\nu_\theta} - \rho_{\nu_0}))\|_{L^1}+\|(\cL^j_\theta-\cL^j_0)(h\rho_{\nu_0})\|_{L^1} \\
    &\leq \sup_{\theta, r} \|\cL^r_\theta\|_{L^1,L^1}\cdot \|h\|_{L^\infty} \|\rho_{\nu_\theta}-\rho_{\nu_0}\|_{L^1}+\|\cL^j_\theta-\cL^j_0\|_{\BV,L^1}\|h\rho_{\nu_0}\|_{\BV} \\
    &\lesssim n^{-1/2}+\|\cL^j_\theta-\cL^j_0\|_{\BV,L^1},
\end{align*}
and
\begin{align*}
    \|\cL^j_\theta-\cL^j_0\|_{\BV,L^1}
    = \sum_{r=0}^{j-1} \|\cL^r_\theta(\cL_\theta-\cL_0)\cL^{j-1-r}_0\|_{\BV,L^1}
    &\lesssim \sup_{\theta,r} \|\cL^r_\theta\|_{L^1,L^1} \|\cL_\theta-\cL_0\|_{\BV, L^1}\sum_{r=0}^{j-1}  \|\cL^{r}_0\|_{\BV,\BV}\\
     &\lesssim  j n^{-1/2},
\end{align*}
Therefore, we have 
$$|\EXP_{\nu_\theta}([S_{k,\theta}(h)]^2)-\EXP_{\nu_0}([S_{k,0}(h)]^2)| \lesssim k^3n^{-1/2}. $$
Finally,
\begin{align*}
   \sigma^2_\theta- \sigma^2_0\lesssim k^2n^{-1/2}+kn^{-1/2}+ o(k^{-1}),
\end{align*}
and choosing $k= \cO(n^{1/6})$, we can obtain the best possible rate of convergence of the standard deviations, 
$$|\sigma^*-\sigma|\lesssim |\sigma^2_\theta - \sigma^2_0| \lesssim n^{-1/6}.$$
So, in the case of piecewise expanding maps, we have the following improved asymptotic accuracy of the non-pivoted bootstrap.
$$
    \sup_{z \in \reals }\Big|\Prob_{\mu}(T_n \leq z) - \Prob(T^*_n \leq z \, |\,x_{0},\dots, x_{n-1})\Big| = \cO_\text{a.s.}(n^{-1/6}). 
$$
\end{exmp}

\begin{exmp}[Smooth expanding maps, a continuation of \Cref{sec:CubicPiecewise}] \label{eg:GaussianError}
As in the previous example, we assume that the mesh-size, $\theta$, of the spline approximation, $\mathfrak{s}$, obtained by $x_0,\dots,x_{n-1}$ is $\cO(n^{-1})$. Since $$\|g-\mathfrak{s}\|_{L^\infty}+\|g'-\mathfrak{s}'\|_{L^\infty} =\cO(\theta^2)+\cO(\theta) = \cO(n^{-1}),$$
from \Cref{sec:SmoothExp}, we have that 
$$\|\cL_\theta - \cL_0\|_{\BV,L^1} \lesssim n^{-1}.$$
Arguing as in the previous example, for all $k>0$,
$$|A^*-A| = |A_\theta - A_0| \lesssim n^{-1}+\cO(k^{-1}).$$
It is easy to see that choosing $k=\cO(n)$ gives the best possible rate of convergence $\cO(n^{-1})$. So, arguing as in the previous example,.
$$|\sigma^*-\sigma|\lesssim |\sigma^2_\theta - \sigma^2_0| \lesssim k^2n^{-1}+kn^{-1}+ o(k^{-1}),$$ and choosing $k=\cO(n^{1/3})$, the best possible rate of convergence is $\cO(n^{-1/3})$. 

Therefore, for smooth expanding maps of $\bbT$, we have we have the following improved asymptotic accuracy of the non-pivoted bootstrap.
$$
    \sup_{z \in \reals }\Big|\Prob_{\mu}(T_n \leq z) - \Prob(T^*_n \leq z \, |\,x_{0},\dots, x_{n-1})\Big| = \cO_\text{a.s.}(n^{-1/3}). 
$$
\end{exmp}

\section{Computer Simulation of the Bootstrap}\label{sec:simulationresults}
\subsection{Data generating processes}\label{sec:DGP}
We let the choices of sample size be $n= 25, 50, 100$ and take the doubling map, the logistic map, and the drill map below as the choices for the transformation function $g$.

\subsubsection{Doubling map} The doubling map also called the dyadic transformation is given by $$g(x)=2x\hspace{-6pt}\mod 1,\,\,\,\,\, x \in [0,1].$$
In general, $r$-adic transformations, including dyadic transformations, allow us to study the statistical properties of the digits of real numbers. Because of its simple yet chaotic nature, von Neumann proposed the doubling map as a random number generator \cite{von1963various}, and later, R\'enyi studied its statistical properties \cite{doubling}. In fact, it is an expanding (therefore, hyperbolic) map, exhibits sensitive dependence on initial conditions, has a unique exponentially mixing acip -- the Lebesgue measure $[0,1]$ -- and also, is a Bernoulli map. Further, it is a $C^\infty-$map of $\bbT$ with constant first derivative $=2 >1$, and hence, falls into the class of examples considered in \Cref{sec:SmoothExp}. 

One could also consider the doubling map as the following map of the binary expansions of data
\begin{equation}\label{eq:doublingBinary}
\sum_{j=1}^{\infty}w_{j}2^{-j} \mapsto \sum_{j=1}^{\infty}w_{j+1}2^{-j}.    
\end{equation}
In other words, the doubling map neglects the first binary digit but shift all the other digit to the left by one unit. Unfortunately, in computer software, the initial state $x_{0}$ only has a binary expansion up to a finite order $K$, e.g., $K=32$ when $x_{0}$ is a uniform random variable generated in R software; as a result of this round-off error, by \eqref{eq:doublingBinary}, after $K$ iterations the orbit of the simulated doubling map will always end up at zero \cite{gora1988computers}. To overcome this problem, we let $x_{i} = g(x_{i-1}) + 2^{-20} \eps_{i}$, where $\{\eps_{i}\}$ are independent \textit{Bernoulli}$(1/2)$ random variables conditional on $x_{i}\in [0,1]$. By the shadowing lemma, this perturbed orbit of $\{x_{i}\}$ stays uniformly close to an unperturbed orbit; see \cite[p.~18]{ott_2002}. 

\subsubsection{Drill map}
The motion of a rotary drill induces a piecewise expanding map of the interval, $g: [0,1]\to [0,1]$, defined as follows:
\begin{align*}
    \alpha &= \frac{\Lambda}{(\Lambda-1)}, \,\,\, q_{[\Lambda]-k}= \max\left\{0, \frac{1}{2} \cdot \frac{\Lambda-1-k}{\Lambda-1}\right\}, \ \ k = 1,\dots, [\Lambda],\\ 
    d_{\Lambda}(x) &= \alpha\left( k- \sqrt{k^{2}-\frac{k}{\alpha}(k+1-2x)} \right), \ \ q_{[\Lambda]-k} < x \leq q_{[\Lambda]-k+1}, \ k = 1,2,\dots, [\Lambda], \\
    g^{1/2}(x)&= x+d_{\Lambda}(x) \ \ \text{(mod $1$)},\,\,\, \text{and}\,\,\,\, g(x) = g^{1/2}(g^{1/2}(x)),
\end{align*}
where $[\Lambda]$ denote the integer part of $\Lambda$ and $\Lambda$ is a parameter indicating the influence of gravity on fluid motion. We set $\Lambda=3$ in our simulation and include the corresponding transformation function $g$ in \Cref{fig:drillTrans}. 
\begin{figure}[ht]
    \centering
    \includegraphics[scale = 0.5]{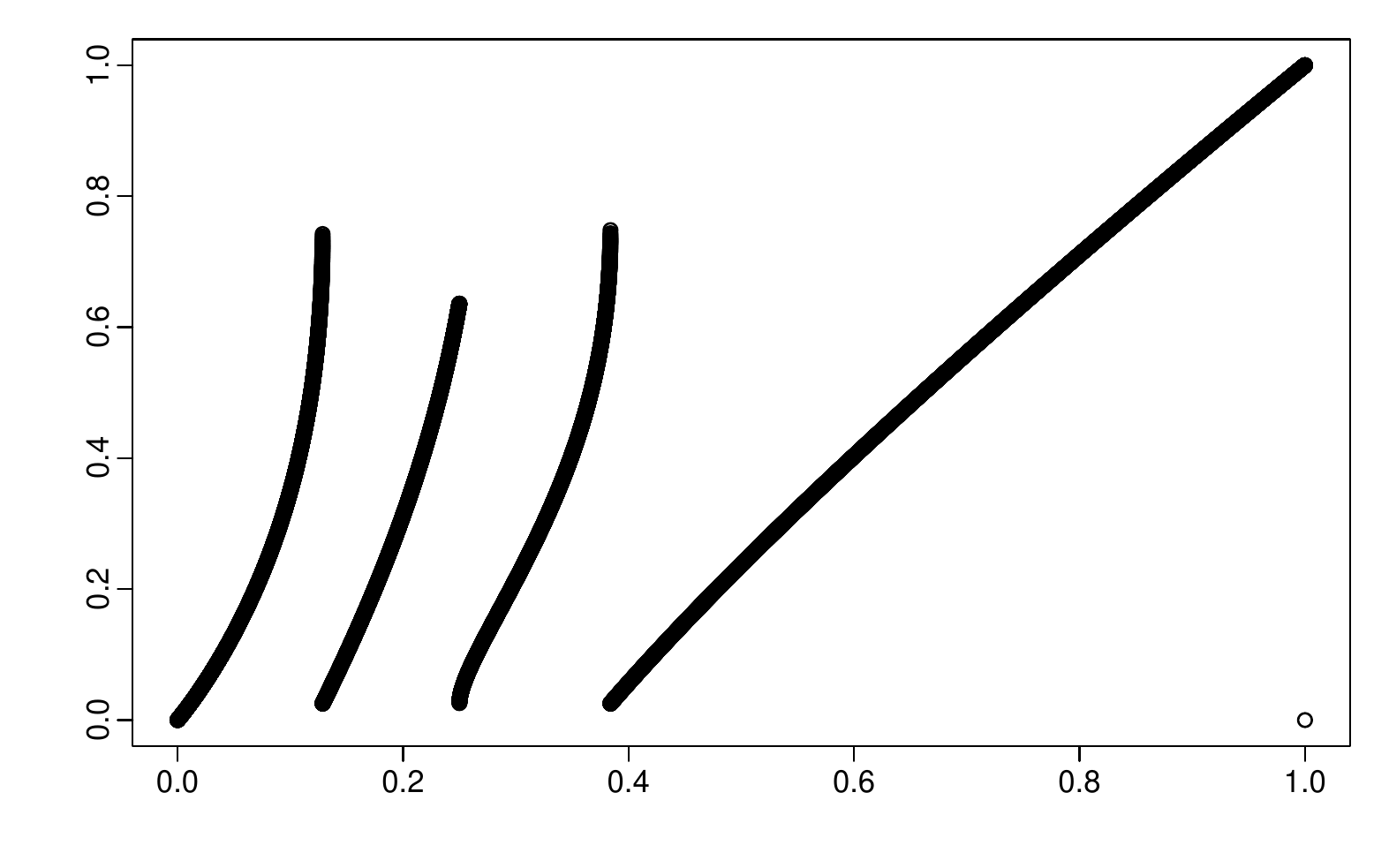}
    \vspace{-15pt}
    \caption{Drill map when $\Lambda=3$.}
    \vspace{-5pt}
    \label{fig:drillTrans}
\end{figure}

This map was first considered in \cite{lasota1974application} to model the movement of an oil drill with real world engineering applications in mind. This is a particular example of transformations discussed in \Cref{sec:PiecewiseExp}; see  also \cite[Section 1.2, Section 13.3]{Gora}.

\subsubsection{Logistic map} The logistic map is useful as a discrete-time population model of various biological species; see \cite{logistic} for an in-depth discussion. The logistic map with parameter $r\in (0,4]$ is defined by
$$g(x) = rx(1-x),\,\,\,\, x\in [0,1].$$
Here, we focus on the case $r=4$. It is well-known that $g$ with $r=4$ exhibits sensitive dependency to initial conditions and has a unique acip (despite not being hyperbolic); see p.~34--35 of \cite{ott_2002}. 

Even though $g$ does not belong to any of the examples in \Cref{sec:Exmp}, its simulations results are includes here because they were comparable to those of the  previous two examples. This is an indication that asymptotic accuracy of the bootstrap may hold even in the case of mostly hyperbolic maps considered in \cite{Young}. In fact, we believe that we can establish continuous Edgeworth expansions for such maps based on ideas in \cite{FernandoPene}. This will be pursued in the future.

Similar to the doubling map, the logistic map suffers from the round-off error of computer software. As a remedy, we first transform the logistic map to the tent map with $\mathcal{T}: [0,1] \to [0,1]$ defined by $\mathcal{T}(x) = 2\arcsin(\sqrt{x})/\pi$ and then add perturbation to the tent map; see p.~33 of \citep{ott_2002}. Specifically, we let
\begin{align*}
    y_{i-1} &= \mathcal{T}(x_{i-1}), &
    y_{i}   &= 
    \begin{cases}
    2y_{i-1}+2^{-20}\eps_{i}, & \text{for} \ x < 1/2, \\
    2(1-y_{i-1})+2^{-20}(1-\eps_{i}), & \text{for} \ x \geq 1/2,
    \end{cases} &
    x_{i} &= \mathcal{T}^{-1}(y_{i})
\end{align*}
where $\{\eps_{i}\}$ are independent \textit{Bernoulli}$(1/2)$ random variables conditional on $y_{i}\in [0,1]$. By the shadowing lemma, as in the doubling map context, the perturbed orbit $\{y_{i}\}$ stays uniformly close to some unperturbed orbit of the tent map, and in turn, $\{x_{i}\}$ shadows an orbit of the logistic map.

\subsection{Quantity of interest}\label{sub:quant_interest}
We aim to construct two-sided, upper-bounded, and lower-bounded 95\% confidence intervals for spatial average $A$. 
When constructing $A$, we let $h(x) = x$, $h(x)=x^{2}$, and $h(x)= x^{4}$; these choices of $h$ are closely related to the mean, variance, and kurtosis, which have wide applications in a variety of fields. 

\subsection{Bootstrap methods}

To construct the confidence intervals mentioned in \cref{sub:quant_interest}, we apply the pivoted bootstrap in Algorithm \ref{alg:pivot} and the non-pivoted bootstrap in Algorithm \ref{alg:nonpivot} with bootstrap iteration $B = 1000$. We include more details of these algorithms below.

\subsubsection{Estimation of the transformation $g$}
Throughout our simulation, we assume that the support and the location of discontinuities of transformation $g$ are known. Given this, we approximate the transformation function $g$ by piecewise cubic spline. When making extrapolations, we apply ``FMM'' \cite{forsythe1977computer} cubic splines in non-pivoted bootstrap and ``natural'' cubic splines in pivoted bootstrap. When the fitted value lies outside the given support, we take it to be the nearest boundary point of the support, and then, move it inside the support by adding or subtracting another small perturbation of $2^{-20}$. In our preliminary simulation, we used piecewise linear splines but found their performance to be inferior to that of the piecewise cubic splines. 

\subsubsection{
Generation of initial bootstrap data $x_{0}^{*}$
} In our simulation, we implement Line \ref{lin:bootInit_pivoted} of Algorithm \ref{alg:pivot} and Line \ref{lin:bootInit_non_pivoted} of Algorithm \ref{alg:nonpivot}, the generation of the initial bootstrap data $x_{0}^{*}$, as follows:
\vs

{\footnotesize
\normalem
\begin{algorithm}[H]
	\caption{Generation of initial bootstrap data}
	\label{alg:bootInitial}
	\DontPrintSemicolon
	\SetAlgoLined
    
	\KwIn{Data $x_{0}, \dots, x_{n-1}$}
    \BlankLine
    \KwOut{Initial bootstrap state $x_{0}^{*}$} 
	\BlankLine
	
	\While{$x_{0}^{*}\notin \text{support of} \ g$}{
	\BlankLine
	$ucv \leftarrow$ unbiased cross-validated bandwidth \cite{scott1987biased} \;
	\BlankLine
	$\fb \leftarrow ucv/4$ \;
	\BlankLine
	Sample $\eps_{0} \sim \text{Normal}(0, \fb^{2})$ \;
    \BlankLine
    Sample $x_{0}'$ from $\{x_{0},\dots,x_{n-1}\}$, independently of $\eps_{0}$\;
    \BlankLine
    $x_{0}^{*} \leftarrow x_{0}'+\eps_{0}$ \;
    \BlankLine
	}
	\BlankLine
	\Return{$x_{0}^{*}$}
\end{algorithm}
\ULforem 
}
\vs

\begin{rem}
Generating $x_{0}^{*}$ with Algorithm \ref{alg:bootInitial} is equivalent to generating $x_{0}^{*}$ from a kernel-estimated density function, where the kernel is Gaussian and the bandwidth is given by $\fb$; see \cite[p.~471]{SY}. On the other hand, Algorithm \ref{alg:bootInitial} is computationally more economical than directly generating data from the kernel-estimated density function; see the footnote on p.~1068 of \cite{HJ}.
\end{rem}

\begin{rem}
When the bandwidth is large, e.g., $\fb= ucv$, where $ucv$ is the \textit{unbiased cross-validated bandwidth} as in \cite{scott1987biased}, the bootstrap initial state could take a value that the true initial state rarely takes, e.g., 0.99. When starting with such a extreme bootstrap initial state, the bootstrap trajectory of $\{x_{i}^{*}\}$ could also be extreme, and consequently the bootstrap distribution may contain too many outliers. As a remedy, we let the bandwidth be rather small, i.e., $\fb= ucv/4$. 
\end{rem}

\subsubsection{Estimation of long-run variance $\sigma^{2}$}\label{sub:long_run_var}

Since we are not aware of any theoretically-justified estimator for $\sigma$ in the dynamical system setting, we first obtain the true $\sigma$ with Monte-Carlo simulation and then replace $\wh{\sigma}$ by this simulated true $\sigma$. Similarly, we replace $\wh{\sigma}^{*,b}$ by $\tilde{\sigma}^{*}$, where
$$\tilde{\sigma}^{*} \coloneqq \sqrt{\frac{1}{B}  \sum_{b=1}^{B}\Big(\frac{S_{n}^{*,b}(h)-\bar{S}_{n}^{*}(h)}{\sqrt{n}}\Big)^{2}}. $$
Indeed, in a similar way to Remark \ref{rem:truebootexp}, when both $B$ and $n$ are large, recalling $\sigma^{*}$ defined in \eqref{eq:bootQuantities},
\[
\tilde{\sigma}^{*}  \approx \sqrt{\EXP
^{*} \bigg(\Big(\frac{S_{n}^{*}(h)-\bar{S}_{n}^{*}(h)}{\sqrt{n}}\Big)^{2}\bigg)} \approx \sqrt{\EXP^{*} \bigg(\Big(\frac{S_{n}^{*}(h)-n A^*}{\sqrt{n}}\Big)^{2}\bigg)} \approx \sigma^{*}.
\]

\subsection{Non-bootstrap methods}

\subsubsection{Gaussian approximation}

To construct the confidence intervals in \cref{sub:quant_interest}, alternatively we can apply the Gaussian approximation, namely
\[
\frac{S_{n}(h)-nA}{\sqrt{n}\wh{\sigma}} \Rightarrow N(0,1).
\]
As in \Cref{sub:long_run_var}, since we do not know  any theoretically-justified estimator for $\sigma$ in the dynamical system setting, we replace $\wh{\sigma}$ by the simulated true $\sigma$ throughout our simulation.

\subsubsection{$t$-approximation}
To construct the confidence intervals, we can also apply a $t$-approximation, which relies on the presumption that
\[
\frac{S_{n}(h)-nA}{\sqrt{n}s} \Rightarrow t_{n-1},
\]
where $s$ is sample standard deviation of $\{x_{i}, i= 0,\dots, n-1\}$ and $t_{n-1}$ is a $t$-distribution with degrees of freedom $n-1$.

\subsection{Results}
To obtain the empirical coverage of the confidence intervals, we run 700 iterations. The results obtained are included in
\Cref{tab:two-sided_doubling,tab:two-sided_logistic,tab:two-sided_drill,tab:upper-bounded_doubling,tab:upper-bounded_logistic,tab:upper-bounded_drill,tab:lower-bounded_doubling,tab:lower-bounded_logistic,tab:lower-bounded_drill}.
When reporting the data, we use the following abbreviations
\begin{itemize}[leftmargin=25pt]
    \item[--] t ($t$-approximation), 
    \item[--] npboot (non-pivoted bootstrap), 
    \item[--] Gaussian (Gaussian approxiamtion), and
    \item[--] pboot (pivoted bootstrap)
\end{itemize}
to differentiate how the $95\%$ confidence intervals were generated. 

First, recall that the $t$-approximation and non-pivoted bootstrap do not require any prior knowledge of the long-run variance $\sigma^{2}$, while the Gaussian approximation and pivoted-bootstrap heavily depend on $\sigma^{2}$. From \Cref{tab:two-sided_doubling,tab:two-sided_logistic,tab:two-sided_drill,tab:upper-bounded_doubling,tab:upper-bounded_logistic,tab:upper-bounded_drill,tab:lower-bounded_doubling,tab:lower-bounded_logistic,tab:lower-bounded_drill}, we see that,  in general, the Gaussian approximation and pivoted-bootstrap prevail over the $t$-approximation and non-pivoted bootstrap, and hence, as expected, the knowledge of $\sigma^{2}$ improves the accuracy.

Second, the non-pivoted bootstrap significantly outperforms the $t$-approximation. In particular, the non-pivoted bootstrap does not suffer from under-coverage in case of  the doubling map and over-coverage in case of the logistic map whereas the $t$-approximation suffers from both. 
Indeed, the non-pivoted bootstrap has a performance that almost matches the Gaussian approximation and the pivoted bootstrap. Consequently, when $\sigma$ is unknown and cannot be easily estimated, the non-pivoted bootstrap may be preferable.

Third, the Gaussian approximation and pivoted-bootstrap give comparable results; none of them uniformly dominates the other. Since Gaussian approximation fails to capture the asymmetry of the finite-sample distribution, it performs slightly inferior to the pivoted-bootstrap on one-sided confidence intervals. However, since the under-coverage on one tail of the distribution is likely compensated by the over-coverage on the other tail, 
the Gaussian approximation has a slight advantage when two-sided confidence intervals are computed. In summary, when $\sigma$ is given, both the Gaussian approximation and pivoted-bootstrap can be considered.

\hypertarget{two-sided-confidence-interval}{%
\subsubsection{Two-sided confidence
interval}\label{two-sided-confidence-interval}}

\begin{center}
\begin{table}[H]
\begin{tabular}{cccccccccc}
\hline
         & \multicolumn{3}{c}{$h(x)=x$} & \multicolumn{3}{c}{$h(x)=x^{2}$} & \multicolumn{3}{c}{$h(x)=x^{4}$} \\ \cline{2-10} 
         & $n=25$   & $n=50$  & $n=100$  & $n=25$   & $n=50$  & $n=100$  & $n=25$   & $n=50$  & $n=100$  \\ \hline
t        & 0.753    & 0.740   & 0.749    & 0.729    & 0.709   & 0.760    & 0.717    & 0.729   & 0.747    \\ 
npboot   & 0.946         & 0.960        & 0.954         & 0.944         & 0.946        & 0.954         & 0.941         & 0.966        & 0.947         \\ 
Gaussian & 0.957         & 0.957       & 0.947         & 0.956         & 0.957        & 0.944         & 0.960         & 0.957        & 0.953         \\ 
pboot    & 0.940         &  0.940       & 0.941         & 0.964         & 0.959        & 0.949         & 0.940         & 0.960        & 0.936         \\ \hline
\end{tabular}
\caption{Empirical coverage of \textit{two-sided} $95\%$ confidence intervals generated 
when data is generated by the doubling map.}
\label{tab:two-sided_doubling}
\end{table}
\vspace{-20pt}
\end{center}

\begin{center}
\begin{table}[H]
\begin{tabular}{cccccccccc}
\hline
         & \multicolumn{3}{c}{$h(x)=x$} & \multicolumn{3}{c}{$h(x)=x^{2}$} & \multicolumn{3}{c}{$h(x)=x^{4}$} \\ \cline{2-10} 
         & $n=25$   & $n=50$  & $n=100$  & $n=25$   & $n=50$  & $n=100$  & $n=25$   & $n=50$  & $n=100$  \\ \hline
t        & 0.943    & 0.934   & 0.923    & 0.913    & 0.926   & 0.924    & 0.856    & 0.907   & 0.910    \\ 
npboot   & 0.937         & 0.951        & 0.947         & 0.940         & 0.941        & 0.946         & 0.893         & 0.924        & 0.941         \\ 
Gaussian & 0.956         & 0.954       & 0.953         & 0.947         & 0.959        & 0.960         & 0.951         & 0.963        & 0.953         \\ 
pboot    & 0.904         &  0.924       & 0.933         & 0.945         & 0.950        & 0.963         & 0.960         & 0.944        & 0.964         \\ \hline
\end{tabular}
\caption{Empirical coverage of \textit{two-sided} $95\%$ confidence intervals 
when data is generated by the drill map.}
\label{tab:two-sided_drill}
\end{table}
\vspace{-20pt}
\end{center}

\begin{center}
\begin{table}[H]
\begin{tabular}{cccccccccc}
\hline
         & \multicolumn{3}{c}{$h(x)=x$} & \multicolumn{3}{c}{$h(x)=x^{2}$} & \multicolumn{3}{c}{$h(x)=x^{4}$} \\ \cline{2-10} 
         & $n=25$   & $n=50$  & $n=100$  & $n=25$   & $n=50$  & $n=100$  & $n=25$   & $n=50$  & $n=100$  \\ \hline
t        & 0.946    & 0.940   & 0.957    & 0.990    & 0.996   & 0.994    & 1.000    & 1.000   & 1.000    \\ 
npboot   & 0.923         & 0.930        & 0.949         & 0.934         & 0.950        & 0.949         & 0.939         & 0.953        & 0.929         \\ 
Gaussian & 0.957         & 0.941       & 0.953         & 0.959         & 0.957        & 0.947         & 0.957         & 0.950        & 0.949         \\ 
pboot    & 0.960         &  0.939       & 0.957         & 0.949         & 0.940        & 0.943         & 0.963         & 0.946        & 0.964         \\ \hline
\end{tabular}
\caption{Empirical coverage of \textit{two-sided} $95\%$ confidence intervals 
when data is generated by the logistic map.}
\label{tab:two-sided_logistic}
\end{table}
\vspace{-20pt}
\end{center}

\hypertarget{upper-bounded-confidence-interval}{%
\subsubsection{Upper-bounded confidence
interval}\label{upper-bounded-confidence-interval}}

\begin{center}
\begin{table}[H]
\begin{tabular}{cccccccccc}
\hline
         & \multicolumn{3}{c}{$h(x)=x$} & \multicolumn{3}{c}{$h(x)=x^{2}$} & \multicolumn{3}{c}{$h(x)=x^{4}$} \\ \cline{2-10} 
         & $n=25$   & $n=50$  & $n=100$  & $n=25$   & $n=50$  & $n=100$  & $n=25$   & $n=50$  & $n=100$  \\ \hline
t        & 0.850    & 0.814   & 0.834    & 0.784    & 0.781   & 0.829    & 0.741    & 0.791   & 0.793    \\ 
npboot   & 0.936         & 0.947        & 0.957         & 0.944         & 0.947        & 0.961         & 0.930         & 0.966        & 0.944         \\ 
Gaussian & 0.949         & 0.963       & 0.953         & 0.970         & 0.969        & 0.949         & 0.980         & 0.976        & 0.966         \\ 
pboot    & 0.944         &  0.941       & 0.937         & 0.959         & 0.956        & 0.944         & 0.946         & 0.960        & 0.941         \\ \hline
\end{tabular}
\caption{Empirical coverage of \textit{upper-bounded} $95\%$ confidence intervals  
when data is generated by the doubling map.}
\label{tab:upper-bounded_doubling}
\end{table}
\vspace{-20pt}
\end{center}

\begin{center}
\begin{table}[H]
\begin{tabular}{cccccccccc}
\hline
         & \multicolumn{3}{c}{$h(x)=x$} & \multicolumn{3}{c}{$h(x)=x^{2}$} & \multicolumn{3}{c}{$h(x)=x^{4}$} \\ \cline{2-10} 
         & $n=25$   & $n=50$  & $n=100$  & $n=25$   & $n=50$  & $n=100$  & $n=25$   & $n=50$  & $n=100$  \\ \hline
t        & 0.939    & 0.930   & 0.924    & 0.903    & 0.903   & 0.910    & 0.820    & 0.871   & 0.871    \\ 
npboot   & 0.923         & 0.944        & 0.947         & 0.930         & 0.939        & 0.956         & 0.884         & 0.930        & 0.934         \\ 
Gaussian & 0.963         & 0.961       & 0.963         & 0.956         & 0.963        & 0.961         & 0.987         & 0.974        & 0.974         \\ 
pboot    & 0.931         &  0.937       & 0.930         & 0.958         & 0.936        & 0.940         & 0.956         & 0.954        & 0.966         \\ \hline
\end{tabular}
\caption{Empirical coverage of \textit{upper-bounded} $95\%$ confidence intervals 
when data is generated by the drill map.}
\label{tab:upper-bounded_drill}
\end{table}
\vspace{-20pt}
\end{center}

\begin{center}
\begin{table}[H]
\begin{tabular}{cccccccccc}
\hline
         & \multicolumn{3}{c}{$h(x)=x$} & \multicolumn{3}{c}{$h(x)=x^{2}$} & \multicolumn{3}{c}{$h(x)=x^{4}$} \\ \cline{2-10} 
         & $n=25$   & $n=50$  & $n=100$  & $n=25$   & $n=50$  & $n=100$  & $n=25$   & $n=50$  & $n=100$  \\ \hline
t        & 0.967    & 0.953   & 0.954    & 0.989    & 0.984   & 0.989    & 0.997    & 0.999   & 0.997    \\ 
npboot   & 0.939         & 0.926        & 0.966         & 0.941         & 0.963        & 0.950         & 0.953         & 0.954        & 0.941         \\ 
Gaussian & 0.946         & 0.933       & 0.944         & 0.946         & 0.934        & 0.950         & 0.929         & 0.924        & 0.936         \\ 
pboot    & 0.946         &  0.944       & 0.943         & 0.937         & 0.946        & 0.950         & 0.936         & 0.937        & 0.960         \\ \hline
\end{tabular}
\caption{Empirical coverage of \textit{upper-bounded} $95\%$ confidence intervals 
when data is generated by the logistic map.}
\label{tab:upper-bounded_logistic}
\end{table}
\vspace{-20pt}
\end{center}

\hypertarget{lower-bounded-confidence-interval}{%
\subsubsection{Lower-bounded confidence
interval}\label{lower-bounded-confidence-interval}}

\begin{center}
\begin{table}[H]
\begin{tabular}{cccccccccc}
\hline
         & \multicolumn{3}{c}{$h(x)=x$} & \multicolumn{3}{c}{$h(x)=x^{2}$} & \multicolumn{3}{c}{$h(x)=x^{4}$} \\ \cline{2-10} 
         & $n=25$   & $n=50$  & $n=100$  & $n=25$   & $n=50$  & $n=100$  & $n=25$   & $n=50$  & $n=100$  \\ \hline
t        & 0.813    & 0.826   & 0.830    & 0.861    & 0.843   & 0.840    & 0.899    & 0.857   & 0.874    \\ 
npboot   & 0.947         & 0.954        & 0.959         & 0.934         & 0.947        & 0.947         & 0.950         & 0.961        & 0.963         \\ 
Gaussian & 0.963         & 0.950       & 0.954         & 0.939         & 0.949        & 0.943         & 0.936         & 0.933        & 0.944         \\ 
pboot    & 0.951         &  0.941       & 0.954         & 0.957         & 0.959        & 0.957         & 0.951         & 0.954        & 0.946         \\ \hline
\end{tabular}
\caption{Empirical coverage of \textit{lower-bounded} $95\%$ confidence intervals 
when data is generated by the doubling map.}
\label{tab:lower-bounded_doubling}
\end{table}
\vspace{-20pt}
\end{center}

\begin{center}
\begin{table}[H]
\begin{tabular}{cccccccccc}
\hline
         & \multicolumn{3}{c}{$h(x)=x$} & \multicolumn{3}{c}{$h(x)=x^{2}$} & \multicolumn{3}{c}{$h(x)=x^{4}$} \\ \cline{2-10} 
         & $n=25$   & $n=50$  & $n=100$  & $n=25$   & $n=50$  & $n=100$  & $n=25$   & $n=50$  & $n=100$  \\ \hline
t        & 0.959    & 0.946   & 0.937    & 0.966    & 0.964   & 0.941    & 0.986    & 0.977   & 0.971    \\ 
npboot   & 0.984         & 0.960        & 0.947         & 0.960         & 0.944        & 0.946         & 0.957         & 0.944        & 0.947         \\ 
Gaussian & 0.951         & 0.957       & 0.939         & 0.936         & 0.947        & 0.946         & 0.931         & 0.943        & 0.924         \\ 
pboot    & 0.914         &  0.940       & 0.951         & 0.921         & 0.940        & 0.950         & 0.964         & 0.930        & 0.957         \\ \hline
\end{tabular}
\caption{Empirical coverage of \textit{lower-bounded} $95\%$ confidence intervals 
when the data is generated by the drill map.}
\label{tab:lower-bounded_drill}
\end{table}
\vspace{-20pt}
\end{center}

\begin{center}
\begin{table}[H]
\begin{tabular}{cccccccccc}
\hline
         & \multicolumn{3}{c}{$h(x)=x$} & \multicolumn{3}{c}{$h(x)=x^{2}$} & \multicolumn{3}{c}{$h(x)=x^{4}$} \\ \cline{2-10} 
         & $n=25$   & $n=50$  & $n=100$  & $n=25$   & $n=50$  & $n=100$  & $n=25$   & $n=50$  & $n=100$  \\ \hline
t        & 0.941    & 0.939   & 0.949    & 0.987    & 0.996   & 0.990    & 1.000    & 1.000   & 1.000    \\ 
npboot   & 0.934         & 0.951        & 0.939         & 0.954         & 0.950        & 0.953         & 0.933         & 0.950        & 0.941         \\ 
Gaussian & 0.961         & 0.949       & 0.956         & 0.971         & 0.970        & 0.960         & 0.973         & 0.973        & 0.969         \\ 
pboot    & 0.946         &  0.944       & 0.956         & 0.959         & 0.950        & 0.947         & 0.971         & 0.950        & 0.956         \\ \hline
\end{tabular}
\caption{Empirical coverage of \textit{lower-bounded} $95\%$ confidence intervals 
when data is generated by the logistic map.}
\label{tab:lower-bounded_logistic}
\end{table}
\vspace{-20pt}
\end{center}

\printbibliography

\end{document}